\documentclass[12pt]{amsart}
\usepackage{extarrows}
\usepackage{amssymb}
\usepackage{amsmath,latexsym}
\usepackage{amsfonts}
\usepackage{amsthm}
\usepackage{eucal}
\usepackage{color}

\usepackage[foot]{amsaddr}

\allowdisplaybreaks[4]

\newcommand{\vol}{{\rm vol}}
\newcommand{\Ric}{{\rm Ric}}

\numberwithin{equation}{section}

\newtheorem{thm}{Theorem}[section]
\newtheorem{defn}[thm]{Definition}
\newtheorem{lem}[thm]{Lemma}
\newtheorem{prop}[thm]{Proposition}
\newtheorem{cor}[thm]{Corollary}

\newtheorem{remark}[thm]{Remark}
\newtheorem{ex}[thm]{Example}

\begin{document}

\title[]{Perelman's functionals on manifolds with non-isolated conical singularities}
\author{Xianzhe Dai}
\address{
Department of Mathematics, 
University of California, Santa Barbara
CA93106,
USA}
\email{dai@math.ucsb.edu}

\author{Changliang Wang}
\address{School of Mathematical Sciences and Institute for Advanced Study, Tongji University, Shanghai 200092, China}
\email{wangchl@tongji.edu.cn}
\date{}


\footnote{X. D. was partially supported by the Simons Foundation. C.W. was partially supported by the Fundamental Research Funds for the Central Universities.}

\begin{abstract}
In this article, we define Perelman's functionals on manifolds with non-isolated conical singularities by starting from a spectral point of view for the Perelman's $\lambda$-functional. (Our definition of non-isolated conical singularities includes isolated conical singularities.) We prove that the spectrum of Schr\"odinger operator $-4\Delta + R$ on manifolds with non-isolated conical singularities consists of discrete eigenvalues with finite multiplicities, provided that scalar curvatures of cross sections of cones have a certain lower bound. This enables us to define the $\lambda$-functional on these singular manifolds, and further, to prove that the infimum of $W$-functional is finite, with the help of some weighted Sobolev inequalities. Furthermore, we obtain some asymptotic behavior of eigenfunctions and the minimizer of the $W$-functional near the singularity, and a more refined optimal partial asymptotic expansion for eigenfunctions near isolated conical singularities. We also study the spectrum of $-4\Delta + R$ and Perelman's functionals on manifolds with more general singularities, i.e. the $r^{\alpha}$-horn singularities which serve as prototypes of algebraic singularities. 
\end{abstract}

\maketitle

\section{Introduction}
\noindent Perelman's functionals are some of the crucial ingredients introduced by Perelman \cite{Per02} in his famous work on 
Ricci flow. These functionals play critical roles in his celebrated proof of Thurston's geometrization conjecture and Poincar\'e conjecture, as well as in later studies of the Ricci flow.

We briefly recall Perelman's functionals on compact manifolds. Let $(M,g)$ be a compact Riemannian manifold without boundary. The $\mathcal{F}$-functional is given as
\begin{equation}\label{eqn: F-functional2}
\mathcal{F}(g,u)=\int_{M}(4|\nabla u|^{2}+R_{g}u^{2})d\vol_{g},
\end{equation}
where $R_{g}$ is the scalar curvature of $g$ and $u>0$ is a smooth function.
Perelman's $\lambda$-functional is then defined as
\begin{equation}\label{eqn: lambda-functional}
\lambda(g)=\inf\left\{\mathcal{F}(g,u) \mid \int_{M}u^{2}d\vol_{g}=1\right\}.
\end{equation}

Clearly, from $(\ref{eqn: lambda-functional})$ and $(\ref{eqn: F-functional2})$,  $\lambda(g)$ is the smallest eigenvalue of the Schr\"odinger operator $-4\Delta_{g}+R_{g}$. Starting from this spectral viewpoint, we developed Perelman's $\lambda$-functional on compact manifolds with isolated conical singularities in \cite{DW18}.

To take the scale into account, Perelman also introduced the $W$-functional and $\mu$-functional on smooth compact manifolds in \cite{Per02}. They play important roles in the study of singularities of the Ricci flow. The $W$-functional is given by
\begin{equation*}
W(g, u, \tau)=\frac{1}{(4\pi\tau)^{\frac{n}{2}}}\int_{M}[\tau(R_{g}u^{2}+4|\nabla u|^{2})-2u^{2}\ln u-nu^{2}]d\vol_{g},
\end{equation*}
where $\tau>0$ is a scale parameter, and $u>0$ is a smooth function.
The $\mu$-functional is defined by, for each $\tau>0$,
\begin{equation}\label{eqn: mu-functional}
\mu(g, \tau)=\inf\left\{W(g, u, \tau) \mid u\in C^{\infty}(M),\ \  u>0, \ \  \frac{1}{(4\pi\tau)^{\frac{n}{2}}}\int_{M}u^{2}d\vol_{g}=1\right\}.
\end{equation}
 It is well-known that 
 the finiteness of the infimum in $(\ref{eqn: mu-functional})$ follows from the logarithmic Sobolev inequality on smooth compact Riemannian manifolds, while the regularity of the minimizer follows from the elliptic estimates and Sobolev embedding. 
 
 In view of the applications to Ricci flow, it is a natural question to develop these functionals for singular spaces. We refer to the excellent survey \cite{Bam21} for some recent developments in Ricci flow and the role played by singular spaces. The possibility of conical singularity appearing in type I singularities of Ricci flow is indicated by the recent work \cite{MW17}. The work \cite{GS18} studies the Ricci flow coming out of manifolds with isolated conical singularities.
 Conical singularity can also be viewed as a generalization of orbifold singularity and we refer to \cite{KL14} for Ricci flow on orbifolds and its applications. 
 
 In \cite{DW20}, based on the work in \cite{DW18}, we studied the $W$-functional and $\mu$-functional on manifolds with isolated conical singularities. In \cite{Ozu20}, Ozuch developed Perelman's functionals on manifolds with isolated conical singularities from a different point of view. In \cite{KV21}, Kr\"oncke and Vertman obtained a partial asymptotic expansion for functions $u$, satisfying $(\Delta_g + c R_g)u=F(u)$ with $F(u) \leq C |u\ln u|$, near isolated conical singularities whose cross sections have scalar curvature the same as the round sphere, and further studied singular Ricci solitons in this case, see Remark 1.7 for more details.

Non-isolated conical singularity plays a crucial role in the resolution of the Yau-Tian-Donaldson conjecture \cite{CDS,T}; see also \cite {S,LZ} for some related work on Ricci flow. In this paper, we develop Perelman's functionals on compact manifolds with non-isolated conical singularities as defined in Definition \ref{defn: ManifoldWithNonisolatedConicalSingularities} below, and also on manifolds with $r^{\alpha}$-horn singularities as defined in Definition \ref{defn: ManifoldWith-alpha-ConicalSingularities} below. Our definition of non-isolated singularities includes isolated singularities. In particular, we establish the existence of the minimizers and obtain asymptotic behavior of the minimizers near the singularity. These asymptotic behaviors are not optimal for non-isolated conical singularity. On the other hand, we obtain an optimal partial asymptotic expansion for eigenfunctions of $-\Delta + c R$ near a general isolated conical singularity.

\subsection{Spectral properties of the operator $-4\Delta + R$ on manifolds with non-isolated conical singularities}
Roughly speaking, a compact manifold with non-isolated conical singularities $(M^{n}, g)$ is a compact manifold $M$ with boundary $\partial M$ and a Riemannian metric $g$ degenerating in a specific way at the boundary. The boundary $\partial M$ is the total space of the fibration $\pi: \partial M \rightarrow B$ with the $f$-dimensional compact manifold $F^{f}$ as a typical fiber, and in a collar neighborhood $(0, 1)\times \partial M$ of $\partial M$ the metric $g$ is asymptotic to $g_{0}=dr^{2}+r^{2}\hat{g}+\pi^{*}\check{g}$, where $r$ is a coordinate on $(0, 1)$, $\check{g}$ is a metric on the base $B$ and $\hat{g}$ is the restricted metric on fibers $F$. These singularities are also called wedge singularities (e.g. \cite{Maz91}) or edge singularities (e.g. \cite{Ver21}). Vertman \cite{Ver21} established the existence of Ricci de Turk flow within the class of these singular manifolds whose curvature satisfy some asymptotic control near singularities. Note that when $B$ is a single point, we are reduced to manifolds with isolated conical singularity.

Our first main result is the following spectral property of the operator $-4\Delta+R$ on these singular manifolds. This enables us to define Perelman's $\lambda$-functional on these singular manifolds as the first eigenvalue of $-4\Delta + R$.

\begin{thm}[Theorem \ref{thm: spectrum property} below]\label{MainResult1}
Let $(M^{n}, g)$, $n\geq 4$, be a compact Riemannian manifold with non-isolated conical singularities as defined in Definition $ \ref{defn: ManifoldWithNonisolatedConicalSingularities} $ below. If $\min\limits_{\partial M}\{R_{\hat{g}}\}>(f-1)$, then the operator $-4\Delta_{g}+R_{g}$ with domain $C^{\infty}_{0}(\mathring{M})$ is semi-bounded, and the spectrum of its Friedrichs extension consists of discrete eigenvalues with finite multiplicities
$\lambda_{1}<\lambda_{2}\leq\lambda_{3}\leq\cdots,$ and $\lambda_{k}\rightarrow+\infty$, as $k\rightarrow+\infty$. The corresponding eigenfunctions form a complete basis of $L^{2}(M)$
\end{thm}

\begin{remark}
{\rm
Similarly Theorem \ref{MainResult1} extends to the Schr\"odinger operator $-\Delta_{g}+ c R_{g}$ for any constant $c$, where the scalar curvature condition becomes $\min\limits_{\partial M}\{R_{\hat{g}}\}> \frac{f[4c(f-1)-(f-2)]-1}{4c}$. For example, for the well known conformal Laplacian $-\Delta_{g}+ \frac{n-2}{4(n-1)} R_{g}$, the scalar curvature condition becomes $\min\limits_{\partial M}\{R_{\hat{g}}\}>\frac{(f-1)(n-f-1)}{n-2}$.
}
\end{remark}

\begin{remark}
{\rm The role of the Schr\"odinger operator $-\Delta_{g}+ \frac{1}{2} R_{g}$ in the study of scalar curvature has attracted a lot attention recently \cite{S17}.
In \cite{G21}, Gromov further highlighted the importance of the operator $-\Delta_{g}+ \frac{1}{2} R_{g}$ in the study of scalar curvature. In particular, he asks for what class of isolated conical singularity this operator will be positive. Our result gives an answer to this question for general conical singularity. Namely, for non-isolated conical singularity, a sufficient condition is
$$ \min\limits_{\partial M}\{R_{\hat{g}}\}> \frac{1}{2}(f^2 -1).
$$
For isolated conical singularity, this becomes
$$ \min\limits_{\partial M}\{R_{\hat{g}}\}> \frac{1}{2}(n -2)n.
$$}
\end{remark}

Moreover, we obtain the following asymptotic behavior for eigenfuctions of $-4\Delta_{g} + R_{g}$ near the singular set $\partial M$, by a Nash-Moser iteration argument.
\begin{thm}[Corollary \ref{cor: asymptotic of eigenfunctions} below]\label{thm: asymptotic of eigenfunctions}
Let $u$ be an eigenfunction of $-4\Delta_{g}+R_{g}$ on a $n$-dimensional manifold $(M^n, g)$ with non-isolated conical singularities.
Then
\begin{equation} \label{eq-eigenfunction-asymptotic}
|\nabla^{i}u|  = o(r^{-\frac{n-2}{2}-i})
\end{equation}
as $r\rightarrow 0$, i.e. approaching the singular set $\partial M$, for $i=0, 1$. Here $ r $ is the radial variable on a conical neighborhood of the singular set $ \partial M $.
\end{thm}

\begin{remark}
{\rm   \eqref{eq-eigenfunction-asymptotic}, as well as \eqref{eq-minimizer-asymptotic} below,  is in general not optimal as one expects that the exponent $-\frac{n-2}{2}-i$ should be replaced by $-\frac{f-2}{2}-i$. Later, we will establish an optimal partial asymptotic expansion for eigenfunctions of $-\Delta + c R$ for isolated conical singularity.}
\end{remark}


\subsection{The $W$-functional on manifolds with non-isolated conical singularities}

For the $W$-functional, we extend the results in \cite{DW20} and obtain:
\begin{thm}\label{thm: W-functional introduction}
Let $(M^{n}, g)$ be a compact manifold with non-isolated conical singularities. If $\min\limits_{\partial M}\{R_{\hat{g}}\}>(f-1)$, then for each fixed $\tau > 0$,
\begin{equation}\label{eqn: finite-infimum}
\mu(g, \tau) = \inf \left\{ W(g, u, \tau) \mid u \in W^{1,2}_{1-\frac{n}{2}}(M), u>0, \left\| \frac{1}{(4 \pi \tau)^{\frac{n}{2}}} u \right\|_{L^{2}(M)} = 1 \right\} > -\infty.
\end{equation}
Here $W^{1,2}_{1-\frac{n}{2}}(M)$ is a weighted Sobolev space defined in $ (\ref{eqn: WSN}) $ below.

Moreover, there exists a smooth function $u>0$ that realizes the infimum in $(\ref{eqn: finite-infimum})$. The minimizer satisfies
\begin{equation} \label{eq-minimizer-asymptotic}
|\nabla^{i} u| = o(r^{-\frac{n-2}{2}-i}),
\end{equation}
as $r \rightarrow 0$, 
for $i=0,1$. Here $r$ is the radial variable of a conical neighborhood of the singular set.
\end{thm}

We would like to mention that Theorem \ref{thm: W-functional introduction} extends to the expander entropy $W_{+}$-functional introduced in \cite{FIN05}. In this paper we will only focus on $W$-functional, since the discussion for $W_{+}$-functional is the same.

The finiteness of infimums of $W$-functional and $W_{+}$-functional, and the existence of the minimizer follow from the direct method in the calculus of variations, with the help of the (weighted and unweighted) Sobolev embeddings and compactness of some weighted Sobolev embeddings on manifolds with non-isolated conical singularities obtained in \S\ref{section: Sobolev-weighted-Sobolev-spaces}. The asymptotic behavior of the minimizing function then follows from a Nash-Moser iteration argument as discussed in \S$\ref{section: W-functional}$.

In contrast to isolated conical singularities, the scaling technique, which is used in \cite{DW18}, does not work any more for deriving weighted Sobolev inequalities and weighted elliptic estimate on manifolds with non-isolated conical singularities. Instead, we adapt the partition argument for an isolated cone used in \cite{DW20} to the non-isolated conical singularities setting, and obtain the (unweighted) Sobolev inequality. A Hardy inequality plays a critical role in this derivation, as well as in the semi-boundedness estimate for the operator $-4\Delta + R$. Then we obtain the weighted Sobolev inequality from the Sobolev inequality. 

One can further analyze the asymptotic behavior of the functional $\mu(g, \tau)$ as $\tau \rightarrow +\infty$ and $\tau \rightarrow 0$ to obtain the finiteness for the $\nu$-entropy on manifolds with non-isolated conical singularities. Indeed, from the Sobolev inequality in Proposition \ref{prop: SobolevInequality}, the logarithmic Sobolev inequality follows. Then, as in the smooth case (see, e.g. Lemma 6.30 in \cite{C07}), one deduces from the logarithmic Sobolev inequality that, when $\min\limits_{\partial M}\{R_{\hat{g}}\}>f-2$,
\begin{equation} \label{eq-mu-asymptotic}
\mu(g, \tau)\rightarrow \left\{ \begin{array}{cc} +\infty, \ \  \mbox{if} \ \lambda(g)>0, \\ -\infty, \ \  \mbox{if} \ \lambda(g)<0. \end{array}  \right.
\end{equation}
On the other hand, one can show the finiteness of the limit $\lim\limits_{\tau\rightarrow 0} \mu(g, \tau)$ on manifolds with non-isolated 
conical singularities, by combining the scaling invariance property: $\mu(cg, c\tau)=\mu(g, \tau)$ for any $c>0$, with the fact that $\mu(g_{\mathbb{R}^n}, 1)=0$, as well as $\mu(g_{C(F)}, 1)< \infty$. Here $g_{\mathbb{R}^n}$ is the Euclidean metric on $\mathbb{R}^n$, and $g_{C(F)}$ is the model cone metric on $(0, +\infty)\times F$. The finiteness of $\mu(g_{C(F)}, 1)$ is due to Ozuch \cite{Ozu20}. Therefore, we obtain that 
\begin{equation}
    \nu(g) = \inf\{\mu(g, \tau) \mid \tau>0\} >-\infty,
\end{equation}
 on manifolds with non-isolated conical singularities $(M, g)$ satisfying $\min_{\partial M}\left\{ R_{\hat{g}} \right\} > f-2$, provided that $\lambda(g)>0$.  

The finiteness of $\nu$-entropy  and Theorems \ref{MainResult1} and \ref{thm: W-functional introduction}, as announced in \cite{DW20}, answer the existence problem of Perelman's functionals on manifolds with non-isolated (edge) conical singularities, which was also raised in the survey \cite{KV21b}.


\subsection{The horn singularity}

We also study the spectrum of Schr\"odinger operator $-\Delta + c R$ and Perelman's functionals on manifolds with (either non-isolated or isolated) more general singularities, i.e. $r^{\alpha}$-horn in \cite{Che80}, whose model is $((0, 1)\times F^{f}, g_{\alpha}= dr^{2} + r^{2\alpha}\hat{g})$ with $\alpha\in \mathbb{N}$ and $\hat{g}$ a Riemannian metric on $F^{f}$. This type of singularity occurs naturally in general relativity (e.g., the negative mass Schwarzschild metric has a horn singularity \cite{ST}) and singular projective variety \cite{HP}.

The study of analysis on manifolds with $r^{\alpha}$-horn singularities with $\alpha>1$ is very different from the case of conical singularities (i.e. the case of $\alpha=1$). For example, neither Hardy inequality nor the partition method used in \cite{DW20} and in the proof of Proposition \ref{prop: SobolevInequality} below works for $r^{\alpha}$-horn singularities with $\alpha>1$. However, we still establish some $\alpha$-weighted Sobolev inequalities in Proposition \ref{prop: alpha-weighted-Sobolev-inequality} below as well as some compact $\alpha$-weighted Sobolev embeddings in Proposition \ref{prop: compact-alpha-weighted embedding} below, by using some change of variables to relate the problem to that on cylinders. Then we obtain:

\begin{thm}\label{thm: horn}
Let $(M^n, g)$ be a compact Riemannian manifold with 
$r^{\alpha}$-horn singularities as in Definition $\ref{defn: ManifoldWith-alpha-ConicalSingularities}$  below. If $\min\limits_{\partial M}\{R_{\hat{g}}\}>0$, then 
\begin{enumerate}
    \item the Friedrichs extension of the Schr\"odinger operator 
          $-\Delta_{g} + c R_{g}\ (c>0)$ with the domain $C^{\infty}_{0}(\mathring{M})$ has the spectrum consisting of discrete eigenvalues with finite multiplicity $-\infty < \lambda_{1}<\lambda_{2}\leq\lambda_{3}\leq \cdots$, and the corresponding eigenfunctions are smooth and form a basis of $L^{2}(M, g)$,
   \item \begin{equation*}
            \inf\left\{W(g, \tau, u) \mid u\in W^{1, 2}_{\alpha-\frac{n\alpha}{2}, \alpha}(M), \ \ \|u\|_{L^{2}(M)}=(4\pi\tau)^{\frac{n}{2}}\right\} > -\infty,
         \end{equation*}
   \item and an eigenfuction of $-\Delta+cR$ or the minimizer of 
         the $W$-functional, $u$, has the asymptotic behavior
         \begin{equation*}
         |\nabla^{i}u| = o(r^{-\frac{n-2}{2}\alpha - i\alpha})
         \end{equation*}
         as $r\rightarrow 0$, i.e. approaching the singular set, for $i=0, 1$. Here $r$ is the radial variable of a horn-like neighborhood of the singular set.
\end{enumerate}
\end{thm}

\begin{remark}
{\rm Note that the assumption for scalar curvature on the cross section in Theorem \ref{thm: horn} is weaker than that in the conical singularity case, Cf. Theorems $\ref{MainResult1}$ and $\ref{thm: W-functional introduction}$, and is independent of the dimension of the manifold and the constant $c$ in the Schr\"odinger operator.
}
\end{remark}


\subsection{A partial asymptotic expansion of eigenfunction on manifolds with isolated conical singularities}

As we mentioned, the asymptotic behaviors of the minimizers, important in application, are not optimal. This can be much improved for isolated conical singularities. Indeed we obtain the following optimal partial asymptotic expansion for the eigenfunctions, by using Melrose's b-calculus theory \cite{Mel93} and a weighted elliptic bootstrapping argument.
\begin{thm}[Theorem \ref{prop: partial asymptotic expansion} below]\label{thm: partial asymptotic expansion}
Let $u$ be an eigenfunciton of $-4\Delta_{g}+R_{g}$ on a $n$-dimensional manifold with an isolated conical singularity. Then 
\begin{equation}
    u=a(x)r^{\mu} + o(r^{\mu^{\prime}})
\end{equation}
as $r\rightarrow0$, where $\mu=-\frac{n-2}{2}+\frac{\sqrt{\nu-(n-2)}}{2}$ and $a(x)$ is an eigenfunction with eigenvalue $\nu$ of $-4\Delta_{\hat{g}}+R_{\hat{g}}$ on the cross section $(F, \hat{g})$ of the cone, and $\mu^{\prime}>\mu$. Here $r$ is the radial variable on the conical neighborhood of the singular point.
\end{thm}

This in particular implies the partial asymptotic expansion for eigenfunctions obtained by Kr\"oncke and Vertman in \cite{KV21}, where they considered the case where $(F, \hat{g})$ has the scalar curvature $R_{\hat{g}}=(n-1)(n-2)$, i.e., that of the round sphere. In this special case, one can easily see that $\mu\geq0$ 
in Theorem \ref{thm: partial asymptotic expansion}. 

On manifolds with non-isolated conical singularities with some additional assumptions, e.g. the metrics on the fibres are isospectral with respect to the operator $-4\Delta + 4 R$ and the difference between asymptotic conical metric and model conical metric is smooth up to the singular set, the work of Mazzeo in \cite{Maz91} gives a full asymptotic expansion for the eigenfunctions. Note that these restrictions are not imposed in our work for manifolds with non-isolated conical singularities.


\subsection{Gradient Ricci solitons with isolated conical singularities}
Gradient Ricci solitons are singularity models of Ricci flow. We show that there are no non-trivial gradient steady or expanding Ricci soliton on compact manifolds with isolated conical singularity, provided the scalar curvature of the cross sections of the conical singularities equals that of the sphere. For this purpose, we first obtain, with the help of Melrose's b-calculus, some asymptotic estimate as in Proposition \ref{prop: asymptotic for potential function of steady soliton} below for the potential functions of compact gradient Ricci solitons with isolated conical singularities. In particular, we show that the potential function of a compact gradient Ricci soliton with isolated conical singularities is bounded and goes to zero when one approaches the singular points, 
when the above condition on the scalar curvature is satisfied. Using this, if one further assumes that the soliton is either steady or expanding, then we show  that it must be Einstein. Here the subtlety is that we do not impose asymptotic condition for the potential functions, and thus, a priori, gradient Ricci solitons with conical singularities may not be critical points of Perelman's functionals, unlike the smooth case.

\begin{thm}\label{thm: soliton}
Let $(M^{n}, g, f)$ be a compact steady or expanding gradient Ricci soliton with isolated conical singularities, i.e.
\begin{equation}
    Ric_g + \nabla^2 f = \Lambda g,
\end{equation}
with constant $\Lambda \leq 0$. If the cross section of the model cone at each singularity has scalar curvature  $R_{\hat{g}}=(n-1)(n-2)$, then $(M^{n}, g)$ must be an Einstein manifold with isolated conical singularities, i.e. 
$
    Ric_g = \Lambda g.
$
\end{thm}

\begin{remark}
{\rm More generally, Kr\"oncke and Vertman \cite{KV21} have proved that compact steady and expanding Ricci solitons with isolated conical singularities and dimension$\geq 4$ have to be Einstein, provided that the scalar curvature of cross section of model cones are the same as that of the round sphere. In their argument, in the 4 dimensional case, some asymptotic control for Ricci curvature near the singularities was assumed. Our results hold with neither dimensional restriction nor any extra condition other than the assumption on the scalar curvature. 
}
\end{remark}



\subsection{Organization of the paper}

The article is organized as follows. In \S\ref{section: manifolds with conical singularities}, we give the precise definition for manifolds with non-isolated conical singularities and recall some basic facts for Riemannian submersions. In \S\ref{section: Sobolev-weighted-Sobolev-spaces}, we derive the Sobolev inequality and some weighted Sobolev inequalities on manifolds with non-isolated conical singularities. We also derive weighted elliptic regularity estimate 
on manifolds with non-isolated conical singularities. In \S\ref{section: spectrum}, we establish the desired spectral property for the Schr\"odinger operator $-4\Delta + R$ on manifolds with non-isolated conical singularities. In \S\ref{section: asymptotic behavior of eigenfuctions} and \ref{section: W-functional}, we obtain asymptotic behavior for eigenfunctions of $-4\Delta+R$ and the minimizer of the $W$-functional by using Nash-Moser iteration. In \S\ref{section: horn}, we study the spectrum of $-\Delta + cR$ and Perelman's functionals on manifolds with $r^{\alpha}$-horn singularities. In \S\ref{section: partial asymptotic expansion isolated singularities}, we derive a partial asymptotic expansion for eigenfunctions of $-4\Delta+R$ on manifolds with isolated conical singularities by using the b-calculus theory of Melrose and a weighted elliptic bootstrapping argument. In \S\ref{section: gradient ricci solitions}, as another application of b-calculus theory and the asymptotic estimates for functions in weighted Sobolev spaces, we prove that steady and expanding gradient Ricci solitons with isolated conical singularities must be Einstein, provided that the scalar curvature of cross section of model cones are the same as that of the standard sphere. 

{\bf Acknowledgements:} The authors would like to thank an anonymous referee whose valuable comments and suggestions have been very helpful in improving the presentation of this work.

\section{Manifolds with non-isolated conical singularities}\label{section: manifolds with conical singularities}

\begin{defn}\label{defn: ManifoldWithNonisolatedConicalSingularities}
We say that $(M^{n},g)$ is a compact Riemannian manifold with non-isolated conical singularities, if
\begin{enumerate}
\item $M^{n}$ is a compact $n$-dimensional manifold with boundary $\partial M$.
\item The boundary $\partial M$ endowed with a Riemannian metric $g_{\partial M}$ is the total space of a Riemannian submersion
      \begin{equation}\label{eqn: Riemannian-submersion1}
      \pi: (\partial M, g_{\partial M}) \longrightarrow (B^{b}, \check{g})
      \end{equation}
      with a $b$-dimensional closed Riemannian manifold $(B^{b}, \check{g})$ as the base and a $f$-dimensional closed manifold $F^{f}$ as a typical fiber; in particular, $b+f=n-1$, and
      \begin{equation}\label{eqn: metric-on-total-space}
      g_{\partial M}=\hat{g}+\pi^{*}\check{g},
      \end{equation}
      where $\hat{g}$ denotes metrics on fibers induced by $g_{\partial M}$.
\item Moreover, there exists a collar neighborhood, 
      $
      U := (0, 1)\times \partial M,
      $
      of the boundary $\partial M$ in $M$ with the radial coordinate $r$ $(r=0$ corresponds to the boundary $\partial M)$, and the metric $g$ restricted on $U$ can be written as $g=g_{0}+h$, where the Riemannian metric $g_{0}$ on $U$ is given by
    \begin{equation}\label{eqn: model conical metric}
    g_{0}=dr^{2}+r^{2}\hat{g}+\pi^{*}\check{g},
    \end{equation}
    and 2-tensor $h$ satisfies
    \begin{equation}\label{eqn: metric asymptotic control}
    |r^{k} \nabla^{k}_{g_{0}} h|_{g_{0}} = O(r^{\alpha}), \quad \text{as} \ \ r\rightarrow 0,
    \end{equation}
    for some $\alpha>0$ and all $k \in \mathbb{N}$.
\end{enumerate}
\end{defn}

\begin{remark}
{\rm We do not require the high order term $h$ in Definition \ref{defn: ManifoldWithNonisolatedConicalSingularities} to be smooth up to singularities.
}
\end{remark}

\begin{remark}
{\rm Note that the Riemannian metric $g_{\partial M}$ on $\partial M$ is not the usual restriction of the Riemannian metric $g$ to the boundary. In fact, in metric sense, the boundary $\partial M$ is collapsed fiberwisely to the base $B$. Thus $M$ is metrically completed by adding $B$ which forms the non-isolated singularities of the completed space. Along the normal direction of each point of this singular stratum, the metric is asymptotic to a cone over a typical fiber.
}
\end{remark}

Now we briefly recall some basic facts about Riemannian submersions, and for more details we refer to \cite{ONe66} and Chapter 9 in \cite{Bes87}. We consider the Riemannian submersion given in (\ref{eqn: Riemannian-submersion1}). Let $\mathcal{H}$ and $\mathcal{V}$ denote respectively the horizontal and vertical projections. Let $\nabla^{\partial M}$ denote the Levi-Civita connections on $(\partial M, g_{\partial M})$. The O'Neill's $A$-tensor and $T$-tensor are given by:
\begin{equation*}
\begin{aligned}
A_{X}Y 
&= \mathcal{H}\nabla^{\partial M}_{\mathcal{H} X}\mathcal{V}Y
+\mathcal{V}\nabla^{\partial M}_{\mathcal{H} X}\mathcal{H}Y,\\
T_{X}Y 
&= \mathcal{H}\nabla^{\partial M}_{\mathcal{V} X}\mathcal{V}Y
+\mathcal{V}\nabla^{\partial M}_{\mathcal{V} X}\mathcal{H}Y,
\end{aligned}
\end{equation*}
for any vector fields $X$ and $Y$ on $\partial M$.
The tensor $T$ is related to the second fundamental form of fibers, and it vanishes identically if and only if each fiber is totally geodesic.

By the formula (9.37) in \cite{Bes87}, the scalar curvature of $g_{0}$ in $(\ref{eqn: model conical metric})$ is given by
\begin{equation}\label{eqn: scalar curvature formula}
R_{g_{0}}=\frac{1}{r^{2}}(R_{\hat{g}}-f(f-1))+R_{\check{g}}\circ\pi-|A|^{2}-|T|^{2}-|N|^{2}-\check{\delta}N,
\end{equation}
where the horizontal vector field $N=\sum\limits^{f}_{j=1}T_{V_{j}}V_{j}$ is the mean curvature vector of fibers, $\check{\delta}N=-\sum\limits^{b}_{i=1}g_{\partial M}(\nabla^{\partial M}_{X_{i}}N, X_{i})$, and $\{V_{1}, \cdots, V_{f}, X_{1}, \cdots, X_{b}\}$ is a local orthonormal frame for $T(\partial M)$ with respect to $g_{\partial M}=\hat{g}+\pi^{*}\check{g}$ such that $V_{1}, \cdots, V_{f}$ are vertical and $X_{1}, \cdots, X_{b}$ are projectable and horizontal.


\section{Sobolev and Weighted Sobolev spaces}\label{section: Sobolev-weighted-Sobolev-spaces}

\noindent In this section, we recall some weighted Sobolev spaces on compact Riemannian manifolds with non-isolated conical singularities. We also review and establish some (unweighted) Sobolev and weighted Sobolev embeddings on these singular manifolds that we will use in the following sections.

Various weighted Sobolev spaces have been introduced and studied on manifolds with conical singularities in \cite{Maz91}, \cite{BP03}, \cite{CLW12},  \cite{Beh13}, \cite{DW18}, \cite{DW20}, \cite{Oik20}, \cite{KV21} and \cite{Ver21}. These weighted Sobolev spaces turn out to be very useful in the study of elliptic and heat equation problems on conically singular manifolds. In particular, weighted Sobolev inequality and weighted elliptic regularity estimate are important tools in these studies.  

Besides weighted Sobolev inequality, in the study of $W$-functional the (unweighted) Sobolev inequality is also necessary (see, \S$\ref{section: W-functional}$ below, and \cite{DW20}).  
In \cite{DY}, the usual $L^{2}$-Sobolev inequality on compact Riemannian manifolds with isolated conical singularities has been established by choosing a suitable partition for a cone and applying a Hardy inequality. We then noticed that their idea also works for $L^{p}$-Sobolev inequality for $1 < p < n $ in \cite{DW20}. Later, Oikonomopoulos \cite{Oik20} employed this idea to study the $C^\infty_0$ density problem for weighted Sobolev spaces on manifolds with conical singularities.  Now we further adapt this idea to non-isolated conical singularities case, and obtain the Sobolev inequality in Proposition \ref{prop: SobolevInequality}. To our knowledge, the (unweighted) Sobolev inequality on manifolds with non-isolated conical singularities and its proof are not in the literature. So we provide that in \S$\ref{subsect: Sobolev inequality}$.

In the isolated conical singularities case, thanks to the homogeneity of cone metric along the radial direction, weighted Sobolev inequality and elliptic estimate were obtained in \cite{Beh13} (also see \cite{DW18}) by employing the scaling technique, which had been used by Bartnik \cite{Bar86} for asymptotically Euclidean manifolds. 
However, in the non-isolated conical singularities case, due to the presence of $\pi^* \check{g}$ term, the model singular metric $g_0$ in $(\ref{eqn: model conical metric})$ does not have such homogeneity, and as a result, the scaling technique does not work any more. Instead, in \S$\ref{subsect: weighted Sobolev spaces}$, we derive the weighted Sobolev inequality in the non-isolated conical singularities case from the (unweighted) Sobolev inequality, which is obtained in Proposition $\ref{prop: SobolevInequality}$. The local version of the weighted Sobolev inequality was obtained in \cite{CLW12} in a different way.

In \S$\ref{subsect: compact weighted Sobolev embeddings}$, we then provide two compact weighted Sobolev embeddings in Propositions $\ref{prop: compactness on manifolds}$ and $\ref{prop: compact-weighted-embedding2}$ that are used in the proof of Theorem $\ref{thm: spectrum property}$ and in the study of the $W$-functional in \S$\ref{section: W-functional}$. These compact embedding results are generalizations, to the non-isolated conical singularities, of our previous results in Theorem 3.1 in \cite{DW18} and Proposition 3.6 in \cite{DW20} for isolated conical singularities.

In \S$\ref{subsect: weighted elliptic estimate}$, we establish a weighted elliptic estimate in Proposition $\ref{prop: weighted elliptic estimate}$. Again because the lack of the homogeneity of the model singular metric $g_0$ in $(\ref{eqn: model conical metric})$, we are not able to prove a weighed elliptic estimate in the non-isolated conical singularities case by using scaling technique, which was used in the isolated conical singularities case in \cite{Beh13} (also see \cite{DW18}). We notice that a weighted elliptic estimate on manifolds with (either isolated or non-isolated) conical singularities can be proved by doing integration by parts, using the asymptotic control of Riemannian curvature tensor near the conical singularities, and an approximation argument. This idea also works for compact smooth manifolds, and should also work for various non-compact manifolds with certain asymptotic controls for the metric tensor. As this proof is not in the literature, we provide it in \S$\ref{subsect: weighted elliptic estimate}$.

\subsection{Sobolev spaces and embedding}\label{subsect: Sobolev inequality}
For any $k \in \mathbb{N}$ and $p \geq 1$, as usual, the Sobolev space $W^{k, p}(M)$ is the completion of $C^{\infty}_{0} (\mathring{M})$ with respect to the Sobolev norm
\begin{equation}\label{eqn: Sobolev norm}
\|u\|_{W^{k, p}(M)} = \left( \int_{M} \left( \sum^{k}_{i=0} |\nabla^{i} u|^{p} d\vol_{g} \right) \right)^{\frac{1}{p}},
\end{equation}
where $\mathring{M}:=M\setminus\partial M$ is the interior of $M$.

\begin{prop}\label{prop: SobolevInequality}
Let $(M^{n}, g)$ be a compact Riemannian manifold with non-isolated conical singularities as in Definition \ref{defn: ManifoldWithNonisolatedConicalSingularities}.
\begin{enumerate}
\item For each $1<p<n$ with $\frac{np}{n-ip} \neq f$, $i=0, \cdots, k-l-1$, one has
\begin{equation}\label{eqn: Sobolev inequality}
 \|u\|_{W^{l, q}(M)}\leq C(M, g, p, k)\|u\|_{W^{k, p}(M)},
\end{equation}
for any $u\in C^{\infty}_{0}(\mathring{M})$ and any $1\leq q\leq q_{l}$, where $C(M, g, p, k)$ is a constant, and $l<k$ and $q_{l}$ satisfy $\frac{1}{q_{l}}=\frac{1}{p}-\frac{k-l}{n}>0$. Hence, there is a continuous embedding $W^{k, p}(M)\subset W^{l, q}(M)$, for any $1\leq q\leq q_{l}$.

\item For $k, j\in\mathbb{N}$ with $k-j>\frac{n}{p}$, one has
     \begin{equation}\label{eqn: Sobolev Cr inequality}
     \|u\|_{C^{j}(M)}\leq C(M, g, p, k) \|u\|_{W^{k, p}(M)},
     \end{equation}
     for any $u\in C^{\infty}_{0}(\mathring{M})$. Hence, there is a continuous embedding $W^{k, p}(M)\subset C^{j}(M)$. Here $C^{j}(M)$ is the completion of $C^{\infty}_{0}(\mathring{M})$ with respect to the norm
     $$\|u\|_{C^{j}(M)}:=\max\limits_{0\leq l \leq j} \sup\limits_{M}|\nabla^{l} u|.$$

\item If $k, l, j \in\mathbb{N}$ satisfy $k-j-l>\frac{n}{p}$, then any $u\in W^{k, p}(M)$ satisfies
      \begin{equation}\label{eqn: Asymptotic of functions in Sobolev spaces}
      |\nabla^{j} u| \leq C(M, g, k, l, j,) \|u\|_{W^{k, p}(M)} r^{l},
      \end{equation}
      near the singular set $r=0$.
\end{enumerate}
\end{prop}

Clearly, the key in the derivation of the above Sobolev inequalities is to establish them on the collar neighborhood $U$ of $\partial M$ in the item (3) in Definition \ref{defn: ManifoldWithNonisolatedConicalSingularities}. As in \cite{DY} and \cite{DW20}, a Hardy inequality (see, e.g. {\bf 330} on p. 245 in \cite{HLP34}) will play a crucial rule. So we recall it: for any $p>1$ and $a \neq 1$, one has
\begin{equation}\label{eqn: Hardy inequality}
\int^{\infty}_{0} |u|^{p} x^{-a} dx \leq \left( \frac{p}{|a-1|} \right)^{p} \int^{\infty}_{0} |u^{\prime}(x)|^{p} x^{p-a} dx,
\end{equation}
for any $u \in C^{\infty}_{0}((0, \infty))$. This then implies a Hardy inequality (c.f. (3.8) in \cite{DW20}) on $U :=(0, 1)\times \partial M$ with the metric $g_{0}= dr^{2} + r^{2} \hat{g} + \pi^{*}\check{g}$: for $p>1$ and $k \in \mathbb{N}$ with $pk \neq f$, and any $u \in C^{\infty}_{0}(U)$, we have
\begin{eqnarray}\label{eqn: cone-Hardy-inequality}
\int_{U} \frac{|u|^{p}}{r^{pk}} d\vol_{g_{0}}
& \leq & \left( \frac{p}{|f-pk|} \right)^{p} \int_{U} \frac{|\nabla u|^{p}_{g_{0}}}{r^{p(k-1)}} d\vol_{g_{0}}.
\end{eqnarray}

\begin{lem}\label{lem: cone-Sobolev-inequality}
For $1 < p < n$ with $p \neq f$, there is a constant $C$ depending only on $g_{0}$ and $p$ such that
\begin{equation}\label{eqn: Sobolev inequality on collar}
\|u\|_{L^{q}(U, g_{0})} \leq C \| \nabla u \|_{L^{p}(U, g_{0})},
\end{equation}
for any $u \in C^{\infty}_{0} (U)$, where $q = \frac{np}{n-p}$.

For $p>n$, there is a constant $C$ depending only on $g_{0}$ and $p$ such that
\begin{equation}\label{eqn: Morrey inquality on collar}
\|u\|_{C^{0}(U)} \leq C \|u\|_{W^{1, p}(U)}
\end{equation}
for any $u\in C^{\infty}_{0}(U)$.
\end{lem}

\begin{proof}
First of all, we note that because $B$ is compact, the metrics restricted to fibers $\hat{g}|_{\pi^{-1}(y)}$ are uniformly quasi-isometric, i.e. there exists a constant $C$ and a Riemannian metric $\hat{g}_{0}$ on $F^{f}$ such that $\frac{1}{C} \hat{g}_{0} \leq \hat{g}|_{\pi^{-1}(y)} \leq C \hat{g}_{0}$ for all $y \in B$.

Now we choose a finite open cover of the collar neighborhood $U$ as follows, so that on each piece the metric $g_{0}$ is quasi-isometric to the standard Euclidean metric on $\mathbb{R}^{n}$.

First, choose a finite open cover $\{ U_{i} \}$ of B, such that each $\{ U_{i}, \check{g}|_{U_{i}} \}$ is quasi-isometric to an open domain in the standard Euclidean space $\mathbb{R}^{b}$, and the submersion $\pi: \partial M \rightarrow B$ can be locally trivialized over this open cover, i.e. there exists diffeomorphisms $\varphi_{i}: \pi^{-1}(U_{i}) \rightarrow U_{i} \times F$. Let $\check{\rho}_{i}$ be a partition of unity subordinate to this open cover of $B$.

Next, choose a finite open cover $\{ V_{j} \}$ of $F$, such that each $V_{j}$ can be embedded into Euclidean unite sphere $\mathbb{S}^{f}$ and the metric $\hat{g}_{0}$ restricted to each $V_{j}$ is quasi-isometric to the the round sphere metric. Let $\hat{\rho}_{j}$ be a partition of unity subordinate to this open cover of $F$.

So $\{ \varphi^{-1}_{i}(U_{i} \times V_{j}) \}$ form a finite open cover of $\partial M$, and further we obtain a finite open cover $\{ (0, 1) \times  \varphi^{-1}_{i}(U_{i} \times V_{j})\}$ of $U = (0, 1) \times \partial M$, with a subordinated partition of unity $\rho_{ij} := \tilde{\pi}^{*} \big( [\pi^{*}(\check{\rho_{i}})] \cdot [(({\rm Pr}_{2}) \circ \varphi_{i})^{*}(\hat{\rho}_{j})] \big)$, where $\tilde{\pi}: (0, 1) \times \partial M \rightarrow \partial M$ and ${\rm Pr}_{2}: U_{i} \times F \rightarrow F$ are natural projections. Then by the construction, one easily sees that $g_{0}$ restricted to $(0, 1) \times \varphi^{-1}_{i}(U_{i} \times V_{j})$ is quasi-isometric to the standard Euclidean metric on $\mathbb{R}^{n}$, and more importantly
\begin{equation}\label{eqn: cone-partition-of-unity-derivative}
|d \rho_{ij}|_{g_0}(r, \theta) \leq C_{ij} r^{-1},
\end{equation}
where $C_{ij}$ are some constants. 

The usual Sobolev inequality on standard Euclidean spaces together with the estimate (\ref{eqn: cone-partition-of-unity-derivative}) and the Hardy inequality (\ref{eqn: cone-Hardy-inequality}) imply the Sobolev inequality (\ref{eqn: Sobolev inequality on collar}) on the collar neighborhood $U$ with metric $g_{0}= dr^{2} + r^{2} \hat{g} + \pi^{*}\check{g}$.

For Morrey's inequality (\ref{eqn: Morrey inquality on collar}), let $u\in C^{\infty}_{0}(U)$, because $u=\sum\limits_{i,j}\rho_{ij}u$, one has
\begin{eqnarray*}
    \|u\|_{C^{0}(U)}& \leq & \sum_{i, j} \| \rho_{ij}u \|_{C^{0}(U)} \cr
    & \leq & \sum_{i, j} C \|\rho_{ij}u\|_{W^{1, p}(U)}  \qquad \qquad (\text{Morrey's inequality on} \ \ \mathbb{R}^{n})\cr
    & \leq & \sum_{i, j} C \left(\int_{U}\left(|d (\rho_{ij}u)|^{p}_{g_0} + |\rho_{ij}u|^{p}\right)d\vol_{g_{0}}\right)^{\frac{1}{p}} \cr
    & \leq & \sum_{i, j} C \left(\int_{U}\left(|d \rho_{ij}|^{p}_{g_0} |u|^{p} + |\rho_{ij}|^{p} |d u|^{p}_{g_0} +|\rho_{ij}u|^{p}\right)d\vol_{g_{0}}\right)^{\frac{1}{p}} \cr
    & \leq & C  \left(\int_{U}\left(r^{-p} |u|^{p} + |d u|^{p}_{g_0} +|u|^{p}\right)d\vol_{g_{0}}\right)^{\frac{1}{p}}  \quad (\text{Estimate} (\ref{eqn: cone-partition-of-unity-derivative})) \cr
    & \leq & C \left(\int_{U}\left(|d u|^{p}_{g_0} +|u|^{p}\right)d\vol_{g_{0}}\right)^{\frac{1}{p}} \qquad (\text{Hardy inequality } (\ref{eqn: cone-Hardy-inequality}))\cr
    & = & C \|u\|_{W^{1, p}(U)}.
\end{eqnarray*}
In the above derivation, the constant $C$ may vary along steps.
\end{proof}
\noindent{\em Proof of Proposition \ref{prop: SobolevInequality}}.
(1) By using the usual Sobolev inequalities on compact manifolds and applying the Kato's inequality: $|\nabla |\nabla^{k}u || \leq |\nabla^{k+1} u|$, Lemma \ref{lem: cone-Sobolev-inequality} implies the Sobolev inequality in (\ref{eqn: Sobolev inequality}).

(2) By combining the inequality (\ref{eqn: Morrey inquality on collar}) and Morrey's inequality on compact smooth manifolds, one has that for any $q>n$ there is a constant $C$  such that
    \begin{equation}\label{eqn: Morrey inequality}
    \|u\|_{C^{0}(M)} \leq C \|u\|_{W^{1, q}(M)}
    \end{equation}
    for any $u\in W^{1, q}(M)$. 
    As in the smooth case, the general inequality in $(\ref{eqn: Sobolev Cr inequality})$ then follows from the Sobolev inequality $(\ref{eqn: Sobolev inequality})$ and the Morrey's inequality $(\ref{eqn: Morrey inequality})$: see e.g. \S5.6.3 in \cite{Eva} for more details.

(3) Let $\varphi$ be a cutoff function with support in $U=(0, 1)\times \partial M$, satisfying $\varphi\equiv1$ on $U_{\frac{1}{2}}:=(0, \frac{1}{2})\times \partial M \subset U$. For $u\in C^{\infty}_{0}(M)$, one has $\varphi u\in C^{\infty}_{0}(U)$. Then Hardy inequality (\ref{eqn: cone-Hardy-inequality}) and Sobolev inequalities (\ref{eqn: Sobolev Cr inequality}) imply
\begin{equation*}
\|r^{-l}|\nabla^{j}(\varphi u)|\|_{C^{0}(U_{\frac{1}{2}})} \leq \|r^{-l}|\nabla^{j}(\varphi u)|\|_{C^{0}(U)}
\leq C \|r^{-l}|\nabla^{j}(\varphi u)|\|_{W^{k-l-j, p}(U)} \leq C \|\varphi u\|_{W^{k, p}(U)}.
\end{equation*}
Thus
\begin{equation*}
\|r^{-l}|\nabla^{j} u|\|_{C^{0}(U_{\frac{1}{2}})} \leq C \| u\|_{W^{k, p}(M)},
\end{equation*}
since $\varphi \equiv 1$ on $U_{\frac{1}{2}}$. This completes the proof of (3) in Proposition \ref{prop: SobolevInequality}. 
\qed


\subsection{Weighted Sobolev spaces and embedding}\label{subsect: weighted Sobolev spaces}
In this subsection, we review some weighted Sobolev spaces and derive a weighted Sobolev inequality from Sobolev inequality in Proposition $\ref{prop: SobolevInequality}$.

Let $(M^{n}, g)$ be a compact Riemannian manifold with non-isolated conical singularities as defined in Definition \ref{defn: ManifoldWithNonisolatedConicalSingularities}.  For each $p\geq1$, $k\in\mathbb{N}$ and $\delta\in\mathbb{R}$, the {\em weighted Sobolev space} $W^{k, p}_{\delta}(M)$ is the completion of $C^{\infty}_{0}(\mathring{M})$ with respect to the {\em weighted Sobolev norm}
\begin{equation}\label{eqn: WSN}
\|u\|_{W^{k, p}_{\delta}(M)} = \left( \int_{M}\left(\sum^{k}_{i=0}\chi^{p(\delta-i)+n}|\nabla^{i}u|^{p}_{g}\right)d\vol_{g} \right)^{\frac{1}{p}},
\end{equation}
where $\nabla^{i}u$ denotes the $i$-times covariant derivative of the function $u$ with respect to the metric $g$, and $\chi\in C^{\infty}(\mathring{M})$ is a positive weight function satisfying
\begin{equation}\label{eqn: WeightFunction-cone}
\chi(x)=
\begin{cases} 1, & \text{if} \ \ x\in M\setminus U, \\ \frac{1}{r}, & \text{if} \ \  r={\rm dist}(x,\partial M)<\frac{1}{10},
\end{cases}
\end{equation}
and $0<(\chi(x))^{-1}\leq 1$ for all $x \in \mathring{M}$. Recall that $U=(0, 1)\times \partial M\subset M$ is a collar neighborhood of the boundary $\partial M$.

By the definition of the weighted Sobolev norm in $(\ref{eqn: WSN})$, we clearly have 
\begin{equation}\label{eqn: weighted Sobolev space inclusion}
\delta^\prime > \delta \ \ \Rightarrow \ \ 
W^{k, p}_{\delta^\prime}(M) \subset W^{k, p}_{\delta}(M).
\end{equation}

Note also that the integral $\int^1_0 r^{2(\mu -\delta) -1 }\left(\ln r \right)^p dr$ is finite if and only if $\mu > \delta$. Therefore, for a smooth function $u$ on $\mathring{M}$ satisfying \begin{equation}
u = O\left(r^\mu \left( \ln r \right)^p\right), \ \ \text{as} \ \ r \to 0,
\end{equation}
we have
\begin{equation}\label{eqn: polynomial weighted Sobolev condition}
u \in W^{k, 2}_{\delta}(M) \Longleftrightarrow \mu > \delta.
\end{equation}

The weighted Sobolev norm $\| \cdot \|_{W^{k, p}_{\delta}}$ defined in $(\ref{eqn: WSN})$ is essentially the same as the weighted Sobolev norm $\| \cdot \|_{\mathcal{H}^{k, p}_{\gamma}}$ with $\delta=\gamma-\frac{n}{p}$ in \cite{CLW12}.

Besides the Sobolev inequality in Proposition \ref{prop: SobolevInequality}, as in Lemma 3.1 in \cite{DW20}, another interesting consequence of the Hardy inequality  (\ref{eqn: cone-Hardy-inequality}) is the following equivalence between some of the weighted Sobolev norms in $(\ref{eqn: WSN})$ with certain weight indices and the usual Sobolev norms in $(\ref{eqn: Sobolev norm})$.
\begin{lem}\label{lem: Sobolev-weighted-Sobolev-equivalence}
Let $(M^{n}, g)$ be a compact Riemannian manifold with non-isolated conical singularities. For each $p>1$ and $k\ \in \mathbb{N}$ satisfying $pi \neq n$ for $i = 1, 2, \cdots, k$, there exists a constant $C = C(g, n, p, k)$ such that
\begin{equation*}
\|u\|_{W^{k, p}(M)} \leq \|u\|_{W^{k, p}_{k-\frac{n}{p}}(M)} \leq C \|u\|_{W^{k, p}(M)}.
\end{equation*}

Consequently, $W^{k, p}_{k-\frac{n}{p}}(M) = W^{k, p}(M)$ for any $p>1$ and $k \in \mathbb{N}$ satisfying $pi \neq n$ for all $i=1, 2, \cdots, k$.
\end{lem}


By Lemma \ref{lem: Sobolev-weighted-Sobolev-equivalence} and the Sobolev inequalities in Proposition \ref{prop: SobolevInequality}, one immediately obtains some weighted Sobolev inequalities with special weight indices. However, we also need weighed Sobolev inequalities with general indices, which can be established by using the Sobolev inequalities in Proposition \ref{prop: SobolevInequality} as follows.

\begin{prop}\label{prop: weighted-Sobolev-inequality}
Let $(M^{n}, g)$ be a compact Riemannian manifolds with non-isolated conical singularities. For any $\delta \in \mathbb{R}$, $1 \leq p < n$, $0 \leq l \leq k$, and $q$ with $\frac{1}{q} = \frac{1}{p} - \frac{k-l}{n} > 0$ and $\frac{np}{n-ip}\neq f, i=0, \cdots, k-l-1$, there exists a constant $C$, such that for any $u \in C^{\infty}_{0}(\mathring{M})$, 
\begin{equation*}
\|u\|_{W^{l, q}_{\delta }} \leq C \|u\|_{W^{k, p}_{\delta}}.
\end{equation*}
Consequently, there is a continuous embedding
\begin{equation*}
W^{k, p}_{\delta}(M) \subset W^{l, q}_{\delta}(M).
\end{equation*}
\end{prop}
\begin{proof}
For any $u\in C^{\infty}_{0}(\mathring{M})$, $\chi^{\delta-i+\frac{n}{q}}|\nabla^{i}u| \in C^{\infty}_{0}(\mathring{M})$ for any $i\in\mathbb{N}$ and $q>1$. Recall that $\chi$ is a weight function in $(\ref{eqn: WeightFunction-cone})$. Then for $q_{1}=\frac{np}{n-p}$, i.e. $\frac{1}{q_{1}}=\frac{1}{p}-\frac{1}{n}$, the Sobolev inequality in Proposition \ref{prop: SobolevInequality} implies
\begin{eqnarray*}
\|u\|_{W^{k-1, q_{1}}_{\delta}(M)}
& = & \left(\int_{M}\left(\sum^{k-1}_{i=0}\chi^{q_{1}(\delta-i)+n}|\nabla^{i}u|^{q_{1}}_{g}\right)d\vol_{g}\right)^{\frac{1}{q_{1}}} \\
& \leq & C \sum^{k-1}_{i=0} \|\chi^{\delta-i+\frac{n}{q_{1}}} |\nabla^{i}u|_{g}\|_{L^{q_{1}}} \\
& \leq & C \sum^{k-1}_{i=0} \|\chi^{\delta-i+\frac{n}{q_{1}}} |\nabla^{i}u|_{g}\|_{W^{1,p}(M)} \\
& \leq & C \sum^{k-1}_{i=0} \left( \|\nabla\left(\chi^{\delta-i+\frac{n}{q_{1}}} |\nabla^{i}u|_{g}\right)\|_{L^{p}(M)} + \|\chi^{\delta-i+\frac{n}{q_{1}}} |\nabla^{i}u|_{g}\|_{L^{p}(M)} \right) \\
& \leq & C \sum^{k-1}_{i=0} \big( \|\chi^{\delta-i+\frac{n}{q_{1}}+1}|\nabla^{i}u|_{g}\|_{L^{p}(M)} + \|\chi^{\delta-i+\frac{n}{q_{1}}}|\nabla^{i+1}u|_{g}\|_{L^{p}(M)} \\
&   & + \|\chi^{\delta-i+\frac{n}{q_{1}}}|\nabla^{i}u|_{g}\|_{L^{p}(M)}\big) \\
& \leq & C \sum^{k}_{i=0} \|\chi^{\delta-i+\frac{n}{p}}|\nabla^{i}u|_{g}\|_{L^{p}(M)} 
\leq C \|u\|_{W^{k, p}_{\delta}(M)}.
\end{eqnarray*}
In the third last inequality, we used $|\nabla \chi|_{g}\leq C \chi^{2}$. In the second last inequality, we used $0< \chi^{-1} \leq 1$, and $\frac{n}{q_{1}}=\frac{n}{p}-1$.

Then for any $0\leq l\leq k$, one can iterate the above inequality to complete the proof.
\end{proof}


\begin{remark}
{\rm
The local version of the weighted Sobolev inequalities in Proposition \ref{prop: weighted-Sobolev-inequality} has been estiblished in Proposition 3.2 and 3.3 in \cite{CLW12} in a different manner.
}
\end{remark}

\begin{prop}\label{prop: weighted Sobolev space regularity}
Let $(M^{n}, g)$ be a compact Riemannian manifold with non-isolated conical singularities. For any $\delta\in \mathbb{R}$ and $k, l \in \mathbb{N}$ with $k-l>\frac{n}{p}$, any $u\in W^{k, p}_{\delta}(M)$ satisfies
\begin{equation*}
|\nabla^{l} u| = o(r^{\delta-k+\frac{n}{p}}),
\end{equation*}
as $r\rightarrow 0$.
\end{prop}
\begin{proof}
First note that if $u\in W^{k, p}_{\delta}(M)$ then $\chi^{\delta-k+\frac{n}{p}}|\nabla^l u| \in W^{k-l, p}(M)$, where $\chi$ is the weight function in (\ref{eqn: WeightFunction-cone}) and it is equal to $\frac{1}{r}$ near the singular set.
Let $U_{j}=\left( \left(\frac{1}{2}\right)^{j}, \left(\frac{1}{2}\right)^{j-1}\right) \times \partial M$ for each $j\in\mathbb{N}$. The Sobolev inequality in Proposition \ref{prop: SobolevInequality} gives
\begin{equation*}
\|r^{-\delta+k-\frac{n}{p}}|\nabla^{l}u|\|_{C^{0}(U_{j})} \leq C \|r^{-\delta+k-\frac{n}{p}}|\nabla^{l}u|\|_{W^{k-l, p}(U_{j})} \leq C \|u\|_{W^{k, p}_{\delta}(U_{j})}.
\end{equation*}
Then because $\sum\limits^{\infty}_{j=1}\|u\|_{W^{k, p}_{\delta}(U_{j})} \leq \|u\|_{W^{k, p}_{\delta}(M)}\leq\infty$, $\|u\|_{W^{k, p}_{\delta}(U_{j})} \rightarrow 0$, as $j\rightarrow\infty$. Thus $|\nabla^{l}u|=o(r^{\delta-k+\frac{n}{p}})$ as $r\rightarrow0$.
\end{proof}

\subsection{Compact weighted Sobolev embeddings}\label{subsect: compact weighted Sobolev embeddings}
In this subsection, we prove that weighted Sobolev spaces $W^{k, 2}_{k-\frac{n}{2}}(M)$ can be compactly embedded into $L^{2}(M)$ space. This compact embedding plays a crucial rule in the study of the spectrum of the operator $-4\Delta_{g}+R_{g}$.



Firstly, we derive a compact embedding on the collar neighborhood: $U=(0, 1) \times \partial M$ of the singular set $\partial M$ with the model metric $g_{0}$ given in $(\ref{eqn: model conical metric})$. 

\begin{lem}\label{lem: compactness on cones}
The continuous embedding
\begin{equation*}
i:W^{k, 2}_{k-\frac{n}{2}}(U, g_{0})\hookrightarrow L^{2}(U, g_{0})
\end{equation*}
is compact for each $k\in\mathbb{N}$.
\end{lem}
The proof of Lemma \ref{lem: compactness on cones} is similar in spirit as the proof of Lemma 3.2 in \cite{DW18} for isolated conical singularities. Roughly speaking, one relates the problem to that on a cylinder. In \cite{DW18}, we did a spectral decomposition for functions on a cone with respect to eigenfunctions of Laplacian on the cross section, when deriving the inequality $(\ref{eqn: cone cylinder inequality})$ below. In the non-isolated conical singularities case, where we have a family of cones, we cannot do the spectral decomposition as in \cite{DW18}, and we derive $(\ref{eqn: cone cylinder inequality})$ without using the spectral decomposition. So we provide details here.
\begin{proof}
Clearly, $\|u\|_{W^{k, 2}_{k-\frac{n}{2}}(U, g_{0})}\geq\|u\|_{W^{l, 2}_{l-\frac{n}{2}}(U, g_{0})}$, for $k\geq l\in\mathbb{N}$. Therefore, it suffices to show that the embedding:
$i:W^{1, 2}_{1-\frac{n}{2}}(C_{\epsilon}(N))\hookrightarrow L^{2}(C_{\epsilon}(N))$, is compact.

Let $(C(\partial M), g_{Cyl})=((0, 1)\times\partial M, dr^{2}+g_{\partial M})$ be a finite cylinder. The usual Sobolev norm on the cylinder is given by
\begin{equation*}
\|u\|_{W^{1, 2}(C(\partial M), g_{Cyl})}=\int_{C(\partial M)}(|\nabla u|^{2}_{g_{Cyl}}+u^{2})dvol_{g_{Cyl}}.
\end{equation*}
Note that the mapping
\begin{equation*}
\begin{aligned}
L^{2}(U, g_{0}) &\longrightarrow L^{2}(C(\partial M), g_{Cyl})\\
u &\longmapsto \tilde{u}=r^{\frac{f}{2}}u,
\end{aligned}
\end{equation*}
is unitary. We will show that
\begin{equation}\label{eqn: cone cylinder inequality}
\|u\|_{W^{1, 2}_{1-\frac{n}{2}}(U, g_{0})}\geq\frac{3}{4}\|\tilde{u}\|_{W^{1,2}(C(\partial M), g_{Cyl})},
\end{equation}
for all $u\in C^{\infty}_{0}(U)$. This then completes the proof, since the embedding
\begin{equation*}
\mathring{W}^{1,2}(C(\partial M), g_{Cyl})\hookrightarrow L^{2}(C(\partial M), g_{Cyl})
\end{equation*}
is compact by the classical Rellich Lemma.

Now we prove the inequality $(\ref{eqn: cone cylinder inequality})$. For any $u\in C^{\infty}_{0}(U)$,
\begin{eqnarray*}
\|u\|^{2}_{W^{1, 2}_{1-\frac{n}{2}}(U, g_{0})}
& = &\int_{U}\left(\frac{1}{r^{2}}u^{2}+|\nabla u|^{2}_{g_{0}}\right)d\vol_{g_{0}}\\
& \geq & 
\int^{1}_{0}\int_{\partial M}\left(\frac{1}{r^{2}}u^{2}+\left|\frac{\partial}{\partial r}u\right|^{2}+|\nabla u|^{2}_{g_{\partial M}}\right)r^{f}d\vol_{g_{\partial M}}dr\\
& = & \int^{1}_{0}\int_{\partial M}\left(\frac{1}{r^{2}}\tilde{u}^{2}+\left|\frac{\partial}{\partial r}\tilde{u}\right|^{2}-\frac{f^{2}}{4}\frac{1}{r^{2}}\tilde{u}^{2}+|\nabla \tilde{u}|^{2}_{g_{\partial M}}\right)d\vol_{g_{\partial M}}dr\\
&   & -\, f\int_{\partial M}\int^{1}_{0}r^{f-1}u\left(\frac{\partial}{\partial r}u\right)drd\vol_{g_{\partial M}}\\
& = & \int^{1}_{0}\int_{\partial M}\left(\frac{1}{r^{2}}\tilde{u}^{2}+\left|\frac{\partial}{\partial r}\tilde{u}\right|^{2}-\frac{f^{2}}{4}\frac{1}{r^{2}}\tilde{u}^{2}+|\nabla \tilde{u}|^{2}_{g_{\partial M}}\right)d\vol_{g_{\partial M}}dr\\
&   & +\, \frac{f(f-1)}{2}\int^{1}_{1}\int_{\partial M}\frac{1}{r^{2}}\tilde{u}^{2}d\vol_{g_{\partial M}}dr\\
& = & \int_{C(\partial M)}\left(\frac{1}{r^{2}}\left(1+\frac{f(f-2)}{4}\right)\tilde{u}^{2}+|\nabla \tilde{u}|^{2}_{g_{Cyl}}\right)d\vol_{g_{Cyl}}\\
& \geq & \int_{C(\partial M)}\left(\frac{1}{r^{2}}\frac{3}{4}\tilde{u}^{2}+|\nabla \tilde{u}|^{2}_{g_{Cyl}}\right)d\vol_{g_{Cyl}}.\\
& \geq & \frac{3}{4}\|\tilde{u}\|^{2}_{W^{1,2}(C(\partial M), g_{Cyl})}.
\end{eqnarray*}
\end{proof}

Then by the asymptotic control for the metric $g$ on the collar neighborhood $U$ in condition $(3)$ of Definition $\ref{defn: ManifoldWithNonisolatedConicalSingularities}$, on a sufficiently small collar neighborhood $U_{\epsilon} = (0, \epsilon) \times \partial M$, metrics $g$ and $g_0$ are quasi-isometric to each other, and so the corresponding weighted Sobolev norms are equivalent. Thus, Lemma $\ref{lem: compactness on cones}$ implies the following compact embedding property.

\begin{prop}\label{prop: compactness on manifolds}
Let $(M^n, g)$ be a compact manifold with non-isolated conical singularities as in Definition $\ref{defn: ManifoldWithNonisolatedConicalSingularities}$.
The continuous embedding
\begin{equation*}
i:W^{k, 2}_{k-\frac{n}{2}}(M)\hookrightarrow L^{2}(M)
\end{equation*}
is compact for each $k\in\mathbb{N}$.
\end{prop}

We end this subsection by stating another compact weighted Sobolev embedding that will be used in proving the existence of a minimizer of $W$-functional.

\begin{prop}\label{prop: compact-weighted-embedding2}
Let $(M^{n}, g)$ be a compact Riemannian manifold with non-isolated conical singularities. The embedding $W^{1, 1}_{1-n}(M) \subset L^{q}(M)$ is compact for any $1 \leq q \leq \frac{n}{n-1}$.
\end{prop}

This proposition can be derived similarly as in the proof of Proposition 3.5 in \cite{DW20}, except here we need to use the finite open cover as in the proof of Lemma $\ref{lem: cone-Sobolev-inequality}$ for a collar neighborhood of $\partial M$.
Thus we omit the proof.


\subsection{Weighted elliptic estimate}\label{subsect: weighted elliptic estimate}
In this subsection, we establish a weighted elliptic estimate on compact manifolds with non-isolated conical singularities in Proposition $\ref{prop: weighted elliptic estimate}$. This will be used in the derivation of asymptotic behavior of eigenfunctions of Schr\"odinger operator $L:=-4\Delta+R$ and the minimizer of W-functional. 

\begin{prop}\label{prop: weighted elliptic estimate}
Let $(M, g)$ be a compact manifold with non-isolated conical singularities as in Definition $\ref{defn: ManifoldWithNonisolatedConicalSingularities}$. 
If $u\in W^{k+1, 2}_{\delta}(M) \cap C^{k+2}(\mathring{M})$ and $Lu\in W^{k, 2}_{\delta-2}(M)$, then
\begin{equation}\label{eqn: weighted elliptic estimate}
\|u\|_{W^{k+2, 2}_{\delta}(M)} \leq C \left( \|Lu\|_{W^{k, 2}_{\delta-2}(M)} + \|u\|_{W^{k+1, 2}_{\delta}(M)} \right)
\end{equation}
holds for some constant $C=C(M, k, \delta)$ independent of function $u$.
\end{prop}

\begin{remark} 
{\rm 
Note that this is slightly weaker than the usual elliptic estimate but it suffices for our purpose.
}
\end{remark}




\begin{proof}
Clearly, it suffices to prove that there exists a constant $C= C(M, k, \delta)$ such that
\begin{equation}\label{eqn: laplace elliptic estimate}
\|u\|_{W^{k+2, 2}_{\delta}(M)} \leq C_{k} \left( \|\Delta u\|_{W^{k, 2}_{\delta-2}(M)}+ \|u\|_{W^{k+1, 2}_{\delta}(M)} \right)
\end{equation}
holds for any $u \in W^{k+1, 2}_\delta (M) \cap C^{k+2}(\mathring{M})$ with $\Delta u \in W^{k, 2}_{\delta -2}(M)$, as
\begin{equation}\label{eqn: laplace L estimate}
\begin{aligned}
\|\Delta u\|_{W^{k, 2}_{\delta-2}(M, g)}
& \leq  \frac{1}{4}\|Lu\|_{W^{k, 2}_{\delta-2}(M)}+\frac{1}{4}\|Ru\|_{W^{k, 2}_{\delta-2}(M, g)} \\
& \leq   C \left( \|Lu\|_{W^{k, 2}_{\delta-2}(M, g)} + \|u\|_{W^{k, 2}_{\delta}(M, g)} \right).
\end{aligned}
\end{equation}
To prove the inequality $(\ref{eqn: laplace elliptic estimate})$, we only need to show it for the collar neighborhood $U=(0, 1) \times \partial M$, 
that is
\begin{equation}\label{eqn: cone laplace elliptic estimate}
\|u\|_{W^{k+2, 2}_{\delta}(U, g)} \leq C_{k} \left( \|\Delta u\|_{W^{k, 2}_{\delta-2}(U, g)}+ \|u\|_{W^{k+1, 2}_{\delta}(U, g)} \right),
\end{equation}
since the estimate on the smooth interior part $M \setminus U$ is classical.

In the rest of the proof, we prove the inequality $(\ref{eqn: cone laplace elliptic estimate})$ by two steps. In the first step, we prove the inequality $(\ref{eqn: cone laplace elliptic estimate})$ for $u\in C^{\infty}_{0}(U)$. This is done by direct computation, using integration by parts as well as commuting the derivatives. In the second step, we prove the inequality $(\ref{eqn: cone laplace elliptic estimate})$ for functions $u\in  C^{k+2}(U)\cap W^{k+2,2}_{\delta}(U, g)$ with $\Delta u\in W^{k,2}_{\delta-2}(U, g)$, by doing a cut-off and an approximation argument.

{\bf Step 1}. Let $u\in C^{\infty}_{0}(U)$. We have:
\begin{eqnarray*}
     \|u\|^{2}_{W^{k+2, 2}_{\delta}(U, g)} 
    & = & \int_{U}\left(\frac{1}{r}\right)^{2(\delta-k-2)+n}\sum^{n}_{i_{1}, \cdots, i_{k+2}=1}(\nabla_{i_{1}}\cdots\nabla_{i_{k+2}}u)(\nabla_{i_{1}}\cdots\nabla_{i_{k+2}}u) d\vol_g \\
   & &   + \|u\|^{2}_{W^{k+1, 2}_{\delta}(U, g)}.
\end{eqnarray*}
Here $\nabla_{i_{j}}$ denotes the covariant derivative $\nabla_{e_{i_{j}}}$ with respect to the metric $g$, where $\left\{e_{1}, e_{2}, \cdots, e_{n}\right\}$ is a local orthornormal frame on $(U, g)$. By doing the integration by parts for the integral in the first term on the right hand side, one obtains
\begin{eqnarray*}
   & & \|u\|^{2}_{W^{k+2, 2}_{\delta}(U, g)} \\
   & = &  -\int_{U} \sum^{n}_{i_{1}, \cdots, i_{k+2}=1} \nabla_{i_{1}}\left(\frac{1}{r}\right)^{2(\delta-k-2)+n}(\nabla_{i_{1}}\cdots\nabla_{i_{k+2}}u)(\nabla_{i_{2}}\cdots\nabla_{i_{k+2}}u) d\vol_{g} \\
   & &  -\int_{U}\left(\frac{1}{r}\right)^{2(\delta-k-2)+n} \sum^{n}_{i_{1}, \cdots, i_{k+2}=1} (\nabla_{i_{1}}\nabla_{i_{1}}\nabla_{i_{2}}\cdots\nabla_{i_{k+2}}u)(\nabla_{i_{2}}\cdots\nabla_{i_{k+2}}u) d\vol_{g}\\
   & &  + \int_U \left(\frac{1}{r}\right)^{2(\delta-k-2)+n} \sum^{n}_{i_{1}, \cdots, i_{k+2}=1} (\nabla_{\nabla_{e_{i_{1}}}e_{i_{1}}}\nabla_{i_{2}}\cdots\nabla_{i_{k+2}}u)(\nabla_{i_{2}}\cdots\nabla_{i_{k+2}}u) d\vol_{g} \\
  &  & + \|u\|^{2}_{W^{k+1, 2}_{\delta}(U, g)}.
\end{eqnarray*}
Then one can switch $\nabla_{i_{1}}$ with $\nabla_{i_{2}}, \cdots, \nabla_{i_{k+2}}$ in the second term on the right hand side to obtain $\nabla_{i_{2}}\cdots\nabla_{i_{k+2}}(\Delta u)$, with some additional terms involving the Riemann curvature tensor and its derivatives. Note that the $i^{th}$ derivatives of curvature can be controlled by $C_{i}\left(\frac{1}{r}\right)^{2+i}$ for some constant $C_{i}$. So by Cauchy-Schwarz inequality, one has
\begin{equation}
\begin{aligned}
   & \|u\|^{2}_{W^{k+2, 2}_{\delta}(U, g)} 
   \leq  \frac{1}{4}\|u\|^{2}_{W^{k+2, 2}_{\delta}(U, g)} + C \|u\|^{2}_{W^{k+1, 2}_{\delta}(U, g)} \\
   & \qquad  -\int_{U}\left(\frac{1}{r}\right)^{2(\delta-k-2)+n} \sum^{n}_{i_{2}, \cdots, i_{k+2}=1} (\nabla_{i_{2}}\cdots\nabla_{i_{k+2}}(\Delta u))(\nabla_{i_{2}}\cdots\nabla_{i_{k+2}}u) d\vol_{g}.
\end{aligned}
\end{equation}
Again, integrating by parts for the third term on the right hand side and switching the order of covariant derivatives, 
one obtains similarly
\begin{equation}
\begin{aligned}
  & \qquad  \|u\|^{2}_{W^{k+2, 2}_{\delta}(U, g)} \\
  & \leq  \frac{1}{2}\|u\|^{2}_{W^{k+2, 2}_{\delta}(U, g)} + C \|u\|^{2}_{W^{k+1, 2}_{\delta}(U, g)} \\
  &  \qquad  +\int_{U}\left(\frac{1}{r}\right)^{2(\delta-k-2)+n} \sum^{n}_{i_{3}, \cdots, i_{k+2}=1} (\nabla_{i_{3}}\cdots\nabla_{i_{k+2}}(\Delta u))(\nabla_{i_{3}}\cdots\nabla_{i_{k+2}}(\Delta u)) d\vol_{g}\\
  & \leq  \frac{1}{2}\|u\|^{2}_{W^{k+2, 2}_{\delta}(U, g)} + C \|u\|^{2}_{W^{k+1, 2}_{\delta}(U, g)} + \|\Delta u\|_{W^{k, 2}_{\delta-2}(U, g)}.
\end{aligned}
\end{equation}
Finally, by rearranging the inequality, one obtains
\begin{equation}
\|u\|_{W^{k+2, 2}_{\delta}(U, g)} \leq C \left( \|\Delta u\|_{W^{k, 2}_{\delta-2}(U, g)} + \|u\|_{W^{k+1, 2}_{\delta}(U, g)} \right),
\end{equation}
for some constant $C$ and any $u\in C^{\infty}_{0}(U)$.

{\bf Step 2}. 
Take a cut-off function $\varphi(r): (0, +\infty) \rightarrow [0, 1]$ such that
\begin{equation}
\varphi(r)=\begin{cases} 0, & r\leq 1, \\ 1, & r\geq  2, \end{cases}
\end{equation}
$0\leq \varphi \leq 1$ on $(1, 2)$, and $|\varphi^{(k)}|\leq C_{k}$ for each $k\in\mathbb{N}$ and some constant $C_{k}$.

For each $i\in \mathbb{N}\cup\{0\}$, let $\varphi_{i}(r):=\varphi(2^{i+1}r)$. Then
\begin{equation}
\varphi_{i}(r)=\begin{cases} 0, & r\leq\left(\frac{1}{2}\right)^{i+1}, \\ 1, & r \geq \left(\frac{1}{2}\right)^{i}, \end{cases}
\end{equation}
$0\leq \varphi_{i}\leq 1$, on $(2^{-(i+1)}, 2^{-i})$, and 
\begin{equation*}
|\varphi^{(k)}_{i}(r)|\leq C_{k}\left(\frac{1}{r}\right)^{k},
\end{equation*}
for any $i, k\in \mathbb{N} \cup \{0\}$ and a constant $C_{k}$.

For any $u\in C^{k+2}(U)\cap W^{k+1,p}_{\delta}(U, g)$ with $\Delta u\in W^{k,2}_{\delta-2}(U, g)$, let $u_{i}=\varphi_{i}(r)u\in C^{k+2}_{0}(U)$. Then one can check that
\begin{equation*}
u_{i} \rightarrow u \ \ \text{in} \ \ W^{k+1, 2}_{\delta}(U, g), \quad \text{as} \ \ i\rightarrow \infty.
\end{equation*}
\begin{equation*}
\Delta u_{i}\rightarrow \Delta u \ \ \text{in} \ \ W^{k, 2}_{\delta-2}(U, g), \quad \text{as} \ \ i\rightarrow \infty.
\end{equation*}

Consequently, we obtain that the inequality (\ref{eqn: cone laplace elliptic estimate}) holds for any $u \in C^{k+2}(U) \cap W^{k+1, 2}_\delta(U, g)$ with $\Delta u \in W^{k, 2}_{\delta -2}(U, g)$, since in step 1 we have shown that it holds for any $u \in C^\infty_0(U, g)$. This completes the proof of the proposition.
\end{proof}


\section{Spectrum of $-4\Delta +R$ on compact manifolds with non-isolated conical singularities}\label{section: spectrum}
\noindent In this section, we study the spectrum of the Schr\"odinger operator $-4\Delta+R$ on compact Riemannian manifolds with non-isolated conical singularities. First we obtain a semi-boundedness estimate for the operator $-4\Delta_{g_0}+R_{g_0}$ on a collar neighborhood of $\partial M$ in Lemma $ \ref{lem: cone semiboundedness estimate}$. This then implies a similar estimate for the operator $-4\Delta_g+R_g$ in Proposition \ref{prop: semiboundedness estimate}. Finally, we obtain the spectral properties of the operator $-4\Delta+R$ on compact manifolds with non-isolated conical singularities in Theorem $\ref{thm: spectrum property}$.

In the following, we still set $L=-4\Delta_{g_0}+R_{g_0}$. 

\begin{lem}\label{lem: cone semiboundedness estimate}
Let $(M^n, g)$ be a compact manifold with non-isolated conical singularities defined as in Definition $\ref{defn: ManifoldWithNonisolatedConicalSingularities}$ satisfying $\min\limits_{\partial M}\{R_{\hat{g}}\}>(f-1)$. Then for a sufficiently small $\epsilon>0$, on the collar neighborhood $U_{\epsilon}=(0, \epsilon)\times \partial M$ of the singular set $\partial M$ with the model metric $g_{0}$, we have
\begin{equation*}
(Lu, u)_{L^{2}(U_{\epsilon}, g_{0})}\geq \delta_{0}\|u\|^{2}_{W^{1, 2}_{1-\frac{n}{2}}(U_{\epsilon}, g_{0})},
\end{equation*}
for all $u\in C^{\infty}_{0}(U_{\epsilon})$, and some $\delta_{0}>0$ depending only on $\min\limits_{\partial M}\{R_{\hat{g}}\}$ and $f$.
\end{lem}

\begin{proof}
Because $\min\limits_{\partial M}\{R_{\hat{g}}\}>(f-1)$, and
\begin{equation}
f(f-1)+2\delta-\frac{4-\delta}{4}(f-1)^{2} \rightarrow f-1, \ \ \text{as} \ \ \delta \to 0,
\end{equation}
there exists a sufficiently small $\delta_{0}>0$ such that
\begin{equation}
\min\limits_{\partial M}\{R_{\hat{g}}\} > f(f-1)+2\delta_{0}-\frac{4-\delta_{0}}{4}(f-1)^{2}.
\end{equation}
Then choose a sufficiently small $\epsilon$ such that on $U_{\epsilon}=(0, \epsilon)\times\partial M$,
\begin{equation} \label{control of R}
\frac{\delta_{0}}{r^{2}}+R_{\check{g}}\circ\pi-|A|^{2}-|T|^{2}-|N|^{2}-2\check{\delta}N>0.
\end{equation}

Set
\begin{equation}
L_{\delta_{0}}:=-(4-\delta_{0})\Delta_{g_{0}}+R_{g_{0}}-\frac{1}{r^{2}}\delta_{0}.
\end{equation}
Then we have
\begin{equation}
L=L_{\delta_{0}}-\delta_{0}\Delta_{g_{0}}+\frac{1}{r^{2}}\delta_{0},
\end{equation}
and for any $u\in C^{\infty}_{0}(U_{\epsilon})$,
\begin{eqnarray}
(Lu, u)_{L^{2}(U_{\epsilon}, g_{0})}
& = & \int_{U_{\epsilon}}(L_{\delta_{0}}u)ud\vol_{g_{0}}+\int_{U_{\epsilon}}\left((-\delta_{0}\Delta_{g_{0}}u)u
+\frac{1}{r^{2}}\delta_{0}u^{2}\right)d\vol_{g_{0}}  \nonumber \\
& = & \int_{U_{\epsilon}}(L_{\delta_{0}}u)ud\vol_{g_{0}}+\delta_{0}\int_{U_{\epsilon}}\left(|\nabla u|^{2}_{g_{0}}+\frac{1}{r^{2}}u^{2}\right)d\vol_{g_{0}} \nonumber \\
& = & (L_{\delta_{0}}u, u)_{L^{2}(U_{\epsilon}, g_{0})}+\delta_{0}\|u\|^{2}_{W^{1, 2}_{1-\frac{n}{2}}(U_{\epsilon}, g_{0})}.
\end{eqnarray}

Thus, it suffices to show that $(L_{\delta_{0}}u, u)_{L^{2}(U_{\epsilon}, g_{0})}\geq0$ to complete the proof. Indeed, we claim that:
\begin{equation}\label{eqn: L estimate claim}
(L_{\delta_{0}}u, u)_{L^{2}(U_{\epsilon}, g_{0})}\geq C\|u\|^{2}_{L^{2}(U_{\epsilon}, g_{0})},
\end{equation}
holds for any $u\in C^{\infty}_{0}(U_{\epsilon})$, where
\begin{equation}\label{eqn: L estimate C}
C=\min\left\{(4-\delta_{0})\frac{(f-1)^{2}}{4}+\min\limits_{\partial M}\{R_{\hat{g}}\}-2\delta_{0}-f(f-1), 1\right\}>0.
\end{equation}

In the rest of the proof, we prove the claim in $(\ref{eqn: L estimate claim})$.
We start with:
\begin{eqnarray}
&   & (L_{\delta_{0}}u, u)_{L^{2}(U_{\epsilon}, g_{0})} \nonumber \\
& = & \int_{U_{\epsilon}}\left((4-\delta_{0})|\nabla u|^{2}_{g_{0}}+\left(R_{g_{0}}-\frac{\delta_{0}}{r^{2}}\right)u^{2}\right)d\vol_{g_{0}} \nonumber \\
& \geq & \int_{U_{\epsilon}}\left((4-\delta_{0})\left|\frac{\partial}{\partial r}u\right|^{2}+\left(R_{g_{0}}-\frac{\delta_{0}}{r^{2}}\right)u^{2}\right)r^{f}d\vol_{g_{\partial M}}dr \label{eqn: L_delta estimate}  
\end{eqnarray}
Setting $\tilde{u}:=r^{\frac{f}{2}}u$, then the right hand side of $(\ref{eqn: L_delta estimate})$ can be rewritten as
\begin{eqnarray}
&  & \int^{\epsilon}_{0} \int_{\partial M}\left((4-\delta_{0})\left|\frac{\partial}{\partial r}\tilde{u}\right|^{2}-\frac{(4-\delta_{0})f^{2}}{4r^{2}}\tilde{u}^{2}+\left(R_{g_{0}}-\frac{\delta_{0}}{r^{2}}\right)\tilde{u}^{2}\right)drd\vol_{g_{\partial M}} \nonumber \\
&   &  -\, (4-\delta_{0})f\int_{\partial M}\int^{\epsilon}_{0}r^{f-1}\tilde{u}\left(\frac{\partial}{\partial r}\tilde{u}\right)drd\vol_{g_{\partial M}} \nonumber 
\end{eqnarray}
It simplifies via integration by parts to
\begin{eqnarray}
 &   & \int^{\epsilon}_{0}\int_{\partial M}\Bigg[(4-\delta_{0})\left|\frac{\partial}{\partial r}\tilde{u}\right|^{2}+(4-\delta_{0})\frac{f(f-2)}{4}\frac{1}{r^{2}}\tilde{u}^{2} \nonumber \\
&   & +\left(R_{g_{0}}-\frac{\delta_{0}}{r^{2}}\right)\tilde{u}^{2}\Bigg]d\vol_{g_{\partial M}}dr. \label{simplified}
\end{eqnarray}
Now apply the one dimensional Hardy inequality $(\ref{eqn: Hardy inequality})$ to the first term and use the formula (\ref{eqn: scalar curvature formula}) for the scalar curvature $R_{g_{0}}$ combined with \eqref{control of R}. Then the quantity in \eqref{simplified} is
\begin{eqnarray}
& \geq & \int^{\epsilon}_{0}\int_{\partial M}\Big[\frac{1}{4}(4-\delta_{0})\frac{1}{r^{2}}\tilde{u}^{2}+(4-\delta_{0})\frac{f(f-2)}{4}\frac{1}{r^{2}}\tilde{u}^{2} \nonumber \\
&     & +\, \frac{1}{r^{2}}(R_{\hat{g}}-2\delta_{0}-f(f-1))\tilde{u}^{2}\Big]d\vol_{g_{\partial M}}dr \nonumber \\
&     & +\, \int^{\epsilon}_{0}\int_{\partial M}\left(\frac{\delta_{0}}{r^{2}}+R_{\check{g}}\circ\pi-|A|^{2}-|T|^{2}-|N|^{2}-\check{\delta}N\right)d\vol_{g_{\partial M}}dr \nonumber \\
& \geq & \int^{\epsilon}_{0}\int_{\partial M}\left((4-\delta_{0})\frac{(f-1)^{2}}{4}+\min\limits_{\partial M}\{R_{\hat{g}}\}-2\delta_{0}-f(f-1)\right)\frac{1}{r^{2}}\tilde{u}^{2}d\vol_{g_{\partial M}}dr \nonumber\\
& \geq & C\int^{\epsilon}_{0}\int_{\partial M}\tilde{u}^{2}d\vol_{g_{\partial M}}dr= C\|u\|^{2}_{L^{2}(U_{\epsilon}, g_{0})}. \label{fn}
\end{eqnarray}
Here the constant $C$ is given by $(\ref{eqn: L estimate C})$. Combining \eqref{eqn: L_delta estimate},  \eqref{simplified}, \eqref{fn}, this proves the claim in $(\ref{eqn: L estimate claim})$, and completes the proof of the lemma.
\end{proof}

Then because on a sufficiently small collar neighborhood of the boundary the metric $g$ is uniformly equivalent to $g_{0}$ and the manifold is compact, Lemma \ref{lem: cone semiboundedness estimate} implies the following semi-boundedness estimate on the manifold with non-isolated conical singularities.

\begin{prop}\label{prop: semiboundedness estimate}
Let $(M^n, g)$ be a compact Riemannian manifold with non-isolated conical singularities with $\inf\limits_{\partial M}\{R_{\hat{g}}\}>(f-1)$. Then there exists a large enough constant $A$ such that the operator $L_{A}:=-4\Delta_{g}+R_{g}+A$ satisfies
\begin{equation}
(L_{A}u, u)_{L^{2}(M, g)}\geq C\|u\|^{2}_{W^{1, 2}_{1-\frac{n}{2}}(M, g)},
\end{equation}
for any $u\in C^{\infty}_{0}(\mathring{M})$ and some constant $C>0$. In particular, the operator $L_{A}$ with the domain ${\rm Dom}(L_{A})=C^{\infty}_{0}(\mathring{M})$ is strictly positive.
\end{prop}

\begin{remark}
{\rm
The constant $A$ in Proposition \ref{prop: semiboundedness estimate} is need to control the scalar curvature of the metric $g$ outside of a collar neighborhood of the boundary, since we do not impose any restrictions for it.
}
\end{remark}

\begin{thm}\label{thm: spectrum property}
Let $(M, g)$ be a compact Riemannian manifold with non-isolated conical singularities with $\inf\limits_{\partial M}\{R_{\hat{g}}\}>(f-1)$. Then the spectrum of the Friedrichs extension of the operator $-4\Delta_{g}+R_{g}$ with domain $C^{\infty}_{0}(\mathring{M})$ consists of discrete eigenvalues with finite multiplicities
\begin{equation}
\lambda_{1}\leq \lambda_{2}\leq \lambda_{2}\leq \cdots,
\end{equation}
and $\lambda_{k}\rightarrow\infty$. The corresponding eigenfunctions $\{u_{i}\}^{\infty}_{i=1}$ form a complete basis of $L^{2}(M)$. Moreover, eigenfunctions $u_{i}$ are smooth on $\mathring{M}$ and satisfy the usual eigenfunction equation on $\mathring{M}$.
\end{thm}
\begin{proof}
The existence of the self-adjoint, strictly positive and surjective Friedrichs extension $\tilde{L}_{A}$ with domain ${\rm Dom}(\tilde{L}_{A})$ of the operator $L_{A}$ in Proposition \ref{prop: semiboundedness estimate} follows from the Neumann Theorem in \cite{EK96} and the estimate in Proposition \ref{prop: semiboundedness estimate}. Furthermore, Proposition \ref{prop: semiboundedness estimate} also implies that the completion of $C^{\infty}_{0}(\mathring{M})$ with respect to the norm $\|u\|_{L_{A}}:=(L_{A}u, u)_{L^{2}(M, g)}$ is contained in the weighted Sobolev space $W^{1, 2}_{1-\frac{n}{2}}(M)$. Thus from the construction of the Friedrichs extension in the proof of the Neumann theorem in \cite{EK96}, one can see that ${\rm Dom}(\tilde{L}_{A})\subset H^{1}(M)$.

Because $\tilde{L}_{A}: {\rm Dom}(\tilde{L}_{A})\rightarrow L^{2}(M)$ is one-to-one and onto, the inverse
\begin{equation}
\tilde{L}_{A}^{-1}: L^{2}(M)\rightarrow {\rm Dom}(\tilde{L}_{A}) \hookrightarrow W^{1, 2}_{1-\frac{n}{2}}(M) \hookrightarrow L^{2}(M)
\end{equation}
exists. This inverse map is self-adjoint and compact, since the embedding $W^{1, 2}_{1-\frac{n}{2}}(M)\hookrightarrow L^{2}(M)$ is compact by Proposition \ref{prop: compactness on manifolds}. Then the spectrum theorem of self-adjoint compact operators implies the desired spectral properties of the Friedrichs extension of $L_{A}$, and hence those for the Friedrichs extension of $L=-4\Delta_{g}+R_{g}$ stated in the theorem.

Finally, the regularity of the eigenfunctions follows from the standard elliptic regularity theory for weak solutions of eigenvalue equations, since this is a local property.
\end{proof}



\section{Asymptotic behavior of eigenfunctions}\label{section: asymptotic behavior of eigenfuctions}

In this section, we study the asymptotic behavior of eigenfunctions of $-4\Delta + R$ on manifolds with non-isolated conical singularities and prove Theorem \ref{thm: asymptotic of eigenfunctions} by a Nash-Moser iteration argument. 


\begin{prop}\label{prop: Nash-Moser iteration}
Let $(M^n, g)$ be a $n$-dimensional compact manifold with non-isolated conical singularities as in Definition $\ref{defn: ManifoldWithNonisolatedConicalSingularities}$, and $u$ an eigenfunction of $-4\Delta_g + R_g$ with eigenvalue $\lambda$. Then $u$ satisfies 
\begin{equation}\label{eqn: eigenfunction estimate}
\| u \|_{L^{\infty}\left( \left( \frac{1}{4} \epsilon, \frac{1}{2} \epsilon \right) \times \partial M, g \right)}
\leq
C \cdot \epsilon^{-\frac{n}{2}+1} \cdot \| u \|_{W^{2, 2}_{1-\frac{n}{2}}\left( \left (\frac{1}{8}\epsilon, \epsilon \right) \times \partial M, g \right)},
\end{equation}
and
\begin{equation}\label{eqn: eigenfunction gradient estimate}
\| \nabla u \|_{L^{\infty}\left( \left( \frac{1}{4} \epsilon, \frac{1}{2} \epsilon \right) \times \partial M, g \right)} 
\leq 
C \cdot \epsilon^{-\frac{n}{2}} \cdot \| u \|_{W^{2, 2}_{1-\frac{n}{2}}\left( \left (\frac{1}{8}\epsilon, \epsilon \right) \times \partial M, g \right)},
\end{equation}
for some constant $C$ and any $\epsilon<1$.
\end{prop}
\begin{proof}

Let $w=|\nabla u|^{2} + \frac{1}{r^{2}}u^{2}$. We apply Nash-Moser iteration to $w$, to obtain the estimates for $u$ in $(\ref{eqn: eigenfunction estimate})$ and $(\ref{eqn: eigenfunction gradient estimate})$.   

\noindent {\bf Step 1.} We derive the following differential inequality for $w$:
\begin{equation}\label{eqn: Laplace estimate}
\Delta w \geq -|R|w - |\lambda|w - C r^{-2}w.
\end{equation}

The Bochner formula implies
\begin{equation}
\begin{split}
 \frac{1}{2} \Delta |\nabla u|^{2}
& = |{\rm Hess}\, u|^{2} + \langle \nabla u, \nabla \Delta u \rangle + \Ric(\nabla u, \nabla u) \\
& \geq \frac{1}{4}\langle \nabla u, \nabla ( Ru- \lambda u )\rangle + \Ric(\nabla u, \nabla u) \\
& = \frac{1}{4} R |\nabla u|^{2} + \frac{1}{4}\langle\nabla u, \nabla R \rangle u -\frac{\lambda}{4}|\nabla u|^{2} + \Ric(\nabla u, \nabla u) \\
& \geq -\frac{1}{4} |R| w -\frac{|\lambda|}{4}w - C_0r^{-2}w - \frac{1}{4}|\nabla u|\cdot |\nabla R|\cdot |u|,
\end{split}
\end{equation}
since Ricci curvature is bounded by $C_0 r^{-2}$.
Note that $|\nabla R|\leq C_1r^{-3}$ near the conically singular points for some positive constant $C_1$. Thus,
\begin{equation}\label{eqn: Laplace estimate 1}
\begin{split}
 \frac{1}{2} \Delta |\nabla u|^{2}
& \geq -\frac{1}{4} |R| w -\frac{|\lambda|}{4}w - C_0r^{-2}w - C_1 \frac{1}{4}r^{-3}|\nabla u|\cdot |u| \\
& \geq -\frac{1}{4} |R| w -\frac{|\lambda|}{4}w - C_0r^{-2}w - C_1 \frac{1}{8} r^{-2}(|\nabla u|^{2} + r^{-2}u^{2}) \\
& = -\frac{1}{4} |R| w -\frac{|\lambda|}{4}w - C_0r^{-2}w - C_1 \frac{1}{8} r^{-2} w
\end{split}
\end{equation}
Moreover, 
\begin{equation}\label{eqn: Laplace estimate 2}
\begin{split}
\Delta (r^{-2}u^{2})
& = \Delta (r^{-2}) u^{2} + 2 \langle \nabla r^{-2}, \nabla u^{2} \rangle + r^{-2} \Delta u^{2} \\
& \geq 4 r^{-4}u^{2} - 4r^{-3}|\nabla u| \cdot |u| - \frac{1}{2}|R|w - \frac{1}{2} |\lambda|w \\
& \geq 2 r^{-2}w - \frac{1}{2} |R|w - \frac{1}{2} |\lambda|w 
\end{split}
\end{equation}
By combining inequalities in $(\ref{eqn: Laplace estimate 1})$ and $(\ref{eqn: Laplace estimate 2})$, we obtain the differential inequality in $(\ref{eqn: Laplace estimate})$

\noindent{\bf Step 2. Iteration process.} 

For any $q>1$ and $\eta \in C^{\infty}_{0}(\mathring{M})$, straightforward calculations give
\begin{eqnarray}
 \int | \nabla(\eta w^{\frac{q}{2}}) |^{2}
& = & \int w^{q} |\nabla \eta|^{2} - \frac{q}{2} \int \eta^{2} w^{q-1} \Delta w \nonumber\\
&   & -  \left( 1 - \frac{2}{q} \right) \int | \nabla(\eta w^{\frac{q}{2}}) |^{2} - \left( 1 - \frac{2}{q} \right) \int \eta w^{q} \Delta \eta.
\end{eqnarray}
By rearranging the equation and then using the inequality (\ref{eqn: Laplace estimate}), we obtain
\begin{eqnarray}\label{eqn: Nash-Moser}
\int | \nabla(\eta w^{\frac{q}{2}}) |^{2}
& = & \frac{q}{2(q-1)} \int w^{q} |\nabla \eta|^{2} - \frac{q^{2}}{4(q-1)} \int \eta^{2} w^{q-1} \Delta w - \frac{q-2}{2(q-1)} \int \eta w^{q} \Delta \eta  \nonumber \\
& \leq & \frac{q}{2(q-1)} \int w^{q} |\nabla \eta|^{2} + \frac{q^{2} |\lambda|}{8 (q-1)} \int \eta^{2} w^{q} + \frac{q^{2}}{8 (q-1)} \int \eta^{2} w^{q} |R|\\
&   & + \frac{q^{2}C}{2(q-1)} \int \eta^{2} r^{-2} w^{q} - \frac{q-2}{2(q-1)} \int \eta w^{q} \Delta \eta \nonumber
\end{eqnarray}

Define $q_{k} := \mu^{k}$, for $k \geq 0$, where $\mu = \frac{n}{n-2}$, and
\begin{eqnarray}
a_{k} & := & \left( \frac{1}{16} + \frac{1}{16} \sum^{k}_{i=0} \left( \frac{1}{2} \right)^{i} \right)\epsilon, \\
b_{k} & := & \left( \frac{5}{4} - \frac{1}{4} \sum^{k}_{i=0} \left( \frac{1}{2} \right)^{i} \right)\epsilon.
\end{eqnarray}
Then let $U_{k} := (a_{k}, b_{k}) \times \partial M \subset U$ with the restricted metric $g|_{U_{k}}$. Choose cut-off functions $\eta_{k} := \varphi_{k}(r) \in C^{\infty}_{0} (U_{k})$ such that
\begin{equation}
\eta_{k} \equiv 1, \ \ \text{on} \ \ U_{k+1}; \quad |\varphi^{\prime}_{k}| \leq 2^{k+5} \epsilon^{-1}, \quad |\varphi^{\prime \prime}_{k}| \leq 2^{2k+10} \epsilon^{-2}.
\end{equation}
Now for any $k \geq 1$, applying the Sobolev inequality (with $l=0, k=1$, $p=2$ and $ q=\frac{2n}{n-2}$) in Proposition \ref{prop: SobolevInequality} and substituting $\eta_{k}$ into the inequality in (\ref{eqn: Nash-Moser}) give
\begin{eqnarray}
\left( \int_{U_{k+1}} w^{q_{k+1}} \right)^{\frac{1}{\mu}}
& \leq & \left( \int_{U_{k}} \left( \eta^{2}_{k} w^{q_{k}} \right)^{\mu} \right)^{\frac{1}{\mu}}\nonumber  \\
& \leq & C \left( \int_{U_{k}} \left| \nabla \left( \eta_{k} w^{\frac{q_{k}}{2}} \right) \right|^{2} + \int_{U_{k}}  \eta^{2}_{k} w^{q_{k}} \right) \nonumber \\
&  \leq   & C \Bigg( \frac{q_{k}}{2(q_{k}-1)} \int_{U_{k}} w^{q_{k}} |\nabla \eta_{k}|^{2} + \frac{q_{k}^{2} |\lambda|}{8 (q_{k}-1)} \int_{U_{k}} \eta^{2}_{k} w^{q_{k}} \nonumber \\
&   & +\frac{q_{k}^{2}}{8 (q_{k}-1)} \int_{U_{k}} \eta^{2}_{k} w^{q_{k}} |R| + \frac{q_{k}^{2}C}{2(q_{k}-1)} \int_{U_{k}} \eta^{2}_{k} r^{-2} w^{q_{k}} \nonumber \\
&   &  - \frac{q_{k}-2}{2(q_{k}-1)} \int_{U_{k}} \eta_{k} w^{q_{k}} \Delta \eta_{k}   + \int_{U_{k}} r^{-2} \eta^{2}_{k} w^{q_{k}} \Bigg) \nonumber \\
& \leq & C \cdot (4 + |\lambda|) \cdot q^{2}_{k} \cdot 2^{2k+10} \cdot \epsilon^{-2} \int_{U_{k}} w^{q_{k}}.
\end{eqnarray}

Then by iterating the above process, one has
\begin{equation}\label{eqn: iteration inequality}
\begin{split}
\|w\|_{L^{q_{k+1}}(U_{k+1})}
& \leq  ( C \cdot (4 + |\lambda|))^{\frac{1}{q_{k}}} \cdot q^{\frac{2}{q_{k}}}_{k} \cdot 2^{\frac{2k+10}{q_{k}}} \cdot \epsilon^{-2\frac{1}{q_{k}}} \cdot \|w\|_{L^{q_{k}}(U_{k})} \\
& \leq  \cdots\\
& \leq  ( C \cdot (4 + |\lambda|))^{\sum\limits^{k}_{i=1}\frac{1}{q_{i}}} \cdot \left( \prod^{k}_{i=1} q^{\frac{2}{q_{i}}}_{i} \right) \cdot 2^{\sum\limits^{k}_{i=1} \frac{2i + 10}{q_{i}}} \cdot \epsilon^{-2\sum\limits^{k}_{i=1}\frac{1}{q_{i}}} \cdot \|w\|_{L^{q_{1}}(U_{1})}.
\end{split}
\end{equation}
By letting $k \rightarrow \infty$ in the inequality in $(\ref{eqn: iteration inequality})$, we obtain
\begin{equation}\label{eqn: iteration limit inequality}
\|w\|_{L^{\infty}(U_{\infty})} \leq C \cdot \epsilon^{-(n-2)} \cdot \|w\|_{L^{q_{1}}(U_{1})}.
\end{equation}

\noindent {\bf Step 3.} We estimate $\|w\|_{L^{q_{1}}(U_{1})}$ and complete the proof. By using Sobolev inequality, we have
\begin{eqnarray}
\||\nabla u|^{2}\|_{L^{q_{1}}(U_{1})} 
& = & \|\nabla u\|^{2}_{L^{\frac{2n}{n-2}}(U_{1})}\nonumber \\
& \leq &  \|\eta_{0}|\nabla u|\|^{2}_{L^{\frac{2n}{n-2}}(U_{0})}\nonumber \\
& \leq & C \left( \|\nabla (\eta_{0}|\nabla u|)\|^{2}_{L^{2}(U_{0})} + \|\eta_{0}|\nabla u|\|^{2}_{L^{2}(U_{0})} \right)\nonumber \\
& \leq & C \left( \|\nabla^{2} u \|^{2}_{L^{2}(U_{0})} + \epsilon^{-2}\|\nabla u\|^{2}_{L^{2}(U_{0})}\right).\label{eqn: u gradient integral estimate}
\end{eqnarray}
Because $u\in W^{1,2}_{1-\frac{n}{2}}(M) \cap C^{\infty}(\mathring{M}) \subset W^{0, 2}_{1-\frac{n}{2}-2}$ and satisfies the eigenvalue equation $Lu=\lambda u$, Proposition \ref{prop: weighted elliptic estimate} implies that $u\in W^{2,2}_{1-\frac{n}{2}}(M)$. So
\begin{eqnarray}
+\infty & > & \|u\|_{W^{2,2}_{1-\frac{n}{2}}(M)} \nonumber\\
        & \geq &  \|u\|_{W^{2,2}_{1-\frac{n}{2}}(U_{0})}\nonumber\\
        & \geq & \left( \int_{U_{0}}\left(\frac{1}{r}\right)^{2(1-\frac{n}{2}-2)+n}|\nabla^{2}u|^{2} \right)^{\frac{1}{2}}\nonumber \\
        & \geq & C \epsilon \|\nabla^{2}u\|_{L^{2}(U_{0})}.
\end{eqnarray}
This gives $\|\nabla^{2}u\|^{2}_{L^{2}(U_{0})} \leq C \epsilon^{-2} \|u\|^{2}_{W^{2,2}_{1-\frac{n}{2}}(U_{0})}$. Then by substitute this inequality into $(\ref{eqn: u gradient integral estimate})$, we obtain
\begin{equation}\label{eqn: w estimate 1}
\||\nabla u|^{2}\|^{2}_{L^{q_{1}}(U_{1})} \leq C \epsilon^{-2} \|u\|^{2}_{W^{2,2}_{1-\frac{n}{2}}(U_{0})}.
\end{equation}

Moreover, using Sobolev inequality again gives
\begin{equation}\label{eqn: w estimate 2}
\begin{split}
\|r^{-2}u^{2}\|_{L^{q_{1}}(U_{1})} 
& =   \|r^{-1}u\|^{2}_{L^{\frac{2n}{n-2}}(U_{1})} \\
& \leq  C \epsilon^{-2} \|\eta_{0}u\|^{2}_{L^{\frac{2n}{n-2}}(U_{0})} \\
& \leq  C \epsilon^{-2} \|u\|^{2}_{W^{1, 2}_{1-\frac{n}{2}}(U_{0})}.
\end{split}
\end{equation}

By combining the inequalities in $(\ref{eqn: w estimate 1})$ and $(\ref{eqn: w estimate 2})$, we obtain
\begin{equation}\label{eqn: w integral estimate}
\|w\|_{L^{q_{1}}(U_{1})} \leq C \epsilon^{-2} \|u\|^{2}_{W^{2, 2}_{1-\frac{n}{2}}(U_{0})}.
\end{equation}
Finally, by substituting the estimate in $(\ref{eqn: w integral estimate})$ into the inequality in $(\ref{eqn: iteration limit inequality})$, we finish the proof.
\end{proof}

Now by setting $\epsilon = \left( \frac{1}{2} \right)^{j}$ for $j \geq 0$ in Proposition \ref{prop: Nash-Moser iteration}, we have
\begin{equation}
\left( \frac{1}{2} \right)^{\left(\frac{n}{2}-1\right) j} \|u\|_{L^{\infty}((2^{-(j+2)}, \, 2^{-(j+1)}) \times \partial M)} 
\leq
 C \cdot \|u\|_{W^{2, 2}_{1-\frac{n}{2}}( (2^{-(j+3)}, \, 2^{-j}) \times \partial M)} \rightarrow 0,
\end{equation}
as $j \rightarrow \infty$, since $\sum\limits^{\infty}_{j=1} \|u\|_{W^{2, 2}_{1-\frac{n}{2}} ((2^{-(j+3)}, \, 2^{-j}) \times \partial M)} \leq 3 \|u\|_{W^{2, 2}_{1-\frac{n}{2}}(M)} <\infty$.

The same arguement gives an estimate for $|\nabla u|$.  We obtain 

\begin{cor}\label{cor: asymptotic of eigenfunctions}
Let $(M^n, g)$ be a $n$-dimensional compact manifold with non-isolated conical singularities as in Definition $\ref{defn: ManifoldWithNonisolatedConicalSingularities}$, and $u$ an eigenfunction of $-4\Delta_g + R_g$. Then $u$ satisfies 
\begin{equation}
|\nabla^{i}u|  = o(r^{-\frac{n}{2}+1-i})
\end{equation}
as $r\rightarrow 0$, 
for $i=0$ and $1$.
\end{cor}

\section{Lower boundedness of the infimum of $W$-functional and asymptotic behavior of the minimizer}\label{section: W-functional}

By using the estimate in Proposition \ref{prop: semiboundedness estimate}, Sobolev inequality in Proposition $\ref{prop: SobolevInequality}$, and compactness of weighted Sobolev embeddings in Proposition \ref{prop: compactness on manifolds} and Proposition \ref{prop: compact-weighted-embedding2}, one can obtain that the infimum of $W$-functional is finite and a smooth positive function realizes the infimum, provided $R_{h_{0}}>(f-1)$, via direct methods in the calculus of variations. The proof is verbatim as in the proof of Proposition 4.2 in \cite{DW18}. So we omit it here.

Moreover, the asymptotic estimate for the minimizer of the $W$-functional can be obtained by combining weighted elliptic bootstrapping and Nash-Moser iteration as we do for eigenfunctions of $-4\Delta+R$ in \S$\ref{section: asymptotic behavior of eigenfuctions}$. However, unlike the eigenfunction equation, the Euler-Lagrange equation of minimizer of $W$-functional in (\ref{eqn: Euler-Lagrange equation of W-functional}) below is nonlinear. In the rest of the section, we explain how to deal with new difficulties caused by the nonlinear term $u\ln u$ in the derivation of asymptotic estimate for the minimizer $u$ in Theorem $\ref{thm: W-functional introduction}$.

Recall that the Euler-Lagrange equation  of the minimizer $u$ of $W$-functional is given by, see e.g. the equation (4.17) in \cite{DW20},
\begin{equation}\label{eqn: Euler-Lagrange equation of W-functional}
-4\Delta_g u + R_{g}u - \frac{2}{\tau}u\ln u - \frac{n + m}{\tau}u = 0,
\end{equation}
with the constraint $\|u\|_{L^{2}(M)} = (4\pi\tau)^{\frac{n}{4}}$, where $m$ is the infimum of the $W$-functional.

Let $v=|\nabla u|^{2} + r^{-2}u^{2}$. As in step 1 in the proof of Proposition \ref{prop: Nash-Moser iteration}, one easily sees that
\begin{equation}
\Delta (r^{-2}u^{2}) \geq 2r^{-2}v - \frac{1}{2}|R|v - \frac{|n+m|}{8\tau}v - \frac{1}{\tau}r^{-2}u^{2}\ln u,
\end{equation}
and
\begin{equation}
\Delta |\nabla u|^{2} \geq -\frac{1}{2}|R|v - \frac{|n+m|}{\tau}v - Cv - C r^{-2}v - \frac{1}{\tau}|\nabla u|^{2}\ln u.
\end{equation}
So
\begin{equation}
\Delta v \geq -|R|v - Cv - Cr^{-2}v - \frac{1}{\tau}v\ln u.
\end{equation}

Then substituting $v$ into the inequality in (\ref{eqn: Nash-Moser}) gives
\begin{eqnarray}
\int | \nabla(\eta v^{\frac{q}{2}}) |^{2}
& = & \frac{q}{2(q-1)} \int v^{q} |\nabla \eta|^{2} - \frac{q^{2}}{4(q-1)} \int \eta^{2} v^{q-1} \Delta v - \frac{q-2}{2(q-1)} \int \eta v^{q} \Delta \eta \nonumber \\
& \leq & \frac{q}{2(q-1)} \int v^{q} |\nabla \eta|^{2} - \frac{q^{2}}{4(q-1)}\int \eta^{2}v^{q}|R| \\
&   & + C\frac{q^{2}}{4(q-1)} \int \eta^{2} v^{q} + C\frac{q^{2}}{4(q-1)} \int \eta^{2} r^{-2} v^{q} \nonumber \\ 
&   & + \frac{q^{2}}{8(q-1)\tau} \int \eta^{2}v^{2}\ln u  - \frac{q-2}{2(q-1)}\int \eta^{2}v^{q}\Delta\eta, \nonumber
\end{eqnarray}
for any nonnegative $\eta \in C^{\infty}_{0}(\mathring{M})$ and $q>1$. Then the only difference from the eigenfunction estimate in Proposition \ref{prop: Nash-Moser iteration} is the term involving the integral of $v^{q}\ln u$. So we briefly describe how to deal with this term.

For a sufficiently small $\gamma < 1$ (with $\frac{2}{\gamma}>\frac{n}{2}$), there exists a constant $a$ such that
\begin{eqnarray}
\int \eta^{2}v^{q}\ln u
&   =  & \int_{\{0< u \leq e\}} \eta^{2}v^{q}\ln u + \int_{\{u>e\}} \eta^{2}v^{q}\ln u \\
& \leq & \int_{\{0< u \leq e\}} \eta^{2}v^{q} + a\int_{\{u>e\}} \eta^{2}v^{q} u^{\gamma} \\
& \leq & \int \eta^{2}v^{q} + a\int \eta^{2}v^{q} u^{\gamma}.
\end{eqnarray}

Then by H\"older inequality and Young's inequality (e.g. as on p. 20 in \cite{DWZ18}),
\begin{eqnarray}
& & a \frac{q^{2}}{8(q-1)\tau} \int \eta^{2}v^{q} u^{\gamma} \nonumber \\
& \leq & a \frac{q^{2}}{8(q-1)\tau} \left( \int u^{2} \right)^{\frac{\gamma}{2}} \cdot \left( \int \left( \eta^{2}v^{q} \right)^{\frac{2}{2-\gamma}} \right)^{\frac{2-\gamma}{2}} \\
& \leq & a (4\pi\tau)^{\frac{n\gamma}{4}} \left[ \delta \left( \int (\eta^{2}v^{q})^{\mu} \right)^{\frac{1}{\mu}} + \delta^{-\frac{n\gamma}{4-n\gamma}} \cdot \left( \frac{q^{2}}{8(q-1)\tau} \right)^{\frac{4-2n\gamma}{4-n\gamma}} \left( \int \eta^{2} v^{q} \right) \right],
\end{eqnarray}
for any small $\delta>0$. We may take $\delta$ so that $C\cdot a \cdot (4\pi\tau)^{\frac{n\gamma}{4}} \cdot \delta = \frac{1}{2}$, where $C$ is the $L^{2}$-Sobolev constant in Proposition \ref{prop: SobolevInequality}. 

Another remark is about the derivation of $u\in W^{0, 2}_{1-\frac{n}{2}}(M)$ by using Proposition \ref{prop: weighted elliptic estimate}. For this one needs to show that $u\ln u\in W^{0, 2}_{1-\frac{n}{2}-2}$, since $Lu=\frac{n+m}{\tau}u + \frac{2}{\tau}u\ln u$. This follows from $u\in W^{1, 2}_{1-\frac{n}{2}}(M) \subset W^{0, \frac{2n}{n-2}}_{1-\frac{n}{2}}(M) = L^{\frac{2n}{n-2}}(M)$ and the fact that there exits a constant $a(n)$ such that $|u\ln u|\leq a(n) + |u|^{\frac{n}{n-2}}$. Indeed, $\int_{M} \chi^{2\left(-\frac{n}{2}\right)+n}|u\ln u|^{2} \leq a(n)^{2} {\rm Vol}(M) + \int_{M} |u|^{\frac{2n}{n-2}} < \infty$. So $u\ln u \in W^{0, 2}_{-\frac{n}{2}} \subset W^{0, 2}_{1-\frac{n}{2}-2}$.

Then the Nash-Moser iteration as in the proof of Proposition \ref{prop: Nash-Moser iteration} gives the asymptotic behavior in Theorem \ref{thm: W-functional introduction}.


\section{Horn singularity}\label{section: horn}
In this section, we study the spectrum and eigenfunctions of the Schr\"odinger operators $-\Delta_{g} + c R_{g} (c\geq 0)$ on compact manifolds with a kind of more general singularities, whose model is $((0, 1)\times F^{f}, g_{\gamma}=dr^{2} + r^{2\gamma}\hat{g})$. Here $\gamma>0$, and $\hat{g}$ is a Riemannian metric on $F^{f}$. This is a {\em $r^{\gamma}$-horn} in \cite{Che80}. If in particular $\gamma=1$, then this is a model cone singularity that we studied in preceding sections. So now we assume $\gamma>1$, and this kind of singularities appear in the study of singular projective varieties.

The model singular metric $g_{\gamma}=dr^{2} + r^{\gamma}\hat{g}$ has Laplacian $\Delta_{g_{\alpha}}$ and scalar curvature $R_{g}$ as
\begin{eqnarray*}
\Delta_{g_{\gamma}} & = & \frac{\partial^{2}}{\partial r^{2}} + \frac{\gamma(n-1)}{r}\frac{\partial}{\partial r} + \frac{1}{r^{2\gamma}}\Delta_{\hat{g}}, \\
R_{g_{\gamma}} & = & \frac{R_{\hat{g}}}{r^{2\alpha}} - \frac{\gamma(n-1)(n\gamma-2)}{r^{2}},
\end{eqnarray*}
where $n=f+1$.

\subsection{manifolds with $r^{\gamma}$-horn singularities}

Similar as Definition \ref{defn: ManifoldWithNonisolatedConicalSingularities}, we define
\begin{defn}\label{defn: ManifoldWith-alpha-ConicalSingularities}
We say that $(M^{n},g)$ is a compact Riemannian manifold with $r^{\gamma}$-horn singularities ($\gamma>1$), if
\begin{enumerate}
\item $M^{n}$ is a compact $n$-dimensional manifold with boundary $\partial M$.
\item The boundary $\partial M$ endowed with a Riemannian metric $g_{\partial M}$ is the total space of a Riemannian submersion
      \begin{equation}\label{eqn: Riemannian-submersion}
      \pi: (\partial M, g_{\partial M}) \longrightarrow (B^{b}, \check{g})
      \end{equation}
      with a $b$-dimensional closed Riemannian manifold $(B^{b}, \check{g})$ as base and a $f$-dimensional closed manifold $F^{f}$ as a typical fiber, in particular, $b+f=n-1$, and
      \begin{equation}\label{eqn: metric-on-total-space}
      g_{\partial M}=\hat{g}+\pi^{*}\check{g},
      \end{equation}
      where $\hat{g}$ denotes metrics on fibers induced by $g_{\partial M}$;
\item Moreover, there exists a collar neighborhood $U=(0, 1)\times \partial M$ of the boundary $\partial M$ in $M$ with the radial coordinate $r$ $(r=0$ corresponds to the boundary $\partial M)$, and the metric $g$ restricted on $U$ satisfies $g=g_{\gamma}+h$ as $r\rightarrow0$, where the Riemannian metric $g_{\gamma}$ on $U$ given by
    \begin{equation*}
    g_{\gamma}=dr^{2}+r^{2\gamma}\hat{g}+\pi^{*}\check{g},
    \end{equation*}
    and $h$ satisfies
    \begin{equation*}
    |r^{\gamma k} \nabla^{k}_{g_{\gamma}} h|_{g_{\gamma}} = O(r^{\alpha}), \quad \text{as} \ \ r\rightarrow 0,
    \end{equation*}
    for some $\alpha>0$ and all $k \in \mathbb{N}$.
\end{enumerate}
\end{defn}

\subsection{$\alpha$-weighted Sobolev spaces}

Let $(M^{n}, g)$ be a compact Riemannian manifold with $\alpha$-horn singularities as defined in Definition \ref{defn: ManifoldWith-alpha-ConicalSingularities}.  For each $p\geq1$, $k\in\mathbb{N}$ and $\delta\in\mathbb{R}$, the {\em $\alpha$-weighted Sobolev space} $W^{k, p}_{\delta, \alpha}(M)$ is the completion of $C^{\infty}_{0}(\mathring{M})$ with respect to the {\em $\gamma$-weighted Sobolev norm}
\begin{equation}\label{eqn: WSN-horn}
\|u\|_{W^{k, p}_{\delta, \gamma}(M)} = \left( \int_{M}\left(\sum^{k}_{i=0}\chi^{p(\delta-\gamma i)+n\gamma}|\nabla^{i}u|^{p}_{g}\right)d\vol_{g} \right)^{\frac{1}{p}},
\end{equation}
where $\nabla^{i}u$ denotes the $i$-times covariant derivative of the function $u$, and $\chi\in C^{\infty}(\mathring{M})$ is a positive weight function satisfying
\begin{equation}\label{eqn: WeightFunction}
\chi(x)=
\begin{cases} 1, & \text{if} \ \ x\in M\setminus U, \\ \frac{1}{r}, & \text{if} \ \  r=dist(x,\partial M)<\frac{1}{10},
\end{cases}
\end{equation}
and $0<(\chi(x))^{-1}\leq 1$ for all $x \in \mathring{M}$. Recall that $U=(0, 1)\times \partial M\subset M$ is the collar neighborhood of the boundary $\partial M$.

On a manifold $(M^{n}, g)$ with isolated $r^{\gamma}$-horn singularities, i.e. the base manifold $B$ in Definition \ref{defn: ManifoldWith-alpha-ConicalSingularities} is a point, one can derive the following weighted Sobolev inequalities.
\begin{prop}\label{prop: isolated alpha weighted Sobolev inequality}
Let $(M^{n}, g)$ be a manifold with an isolated $r^{\gamma}$-horn singularity. For any $\delta \in \mathbb{R}$, $1 \leq p < n$, $0 \leq l \leq k$, and $q$ with $\frac{1}{q} = \frac{1}{p} - \frac{k-l}{n} > 0$, there exists a constant $C$, such that 
\begin{equation*}
\|u\|_{W^{l, q}_{\delta, \gamma}(M, g)} \leq C \|u\|_{W^{k, p}_{\delta, \gamma}(M, g)}
\end{equation*}
holds for all $u \in C^{\infty}_{0}(\mathring{M})$.

Consequently, there is a continuous embedding
\begin{equation*}
W^{k, p}_{\delta, \gamma}(M, g) \subset W^{l, q}_{\delta, \gamma}(M, g).
\end{equation*}
\end{prop}

\begin{proof}
The key of establishment of the $\gamma$-weighted Sobolev embedding is to derive it on a collar neighborhood $(0, 1)\times \partial M$ of $\partial M$ with the model metric $g_{\gamma}=dr^{2}+r^{2\gamma}\hat{g}$.

For any $u\in C^{\infty}_{0}((0,1)\times \partial M)$,
\begin{eqnarray*}
   \|u\|_{W^{0, \frac{np}{n-p}}_{\delta, \gamma}((0,1)\times\partial M, g_{\gamma})}
   & = & \left(\int_{(0,1)\times \partial M}(u(r,x))^{\frac{pn}{n-p}}\left(\frac{1}{r}\right)^{\frac{pn}{n-p}\delta+n\gamma} d\vol_{g_{\gamma}}\right)^{\frac{n-p}{np}} \\
   & = & \left(\int_{\partial M}\int^{1}_{0}\left(u(r,x)\right)^{\frac{pn}{n-p}}\left(\frac{1}{r}\right)^{\frac{pn}{n-p}\delta+n\gamma}r^{(n-1)\gamma}dr d\vol_{\hat{g}}\right)^{\frac{n-p}{np}} \\
   & = & \left(\int_{\partial M}\int^{1}_{0}\left(u(r,x)r^{-\delta}\right)^{\frac{pn}{n-p}}r^{-\gamma}dr d\vol_{\hat{g}}\right)^{\frac{n-p}{np}}
\end{eqnarray*}
In the last integral, we make a change of variable $r=s^{\frac{1}{1-\gamma}}$, and set $\tilde{u}(s, x)=u(s^{\frac{1}{1-\gamma}}, x)$. Then $\tilde{u}\in C^{\infty}_{0}((1, \infty)\times\partial M)$, and one has
\begin{eqnarray*}
\|u\|_{W^{0, \frac{np}{n-p}}_{\delta, \gamma}((0,1)\times\partial M, g_{\gamma})}
& = & \left(\int_{\partial M}\int^{\infty}_{1}\left(\tilde{u}(s, x)s^{-\frac{\delta}{1-\gamma}}\right)^{\frac{np}{n-p}}\frac{1}{\gamma-1}ds d\vol_{\hat{g}}\right)^{\frac{n-p}{np}} \\
& = & \left(\frac{1}{\gamma-1}\right)^{\frac{n-p}{np}} \|\tilde{u}(s,x)s^{-\frac{\delta}{1-\gamma}}\|_{L^{\frac{np}{n-p}}((1, \infty)\times\partial M, g_{0})}.
\end{eqnarray*}
Here $g_{0}=ds^{2}+\hat{g}$ is the product metric on $(1, \infty)\times \partial M$. Then the Sobolev inequality on the cylinder $((1, \infty)\times \partial M, g_{0})$ produces
\begin{eqnarray*}
&   &  \|u\|_{W^{0, \frac{np}{n-p}}_{\delta, \gamma}((0,1)\times\partial M, g_{\gamma})}  \\
& \leq & \left(\frac{1}{\gamma-1}\right)^{\frac{n-p}{np}} C \|\tilde{u}(s,x)s^{-\frac{\delta}{1-\gamma}}\|_{W^{1, p}((1, \infty)\times\partial M, g_{0})} \\
& = & \left(\frac{1}{\gamma-1}\right)^{\frac{n-p}{np}} C \left(\int_{\partial M}\int^{\infty}_{1}\left(|\nabla(\tilde{u}(s,x)s^{-\frac{\delta}{1-\gamma}})|^{p}_{g_{0}} + |\tilde{u}(s, x)|^{p}s^{-\frac{\delta p}{1-\gamma}}\right)ds d\vol_{\hat{g}}\right)^{\frac{1}{p}} \\
& \leq & \left(\frac{1}{\gamma-1}\right)^{\frac{n-p}{np}} C \Bigg( \int_{\partial M}\int^{\infty}_{1} \Bigg( \left(\frac{\delta}{\gamma-1}\right)^{p}s^{\frac{p\delta}{\gamma-1}-p}|\tilde{u}(s, x)|^{p} + s^{-\frac{\delta p}{1-\gamma}}\left|\frac{\partial}{\partial s}\tilde{u}(s, x)\right|^{p} \\
&   & \qquad + s^{-\frac{p\delta}{1-\gamma}}|\nabla^{\hat{g}}\tilde{u}(s,x)|^{p}_{\hat{g}} + s^{-\frac{p\delta}{1-\gamma}}|\tilde{u}(s, x)|^{p} \Bigg) ds d\vol_{\hat{g}} \Bigg)^{\frac{1}{p}}
\end{eqnarray*}
Now change the variable back, i.e. let $s=r^{1-\gamma}$, and note $\frac{\partial \tilde{u}}{\partial s}(s, x)=\frac{\partial u}{\partial r}(r,x)\left(\frac{1}{1-\gamma}\right)r^{\gamma}$.One has
\begin{eqnarray*}
&    &  \|u\|_{W^{0, \frac{np}{n-p}}_{\delta, \gamma}((0,1)\times\partial M, g_{\gamma})} \\
& \leq & \left(\frac{1}{\gamma-1}\right)^{\frac{n-p}{np}} C \Bigg( \int_{\partial M}\int^{1}_{0} \Bigg( \left(\frac{\delta}{\gamma-1}\right)^{p}r^{-p\delta+p(\gamma-1)}|u(r, x)|^{p} \\
&   & \qquad  + r^{-p\delta+p\gamma}\left|\frac{1}{\gamma-1}\right|^{p}\left|\frac{\partial u}{\partial r}(r,x)\right|^{p}
   + r^{-p\delta}|\nabla^{\hat{g}}u(r, x)|^{p}_{\hat{g}} \\
&    & \qquad  + r^{-p\delta}|u(r, x)|^{p} \Bigg) (\gamma-1) r^{-\alpha}dr d\vol_{\hat{g}} \Bigg)^{\frac{1}{p}} \\
& \leq & C \Bigg( \int_{\partial M}\int^{1}_{0} \Bigg[ \left(\frac{1}{r}\right)^{p(\delta-\gamma)+n\gamma}\left(\left|\frac{\partial u}{\partial r}(r, x)\right|^{p} + \left(\frac{1}{r}\right)^{p\gamma}|\nabla^{\hat{g}}u(r, x)|^{p}_{\hat{g}}\right) \\
&    & \qquad  + \left(\frac{1}{r}\right)^{p\delta + n\gamma}|u(r, x)|^{p} \Bigg]  r^{(n-1)\gamma}dr d\vol_{\hat{g}} \Bigg)^{\frac{1}{p}} \\
& \leq & C \Bigg( \int_{(0, 1)\times \partial M} \Bigg[ \left(\frac{1}{r}\right)^{p(\delta-\gamma)+n\gamma}|\nabla^{g_{\gamma}} u|^{p}_{g_{\gamma}}
   + \left(\frac{1}{r}\right)^{p\delta+n\gamma}|u(r, x)|^{p} \Bigg]  d\vol_{g_{\gamma}} \Bigg)^{\frac{1}{p}} \\
& = & C \|u\|_{W^{1, p}_{\delta, \gamma}((0,1)\times\partial M, g_{\gamma})}.
\end{eqnarray*}

Now combining this with the usual Sobolev inequality in the interior part of the manifold $M$, one can obtain that for any given $\delta\in \mathbb{R}$ and $p\geq1$, there exists a constant $C$ such that
\begin{equation}\label{eqn: basic alpha weighted Sobolev inequality}
\|u\|_{W^{0, \frac{np}{n-p}}_{\delta, \gamma}(M, g)} \leq C \|u\|_{W^{1, p}_{\delta, \gamma}(M, g)}
\end{equation}
for all $u\in C^{\infty}_{0}(\mathring{M})$.

Then by the inequality (\ref{eqn: basic alpha weighted Sobolev inequality}), Kato's inequality and the definition of the $\gamma$-weighted Sobolev norm, one has
\begin{equation*}
\|\nabla u\|_{W^{0, \frac{np}{n-p}}_{\delta, \gamma}(M, g)} \leq C \|\nabla u\|_{W^{1, p}_{\delta, \gamma}(M, g)} \leq C \|u\|_{W^{2, p}_{\delta+\gamma, \gamma}(M, g)}.
\end{equation*}
Moreover,
\begin{equation*}
\|u\|_{W^{0, \frac{np}{n-p}}_{\delta+\gamma, \gamma}(M, g)} \leq C \|u\|_{W^{1, p}_{\delta+\gamma, \gamma}(M, g)} \leq C \|u\|_{W^{2, p}_{\delta+\gamma, \gamma}(M, g)}.
\end{equation*}
Thus
\begin{equation*}
\|u\|_{W^{1, \frac{np}{n-p}}_{\delta+\gamma, \gamma}(M, g)}\leq C \|u\|_{W^{2, p}_{\delta+\gamma, \gamma}(M, g)}.
\end{equation*}
Because in this inequality $\delta$ is arbitrary, one obtains that for any given $\delta\in\mathbb{R}$ and $p\geq1$ there exits a constant $C$ such that
\begin{equation*}
\|u\|_{W^{1, \frac{np}{n-p}}_{\delta, \gamma}(M, g)}\leq C \|u\|_{W^{2, p}_{\delta, \gamma}(M, g)}
\end{equation*}
for all $u\in C^{\infty}_{0}(\mathring{M})$.

Then one can inductively prove the $\gamma$-weighted Sobolev inequality and complete the proof of the proposition.
\end{proof}

In Proposition \ref{prop: weighted-Sobolev-inequality}, we established the weighted Sobolev inequality on manifolds with conical singularities by using Sobolev inequality on these manifolds. A Hardy's inequality is the key ingredient in the derivation of the Sobolev inequality on manifolds with conical singularities. This does not work for manifolds with $r^{\gamma}$-horn singularities any more. 
On the other hand, it is interesting to note that, for projective varieties, 
the work of P. Li and G. Tian \cite{LT95} implies the Sobolev inequality via their upper bound for the heat kernel.

In the isolated $r^{\gamma}$-horn singularities case, we have obtained the $\gamma$-weighted Sobolev inequalities. In the non-isolated singularities case, if we assume that the Sobolev inequalities holds, e.g. on projective varieties with $r^{\gamma}$-horn singularities, then similarly as in the proof of Proposition \ref{prop: weighted-Sobolev-inequality}, we can derive $\gamma$-weighted Sobolev inequalities.

\begin{prop}\label{prop: alpha-weighted-Sobolev-inequality}
Let $(M^{n}, g)$ be a compact Riemannian manifold with non-isolated $r^{\gamma}$-horn singularities, on which the Sobolev inequalities hold. For any $\delta \in \mathbb{R}$, $1 \leq p < n$, $0 \leq l \leq k$, and $q$ with $\frac{1}{q} = \frac{1}{p} - \frac{k-l}{n} > 0$, there exists a constant $C$, such that 
\begin{equation*}
\|u\|_{W^{l, q}_{\delta, \gamma}} \leq C \|u\|_{W^{k, p}_{\delta, \gamma}}.
\end{equation*}
holds for all $u \in C^{\infty}_{0}(\mathring{M})$.

Consequently, there is a continuous embedding
\begin{equation*}
W^{k, p}_{\delta, \gamma}(M) \subset W^{l, q}_{\delta, \gamma}(M).
\end{equation*}
\end{prop}

Moreover, we have the following compact $\alpha$-weighted Sobolev embedding, which will play an important role in the study of the spectrum of the Sch\"odinger operator $\Delta + c R$.

\begin{prop}\label{prop: compact-alpha-weighted embedding}
Let $(M, g)$ be a compact Riemannian manifold with $r^{\gamma}$-horn singularities. The continuous embedding $W^{1, 2}_{\gamma-\frac{n\gamma}{2}, \gamma}(M, g) \hookrightarrow L^{2}(M, g)$ is compact.
\end{prop}

\begin{remark}
{\rm
Note that the embedding $W^{1, 2}_{\gamma-\frac{n\gamma}{2}, \gamma}(M, g) \hookrightarrow L^{2}(M, g)$ does not follow from Proposition \ref{prop: alpha-weighted-Sobolev-inequality}, since in Proposition \ref{prop: compact-alpha-weighted embedding} we do not assume that the Sobolev inequalities hold on the manifold with $r^{\gamma}$-horn singularities. In the following proof, we will obatin both the embedding $W^{1, 2}_{\gamma-\frac{n\gamma}{2}, \gamma}(M, g) \hookrightarrow L^{2}(M, g)$ and its compactness simultaneously.
}
\end{remark}

\begin{proof}
Clearly, it suffices to obtain the compact embedding on a collar neighborhood $U=(0, 1)\times \partial M$ of $\partial M$ with the model metric $g_{\alpha}=dr^{2}+r^{2\gamma}\hat{g}+\pi^{*}\check{g}$. Similarly as in the proof of Lemma \ref{lem: compactness on cones}, we relate the function spaces on the model collar neighborhood $(U, g_{\gamma})$ to that on a finite cylinder $(C(\partial M) = (0, 1) \times \partial M, g_{C}=dr^{2}+g_{\partial M})$.

The mapping
\begin{eqnarray*}
L^{2}(U, g_{\gamma}) & \rightarrow & L^{2}(C(\partial M), g_{C}) \\
   u & \mapsto & \tilde{u}=r^{\frac{\gamma f}{2}}u
\end{eqnarray*}
is unitary. Similarly as the proof of the inequality in $(\ref{eqn: cone cylinder inequality})$ in Lemma $\ref{lem: compactness on cones}$, we can show that
\begin{equation}\label{eqn: horn cylinder inequality}
\|u\|_{W^{1,2}_{\gamma-\frac{n\gamma}{2}, \gamma}(U, g_{\gamma})} \geq \|u\|_{W^{1, 2}(C(\partial M), g_{C})}.
\end{equation}
Then by combining the inequality in $(\ref{eqn: horn cylinder inequality})$ with the classic Rellich lemma on a finite cylinder, we complete the proof.
\end{proof}

Finally, we derive the following compact $\alpha$-weighted Sobolev embedding that will be used in the study of $W$-functional on manifolds with $r^{\gamma}$-horn singularities. 
\begin{prop}
Let $(M^{n}, g)$ be a compact Riemannian manifold with $r^{\gamma}$-horn singularities. There is a compact embedding $W^{1,1}_{\gamma -n\gamma, \gamma}(M, g)\subset L^{1}(M, g)$.
\end{prop}
\begin{proof}
Again, it suffices to derive the compact embedding on a small collar neighborhood of the singular set $\partial M$ with the model metric $g_{\gamma}$. Because $\gamma > 1$, there exists $\epsilon>0$ such that for all $r<\epsilon$
\begin{equation*}
r^{-\gamma} - \gamma f r^{-1} >1.
\end{equation*}

The map
\begin{eqnarray*}
L^{1}((0, \epsilon)\times\partial M, g_{\gamma}) & \rightarrow & L^{1}((0, \epsilon)\times\partial M, g_{Cyl}) \\
  u & \mapsto & \tilde{u}=r^{\gamma f}u
\end{eqnarray*}
is unitary. Here $g_{Cyl}$ is the product metric $dr^{2}+g_{\partial M}$.

On the other hand,
\begin{eqnarray*}
\|u\|_{W^{1, 1}_{\gamma - n\gamma, \gamma}((0, \epsilon)\times\partial M, g_{\gamma})}
& = & \int_{\partial M}\int^{\epsilon}_{0} \left(|\nabla u|_{g_{\gamma}} + \frac{1}{r^{\gamma}}|u|\right)r^{\gamma f}dr d\vol_{g_{\partial M}} \\
& \geq &  \int_{\partial M}\int^{\epsilon}_{0} \left( |\partial_{r}u| + |\nabla u|_{g_{\partial M}} + \frac{1}{r^{\gamma}}|u| \right)r^{\gamma f}dr d\vol_{g_{\partial M}} \\
& \geq & \int_{\partial M}\int^{\epsilon}_{0} \left( |\partial_{r}\tilde{u}| + |\nabla \tilde{u}|_{g_{\partial M}} + (r^{-\gamma}-\gamma f r^{-1})|\tilde{u}| \right)dr d\vol_{g_{\partial M}} \\
& \geq & \|\tilde{u}\|_{W^{1,1}((0, \epsilon)\times\partial M, g_{Cyl})}.
\end{eqnarray*}
Then Rellich lemma on the cylinder $((0, \epsilon)\times\partial M, g_{Cyl})$ completes the proof.
\end{proof}

\subsection{Spectrum of $-\Delta_{g} + c R_{g}$ on manifolds with $r^{\gamma}$-horn singularities}

\begin{thm}\label{thm: spectrum property horn singularity}
Let $(M, g)$ be a compact Riemannian manifold with $r^{\gamma}$-horn singularities as in Definition $\ref{defn: ManifoldWith-alpha-ConicalSingularities}$ with $\gamma > 1$. If $\min\limits_{\partial M}\{R_{\hat{g}}\}>0$, then the Friedrichs extension of the Schr\"odinger operator $-\Delta_{g} + c R_{g} (c>0)$ with the domain $C^{\infty}_{0}(\mathring{M})$ has the spectrum consisting of discrete eigenvalues with finite multiplicity $-\infty < \lambda_{1}<\lambda_{2}\leq\lambda_{3}\leq \cdots$, and the corresponding eigenfunctions are smooth and form a basis of $L^{2}(M, g)$.
\end{thm}

\begin{proof}
The scalar curvature of the metric $g$ restricted on a small collar neighborhood of $\partial M$ is
\begin{equation}\label{eqn: alpha conical scalar curvature formula}
R_{g}= \frac{R_{\hat{g}}}{r^{2\gamma}}  - \frac{f \gamma (n\gamma+2)}{r^{2}} + R_{\check{g}}\circ\pi-|A|^{2}-|T|^{2}-|N|^{2}-\check{\delta}N + o(r^{-2\gamma + \alpha}).
\end{equation}
Thus for sufficiently small $\epsilon>0$, on $(0, \epsilon)\times \partial M$, one has
\begin{equation*}
R_{g}\geq \frac{\delta}{r^{2\gamma}},
\end{equation*}
for some $\delta>0$.
Consequently, for any $u\in C^{\infty}_{0}((0, \epsilon)\times \partial M)$,
\begin{equation*}
  \begin{split}
     \langle (-\Delta_{g}+cR_{g})u, u \rangle_{L^{2}(M, g)}
     & = \int_{M}(-\Delta_{g} u + cR_{g}u)u d\vol_{g} \\
     & = \int_{M} |\nabla u|^{2}_{g} + cR_{g}u^{2} d\vol_{g} \\
     & \geq \int_{M} |\nabla u|^{2}_{g} + \frac{c\delta}{r^{2\gamma}} u^{2} d\vol_{g} \\
     & \geq \min\{1, c\delta \} \|u\|_{W^{1, 2}_{\gamma-\frac{n\gamma}{2}, \gamma}(M, g)}.
  \end{split}
\end{equation*}
Then by choosing a sufficiently large constant $A$ to control the scalar curvature $R_{g}$ on the interior part $M\setminus (0, \epsilon/2)\times \partial M$, one can obtain that there exists a constant $C$ such that
\begin{equation*}
\langle (-\Delta_{g}+cR_{g}+A)u, u \rangle_{L^{2}(M, g)} \geq C \|u\|_{W^{1, 2}_{\gamma - \frac{n\gamma}{2}, \gamma}(M, g)},
\end{equation*}
for any $u\in C^{\infty}_{0}(\mathring{M})$. So the operator $-\Delta_{g}+cR_{g}+A$ with domain $C^{\infty}_{0}(\mathring{M})$ has a Friedrichs self-adjoint extension with the domain contained in $W^{1, 2}_{\gamma -\frac{n\gamma}{2}, \gamma}(M ,g)$ and the image equal to $L^{2}(M, g)$. Now the compactness of the embedding in Proposition \ref{prop: compact-alpha-weighted embedding} implies that the inverse of $-\Delta_{g}+cR_{g}+A$ is a compact operator from $L^{2}(M, g)$ to itself. Then the spectrum theory of self-adjoint compact operators completes the proof.
\end{proof}

\subsection{Perelman's $\lambda$-functional and $W$-functional on manifolds with $r^{\gamma}$-horn singularities}

By Theorem \ref{thm: spectrum property horn singularity}, one can define the $\lambda$-functional on a manifold $(M^{n}, g)$ with $r^{\gamma}$-horn singularities as the first eigenvalue of the Friedrichs extension of the operator $-4\Delta_{g}+R_{g}$, or equivalently
\begin{equation*}
\lambda(g):=\inf\left\{\int_{M}\left(4|\nabla u|^{2}_{g} + R_{g}u^{2}\right)d\vol_{g} \mid u\in W^{1,2}_{\gamma-\frac{n\gamma}{2}}, \quad \|u\|_{L^{2}(M, g)}=1\right\} > -\infty,
\end{equation*}
provided $\min\limits_{\partial M}\{R_{\hat{g}}\}>0$.

For the $W$-functional, we can use $\gamma$-weighted Sobolev inequality in Proposition \ref{prop: alpha-weighted-Sobolev-inequality} to control $\int_{M}u^{2}\ln u d\vol_{g}$ term as follows. For a fixed $\tau>0$ and any $u\in W^{1, 2}_{\gamma - \frac{n\gamma}{2}}(M)$ with $\|u\|_{L^{2}}=(4\pi\tau)^{\frac{n}{2}}$, one has 
\begin{eqnarray*}
\int_{M} u^{2}\ln u d\vol_{g} 
& \leq & c(n) \int_{M} u^{2+\frac{2}{n}} d\vol_{g} \\
& \leq & c(n) \epsilon \int_{M}u^{2+\frac{4}{n}} + c(n)\frac{1}{\epsilon}\int_{M}u^{2}d\vol_{g} \\
& \leq & c(n) \epsilon \left(\int_{M}u^{\frac{2n}{n-2}} d\vol_{g}\right)^{\frac{n-2}{n}}\cdot \left(\int_{M}u^{2}d\vol_{g}\right)^{\frac{2}{n}} + c(n)\frac{1}{\epsilon}(4\pi\tau)^{n} \\
& = & c(n)\epsilon(4\pi\tau)^{2}\left(\int_{M}u^{\frac{2n}{n-2}} d\vol_{g}\right)^{\frac{n-2}{n}} + c(n)\frac{1}{\epsilon}(4\pi\tau)^{n} \\
& \leq & c(n)\epsilon (4\pi\tau)^{2}C\int_{M}\left(|\nabla u|^{2}_{g}+\chi^{2\gamma}u^{2} d\vol_{g}\right)d\vol_{g} + c(n)\frac{1}{\epsilon}(4\pi\tau)^{n}.
\end{eqnarray*}
Here $c(n)$ is a constant depending only on $n$. Now we can choose $\epsilon$ sufficiently small so that 
\begin{equation*}
2c(n)\epsilon (4\pi\tau)^{2}C < \tau\min\limits_{\partial M}\{R_{\hat{g}}\}.
\end{equation*}
Thus one has
\begin{equation*}
\inf\left\{\int_{M}[\tau(4|\nabla u|_{g}+R_{g}u^{2})]-2u^{2}\ln u d\vol_{g} \mid u\in W^{1, 2}_{\gamma-\frac{n\gamma}{2}, \gamma}(M), \ \ \|u\|_{L^{2}(M)}=(4\pi\tau)^{\frac{n}{2}}\right\} > -\infty,
\end{equation*}
provided $\min\limits_{\partial M}R_{\hat{g}}>0$. 

Hence on a compact manifold with $r^{\gamma}$-horn singularities with $\min\limits_{\partial M} R_{\hat{g}}>0$, one has
\begin{equation*}
\inf\left\{W(g, \tau, u) \mid u\in W^{1, 2}_{\gamma-\frac{n\gamma}{2}, \gamma}(M), \ \ \|u\|_{L^{2}(M)}=(4\pi\tau)^{\frac{n}{2}}\right\} > -\infty.
\end{equation*}

\subsection{Asymptotic of eigenfunctions of $-\Delta_{g} + c R_{g}$ and the minizer of the $W$-functional on manifolds with $r^{\gamma}$-horn singularities}
Similar as the derivation of weighted elliptic estimate inequality in the conical singularities case in Proposition \ref{prop: weighted elliptic estimate}, one can obtain the $\gamma$-weighted elliptic estimate inequality. Then combining this with $\gamma$-weighted Sobolev inequality in Proposition \ref{prop: alpha-weighted-Sobolev-inequality}, Nash-Moser iteration give the following asymptotic estimate for eigenfunctions of of $-\Delta_{g} + c R_{g}$ and the minizer of the $W$-functional on manifolds with $r^{\gamma}$-horn singularities.   

\begin{prop}
Let $u$ be an eigenfunction of $-\Delta_{g} + c R_{g}$ or the minimizer of the $W$-functional on manifolds with $r^{\gamma}$-horn singularities. Then one has
\begin{equation*}
    |\nabla^{i}u| = o(r^{-\frac{n-2}{2}\gamma - i\gamma})
\end{equation*}
as $r\rightarrow 0$, i.e. approaching to the singular set, for $i=0, 1$.
\end{prop}


\section{Partial asymptotic expansion of eigenfunctions on manifolds with isolated conical singularities}\label{section: partial asymptotic expansion isolated singularities}

In  this section, we derive a partial asymptotic expansion for the eigenfunctions of the Schr\"odinger operator $L:=-4\Delta+R$ on manifolds with isolated conical singularities in Theorem $\ref{prop: partial asymptotic expansion}$. The b-calculus theory of Melrose is the main tool in this derivation. So in \S$\ref{subsect: b-calculus}$ we recall some basics of the b-calculus theory.


\subsection{Basics of b-calculus of Melrose}\label{subsect: b-calculus}
In this subsection, we briefly recall some basic notions and results in b-calculus theory of Melrose that we need in this paper, and we refer to Melrose's book, \cite{Mel93}, for details. The b-calculus theory is very useful in studies of elliptic operators on various non-compact manifolds, for example, complete manifolds with certain ends (e.g. Example $\ref{ex: cylinder}$), and manifolds with conical singularities (Example $\ref{ex: cone}$). 

Let $M^n$ be  a $n$-dimensional manifold with boundary $\partial M$. Let $\left\{ y_1, \cdots, y_{n-1} \right\}$ be a local coordinate system on $\partial M$. Combining this with a defining function $x$ of the boundary, that is $x: \overline{M} \rightarrow [0, +\infty)$ smooth with $x^{-1}(0) = \partial M$ and $\nabla x\not=0$ on $\partial M$, we obtain a local coordinate system on $M$. 

The fundamental object in the b-calculus theory is the space of vector fields tangent to  the boundary at each point on the boundary. That is denoted by $\mathcal{V}_b$ and can be expressed as
\begin{equation}
\mathcal{V}_b = \left\{ X \in C^{\infty}(M, TM) \ \  \mid  \ \ X_{\upharpoonright{\partial M}} \in C^{\infty}(\partial M, T\partial M)  \right\},
\end{equation}
where $TM$ and $T\partial M$ are tangent bundles of $M$ and $\partial M$, respectively, and $C^{\infty}(M, TM)$ and $C^{\infty}(\partial M, T\partial M)$ are their smooth sections.  Then a vector field $X \in \mathcal{V}_b$ can be expressed locally as:
\begin{equation}
X = X_0 x \partial_x + \sum^{n-1}_{i=1} X_i \partial_{y_i}, \quad \text{where} \ \ X_0, X_i \in C^{\infty}(\overline{M}).
\end{equation}
An important point here is that $\mathcal{V}_b$ can be realized as the space of all smooth sections of a vector bundle, which is denoted by $^bTM$ and called {\em $b$-tangent bundle}, i.e.
\begin{equation}
\mathcal{V}_b = C^\infty(M, {^bT}M).
\end{equation}
Then the dual bundle of $^bTM$ is denoted by $^bT^*M$, whose smooth sections $\omega \in C^\infty(M, {^bT}^*M)$ can be expressed locally as:
\begin{equation}
\omega = \omega_0 \frac{dx}{x} + \sum^{n-1}_{i=1} \omega_i dy_i, \quad \text{where} \ \ \omega_0, \omega_i \in C^{\infty}(\overline{M}).
\end{equation}

A  {\em $b$-differential operator} of order $k$ is defined to be a linear map $P: C^{\infty}(M) \to C^\infty(M)$ given by a finite sum of up to $k$-fold products of elements of $\mathcal{V}_b$ and $C^\infty(M)$, i.e.
\begin{equation}
P = \sum_{j\in J, \ \  l(j) \leq k} a_j X_{1, j} \cdots X_{l(j), j}, \quad \text{where} \ \ a_j \in C^\infty(M), \ \ X_i, j \in \mathcal{V}_b
\end{equation}
The set of all $b$-differential operators of order $k$ on $M$ is denoted by ${\rm Diff}^k_b(M)$. Then $P \in {\rm Diff}^k_b(M)$ if and if only $P$ is a differential operator of order $k$ (in the usual sense), and near $\partial M$ it can be expressed as
\begin{equation}\label{eqn: b-operator local expression}
P = \sum_{|\alpha| \leq k} a_{\alpha}(x, y_1, \cdots, y_{n-1}) (x\partial x)^{\alpha_0} (\partial_{y_1})^{\alpha_1} \cdots (\partial_{y_{n-1}})^{\alpha_{n-1}},
\end{equation}
where $\alpha = (\alpha_0, \alpha_1, \cdots, \alpha_{n-1}) \in \mathbb{N}_0^n$ are multi-indices, and the coefficients $a_\alpha$ are smooth up to $\partial M$.

For $P \in {\rm Diff}^k_b(M)$ , the {\em principal symbol} of $P$, denoted by ${^b\sigma}_k(P)$,  is a smooth section of the bundle of $k$-symmetric powers $\bigodot^k \left({^bTM}\right)$. It can be viewed as a homogenous polynomial of degree $k$ on $^bT^*M$, and if $P$ is locally expressed as $(\ref{eqn: b-operator local expression})$ near $\partial M$, then ${^b\sigma}_k(P)$ is locally given by
\begin{equation}
{^b\sigma}_k(P)(x, y, \omega) = \sum_{|\alpha| = k} a_{\alpha}(x, y_1, \cdots, y_{n-1})(\omega_0)^{\alpha_0} (\omega_1)^{\alpha_1} \cdots (\omega_{n-1})_{\alpha_{n-1}},
\end{equation}
for $ \omega = \omega_0(x, y) x \partial_x + \sum\limits^{n-1}_{i=1}\omega_i(x, y) \partial_{y_i} \in {^bT}^*M$. Then $P \in {\rm Diff}^k_b(M)$ is said to be {\em $b$-elliptic} if ${^b\sigma}_k(P)\neq 0$ on ${^bT}^*M \setminus 0$. 

In order to extend the parametrix construction of elliptic operators on compact smooth manifolds without boundary to b-elliptic operators $P$, it is natural to consider weighted Sobolev spaces $W^{m, 2}_{\delta}(M)$ to control asymptotic behavior of functions near boundary. But the map:
\begin{equation}\label{eqn: b-map between weighted spaces}
P: W^{m, 2}_{\delta}(M) \rightarrow W^{m-k, 2}_{\delta}(M)
\end{equation}
may not have a good pararmetrix for some $\delta$ to get Fredholm property. 
To find all indices $\delta$ such that the b-elliptic operator $P$ in $(\ref{eqn: b-map between weighted spaces})$ has a nice parametrix, one needs to investigate the so called indicial operator of $P$. In the b-calculus theory, the weighted Sobolev spaces $W^{k, 2}_{\delta}(M)$ are usually defined to be
\begin{equation}
W^{k, 2}_{\delta}(M) := x^\delta W^{k, 2}_b(M) = \{ u = x^\delta v \ \ \mid \ \ v \in W^{k, 2}_{b}(M)\},
\end{equation}
where $W^{k, 2}_{b}(M)$ is the usual Sobolev space on $M$ with respect to the {\em b-metric}, $g$, which 
near $\partial M$ is
\begin{equation}\label{eqn: b-metric}
g = \frac{dx^2}{x^2} + g^{\partial M}
\end{equation}
with $g^{\partial M}$ a smooth metric on $\partial M$. One can check that these weighted Sobolev spaces are the same as that defined  on manifolds with a single conical singularity in \cite{DW18} (also see \S$\ref{subsect: weighted Sobolev spaces})$, once we view a cone metric as a conformal change of a b-metric as discussed in Example $\ref{ex: cone}$

For an operator $P \in {\rm Diff}^k_b(M)$ locally given by (\ref{eqn: b-operator local expression}), its {\em indicial operator}, denoted by $I(P)$, is an (${\mathbb R}^+$-invariant) operator on $\partial M \times \mathbb{R}^{+}$, 
\begin{equation}
I(P) = \sum_{|\alpha| \leq k} a_{\alpha}(0, y_1, \cdots, y_{n-1}) (t\partial_t)^{\alpha_0} (\partial_{y_{1}})^{\alpha_1} \cdots (\partial_{y_{n-1}})^{\alpha_{n-1}}.
\end{equation}
It can be identified, via the Mellin transform on $\mathbb{R}^+$, with a family of differential operator on $\partial M$ parametrized by $z \in \mathbb{C}$, called the indicial family, given by
\begin{equation}
I(P)_{M}(z) := \sum_{|\alpha| \leq m} a_{\alpha}(0, y_1, \cdots, y_{n-1}) z^{\alpha_0}  (\partial_{y_{1}})^{\alpha_1} \cdots (\partial_{y_{n-1}})^{\alpha_{n-1}}.
\end{equation}
The set  
\begin{equation}
spec_b(P) := \left\{ z \in \mathbb{C}  \ \ \mid \ \ I(P)_{M}(z) \text{ is not invertible on} \ \ C^\infty(\partial M) \right\}
\end{equation}
is a discrete subset of $\mathbb{C}$. Then a real number $\delta \in \mathbb{R}$ is said to be a {\em indicial root} of the b-differential operator $P$, if $\delta = {\rm Re}(z)$ for some $z \in spec_b(P)$.

For any $\delta \in \mathbb{R}$ that is not a indicial root of b-elliptic operator $P$, roughly speaking, with the help of inverting the indicial operator $I(P)$, a parametrix of the operator
\begin{equation}\label{eqn: b-elliptic map}
P: W^{m, 2}_{\delta}(M) \rightarrow W^{m-k, 2}_{\delta}(M)
\end{equation}
can be constructed, denoted by 
\begin{equation}
Q: W^{m-k, 2}_{\delta}(M) \rightarrow W^{m, 2}_{\delta}(M)
\end{equation}
such that
\begin{equation}
    \begin{cases}
    Q \circ P  = id - R_{1} , \cr
    P \circ Q  = id - R_{2},
    \end{cases}
\end{equation}
where
\begin{equation}
\begin{cases}
    R_{1}: &  W^{k, 2}_{\delta}(M) \rightarrow W^{\infty, 2}_{\delta}(M) \cap \mathcal{A}_{phg}(M), \cr
    R_{2}: &  W^{m-k, 2}_{\delta}(M) \rightarrow W^{\infty, 2}_{\delta}(M) \cap \mathcal{A}_{phg}(M)
\end{cases}
\end{equation}
are compact operators. Here $W^{\infty, 2}_{\delta}=\cap^{\infty}_{k=1}W^{k, 2}_{\delta}$, and $\mathcal{A}_{phg}(M)$ denotes the space of smooth functions on $M$ that admit a full asymptotic expansion, as $x \rightarrow0$, in terms of power of $x$ and $\ln x$ with smooth coeffiecients. In particular, the map in $(\ref{eqn: b-elliptic map})$ is Fredholm, if $P$ is b-elliptic and $\delta$ is not a indicial root of $P$.

\begin{ex}[Manifolds with a cylindrical end]\label{ex: cylinder}
{\rm
Let $(M^n, g)$ be a complete manifold with a cylindrical end, that is, there exists a relatively compact open submanifold $M_0 \subset M$ such that  its complement, $M \setminus M_0$, with metric $g$ is isometric to a cylinder. In other words, $M \setminus M_0 $ is diffeomorphic to $[0, +\infty) \times N$ with metric
\begin{equation}
g = ds^2 + g^N,
\end{equation}
where $g^N$ is a smooth metric on compact manifold $N$.  

By doing a change of variable, $x = e^{-s}$, the metric $g$ becomes
$
g = \frac{dx^2}{x^2} + g^N,
$
which is a b-metric in $(\ref{eqn: b-metric})$, and $M$ can be viewed as the interior of the compact manifold with boundary with a b-metric and $x$ as a defining function of the boundary. In this change of variable, Laplace operator on $(M, g)$, $\Delta = \frac{\partial^2}{\partial s^2} + \Delta_{g^N}$ becomes {$(x\partial_x)^2 + \Delta_{g^N}$}. Clearly, this is a b-elliptic differential operator, with $I(\Delta)$ identified with $z^2 + \Delta_{g^N}$. Thus, the indicial roots are the square roots of the eigenvalues of $\Delta_{g^N}$. 
\qed
}
\end{ex}

\begin{ex}[Manifolds with a conical singularity]\label{ex: cone}
{\rm
Let $(M^n, g, o)$ be a compact manifold with a single conical singularity, that is a manifold as defined in Definition $\ref{defn: ManifoldWithNonisolatedConicalSingularities}$ such that base $B = o$ is a point and near the singular point $o$ the manifold is isometric to $(0, 1) \times F$ with metric
\begin{equation}
g_0 = dr^2 + r^2 \hat{g}
\end{equation}
where $\hat{g}$ is a smooth metric on $F$ (To emphasize this is also called an {\em exact} conical singularity). Clearly, this metric is conformal to the b-metric $\frac{dr^2}{r^2} + \hat{g}$ with conform factor $r^2$.

The regular part of the manifold can be viewed as a compact manifold with boundary. The elliptic operator, $L := -4\Delta_g + R_g$, which we study in this paper,  can be expressed near the singular point $o$ as
\begin{equation}
L_0 := -4\Delta_{g_0} + R_{g_0} = -4\left( \partial_r^2 + \frac{n-1}{r}\partial_r + \frac{1}{r^2} \Delta_{\hat{g}}\right) + \frac{R_{\hat{g}} - (n-1)(n-2)}{r^2}.
\end{equation}
Then 
\begin{equation}
r^2 L_0 = -4(r \partial_r)^2 - 4(n-2)r\partial_r - 4 \Delta_{\hat{g}} + R_{\hat{g}} - (n-1)(n-2).
\end{equation}
Thus, $r^2 L_0$ is a b-elliptic operator, 
with its indicial family given by 
\begin{equation}
I(r^2 L_0)_M (z) = -4z^2 - 4(n-2) z + \left( -4\Delta_{\hat{g}} + R_{\hat{g}} \right) - (n-1)(n-2).
\end{equation}
This is not invertible if and only if $z = \frac{-(n-2) \pm \sqrt{\nu - (n-2)}}{2}$, where $\nu$ is an eigenvalue of the operator $-4\Delta_{\hat{g}} + R_{\hat{g}}$. Note that these $z$ are all real numbers, if we assume $R_{\hat{g}} > n-2$. Then for any 
\begin{equation}\label{eqn: indicial root}
\delta \not\in \left\{ \frac{-(n-2) \pm \sqrt{\nu - (n-2)}}{2} \ \ \mid \ \  \nu \ \  \text{is an eigenvalue of} \ \ -4\Delta_{\hat{g}} + R_{\hat{g}} \right\}
\end{equation}
and $k \in \mathbb{N}_0$, the operator
\begin{equation}
r^2 L_0: W^{k+2, 2}_\delta(M) \rightarrow W^{k, 2}_{\delta}(M)
\end{equation}
has a parametrix $Q$ such that
\begin{equation}\label{eqn: parametrix equation}
    \begin{cases}
    Q \circ r^{2}L_{0}  = id - R_{1} , \cr
    r^{2}L_{0} \circ Q  = id - R_{2},
    \end{cases}
\end{equation}
where
\begin{equation}\label{eqn: parametrix map}
    \begin{cases}
    Q: & W^{k, 2}_{\delta}(M, g) \rightarrow W^{k+2, 2}_{\delta}(M, g), \cr
    R_1: & W^{k+2, 2}_{\delta}(M, g) \rightarrow W^{\infty, 2}_{\delta}(M, g) \cap \mathcal{A}_{phg}(M), \cr
    R_2: & W^{k, 2}_{\delta}(M, g) \rightarrow W^{\infty, 2}_{\delta}(M, g) \cap \mathcal{A}_{phg}(M)
    \end{cases}
\end{equation}
are bounded operators, where $\mathcal{A}_{phg}(M)$ denotes the space of smooth functions on $M$ that admit a full asymptotic expansion, as $r\rightarrow0$, in terms of power of $r$ and $\ln r$ with smooth coeffiecients.  
\qed
}
\end{ex}


\subsection{A partial asymptotic expansion of eigenfunctions}\label{subsect: eigenfunction partial asymptotic expansion}

Let $(M^n, g, o)$ be a compact manifold with a single conical singularity at $o$, as in Definition $\ref{defn: ManifoldWithNonisolatedConicalSingularities}$ but with the base manifold $B$ being a point $o$. The singular point $o$ has a neighborhood that is isometric to $C_1(F) := (0, 1) \times F$ with metric $g = dr^2 + r^2 \hat{g} + h = g_0 +h $, where $h$ has asymptotic control as in $(\ref{eqn: metric asymptotic control})$, i.e.
\begin{equation}\label{eqn: asymptotic control of metric}
|r^{k}\nabla^{k}_{g_{0}}h|_{g_{0}}=o(r^{\alpha}),
\end{equation}
as $r\rightarrow 0$, for some $\alpha>0$ and all $k\in\mathbb{N}$. Set $ L := -4\Delta_g + R_g$.

Now we take a cut-off function $\psi$ on $(0, 1) \times F $ such that 
\begin{equation}
\psi = \begin{cases}
       0, & \text{on} \left(0, \frac{1}{4}\right)\times F, \cr
       1, & \text{on}  \left(\frac{3}{4}, 1\right) \times F,
       \end{cases}
\end{equation}
and use it to modify the metric $g$ near the singular point by letting $g_0 + \psi h$ be the new metric on $(0, 1) \times F$ and keeping the metric $g$ on $M \setminus \left( (0, 1) \times F \right)$ unchanged. This new metric is exactly conical 
near the singular point $o$. Let $L_0$ denote the Schr\"odinger operator $-4\Delta + R$ with respect to the new metric on $M$. Then as seen in Example $\ref{ex: cone}$, $r^2 L_0$ is a b-elliptic differential operator. Here we can make a smooth extension of the radial variable $r$ on $(0, 1) \times F$ to the whole manifold $M$.

By the asymptotic control (\ref{eqn: asymptotic control of metric}) of the perturbation term $h$ of the metric $g$, 
the difference 
\begin{equation}\label{eqn: difference map}
    L - L_{0}: W^{k+2, 2}_{\delta} \rightarrow W^{k, 2}_{\delta-2+\alpha}
\end{equation}
is a bounded second order differential operator for any fixed $k\in\mathbb{N}$ and $\delta\in\mathbb{R}$. Here and in the rest of the paper, we will use $W^{k, p}_\delta$ to denote $W^{k, p}_\delta(M, g)$ for simplicity of notations.

\begin{thm}\label{prop: partial asymptotic expansion}
Let $u$ be an eigenfunciton of $-4\Delta_{g}+R_{g}$ on a $n$-dimensional manifold $(M^n, g)$ with a conical singularity, satisfying that the scalar curvature of the cross section $(F, \hat{g})$ of the model cone: $R_{\hat{g}} > (n-2)$. Then 
\begin{equation}
    u=a(y)r^{\mu} + o(r^{\mu^{\prime}}),
\end{equation}
as $r\rightarrow0$ (approaching the singular point), where $\mu=-\frac{n-2}{2}+\frac{\sqrt{\nu-(n-2)}}{2}$ and $a(y), y\in F,$ is an eigenfunction of $-4\Delta_{\hat{g}}+R_{\hat{g}}$ with eigenvalue $\nu$ on $(F, \hat{g})$, and $\mu^{\prime}>\mu$.

Moreover,
\begin{equation}
|\nabla^{i} (u-a(y)r^{\mu})| = o(r^{\mu^{\prime}-i})
\end{equation}
as $r\rightarrow 0$, for all $i\in \mathbb{N}$.
\end{thm}

\begin{proof}
The eigenfunction $u$ satisfies the equation: $Lu=\lambda u$, and $u \in W^{1, 2}_{1-\frac{n}{2}}$ by the estimate in Proposition $\ref{prop: semiboundedness estimate}$. Then by using weighted elliptic estimate in Proposition $\ref{prop: weighted elliptic estimate}$, a weighted elliptic bootstrapping argument gives:
\begin{equation}\label{eqn: eigenfunction regularity}
    u\in W^{k, 2}_{1-\frac{n}{2}}, \quad \forall k \in \mathbb{N}.
\end{equation}

The eigenfunction equation $Lu = \lambda u$ can be rewritten as
\begin{equation}\label{eqn: rearranged eigen equation}
    r^{2}L_{0}u = \lambda r^{2}u + r^{2}(L_{0} - L)u.
\end{equation}
Clearly, the eigenvalues of $-4\Delta_{\hat{g}} + R_{\hat{g}}$: $\nu > (n-2)$, since $R_{\hat{g}} > (n-2)$. In particular, this implies that $\delta = 1 - \frac{n}{2}$ is not a indicial root of the operator $r^2 L_0$ by $(\ref{eqn: indicial root})$. 
Applying the parametrix $Q$ in Example $\ref{ex: cone}$ for the operator: $r^2 L_0: W^{k, 2}_{1-\frac{n}{2}} \rightarrow W^{k-2, 2}_{1-\frac{n}{2}}$,  to the equation $(\ref{eqn: rearranged eigen equation})$ and using (\ref{eqn: parametrix equation}), one obtains a decomposition of $u$ as
\begin{equation}\label{eqn: eigenfunction decomposition}
    u=R_{1}u + Q(\lambda r^{2}u) + Q(r^{2}(L_{0}-L)u).
\end{equation}

By the regularity of $u$ in $(\ref{eqn: eigenfunction regularity})$, we have $r^{2}u\in W^{k, 2}_{3-\frac{n}{2}}$
and so 
\begin{equation}\label{eqn: Q first term}
Q(r^{2}u)\in W^{k, 2}_{3-\frac{n}{2}}, \quad \forall k \in \mathbb{N}.
\end{equation}
Moreover, by (\ref{eqn: difference map}), 
we have
$
r^{2}(L_{0} - L)u \in W^{k, 2}_{1-\frac{n}{2}+\alpha},
$
and so 
\begin{equation}\label{eqn: Q second term}
Q(r^{2}(L_{0} - L)u) \in W^{k, 2}_{1-\frac{n}{2}+\alpha}, \quad \forall k \in \mathbb{N}.
\end{equation}

Furthermore, $R_{1}u\in W^{\infty, 2}_{1-\frac{n}{2}} \cap \mathcal{A}_{phg}$; in particular, by looking at the first two leading order terms in the full asymptotic expansion of $R_1 u$, we have
\begin{equation}\label{eqn: asymptotic R1u}
    R_{1}u \sim a_p(y)r^{\mu}(\ln r)^{p} + a_{p-1}(y) r^{\mu}(\ln r)^{p-1},
\end{equation}
as $r\rightarrow0$ for some $\mu\in\mathbb{R}$, $p\in \mathbb{N} \cup \{0\}$, {\em nonzero} smooth function $a_p(y)$ on $F$, and a smooth function $a_{p-1}(y)$ on $F$ ($a_{p-1}(y)$ may be zero).  Hence (up to higher order terms), 
\begin{eqnarray}
    & & r^{2}L_{0} \circ R_{1}u 
      \sim 
    \{ -\nu a_p(y) + (-4\Delta_{\hat{g}}+R_{\hat{g}})a_p(y)\} r^{\mu}(\ln r)^{p} \nonumber \\
    &  & \quad + \left\{ \sigma a_p(y) - \nu a_{p-1}(y) + \left(-4\Delta_{\hat{g}} + R_{\hat{g}}\right) a_{p-1}(y)  \right\} r^\mu (\ln r)^{p-1},\label{eqn: asymptotic r2L0R1u}
\end{eqnarray}
as $r\rightarrow0$, where
\begin{equation}\label{eqn: nu}
\nu := - [ - 4\mu(n+\mu-2) - (n-1)(n-2)], \ \ \  \sigma := -4p (2\mu -1) - 4(n-1)p 
\end{equation}

On the other hand, by applying $r^{2}L_{0}$ to the equation (\ref{eqn: eigenfunction decomposition}), and using eigenfunction equation and (\ref{eqn: parametrix equation}), we obtain (assuming without loss of generality that $0<\alpha<2$)
\begin{equation}\label{eqn: r2L0 acts on remainder}
    r^{2}L_{0}\circ R_{1}u = \lambda R_{2}(r^{2}u) + R_{2}(r^{2}(L_{0}-L)u) \in W^{\infty, 2}_{1-\frac{n}{2}+\alpha}.
\end{equation}


{\bf Claim}:
\begin{equation}\label{eqn: claim}
-\nu a_p(y) + (-4\Delta_{\hat{g}}+R_{\hat{g}})a_p(y) \equiv 0,
\end{equation} 
and
\begin{equation}\label{eqn: claim2}
\sigma a_p(y) - \nu a_{p-1}(y) + \left(-4\Delta_{\hat{g}} + R_{\hat{g}}\right) a_{p-1}(y)  \equiv 0.
\end{equation}
Otherwise,
by the asymptotic behavior in $(\ref{eqn: asymptotic r2L0R1u})$, the inclusion in $(\ref{eqn: r2L0 acts on remainder})$, and the observation in $(\ref{eqn: polynomial weighted Sobolev condition})$, we obtain 
\begin{equation}\label{eqn: mu-inequality}
\mu > 1 - \frac{n}{2} + \alpha.
\end{equation}
Here note that $r^\mu (\ln r)^p, r^\mu (\ln r)^{p-1}, \cdots$ and $r^\mu$ are linearly independent functions. By  $(\ref{eqn: asymptotic R1u})$, and again $(\ref{eqn: polynomial weighted Sobolev condition})$, $(\ref{eqn: mu-inequality})$ then implies:
\begin{equation}\label{eqn: R1u estimate}
R_{1}u \in W^{\infty, 2}_{1-\frac{n}{2}+\alpha}.
\end{equation}
Hence by (\ref{eqn: eigenfunction decomposition}), (\ref{eqn: Q first term}), (\ref{eqn: Q second term}) and (\ref{eqn: R1u estimate}), we obtain
\begin{equation}
    u \in W^{\infty, 2}_{1-\frac{n}{2}+\alpha}.
\end{equation}
Iterating the above process, we obtain 
\begin{equation}
    u\in W^{\infty, 2}_{1-\frac{n}{2}+k\alpha}, \ \ \forall k\in \mathbb{N}.
\end{equation}
As a result, by Lemma 8.1 in \cite{DW18}, 
\begin{equation}
    u=o(r^{1-\frac{n}{2}+k\alpha}), \ \ \forall k \in \mathbb{N}.
\end{equation}
The argument also implies  
\begin{equation}
Q(\lambda r^{2}u) + Q(r^{2}(L_{0}-L)u) = o (r^{1-\frac{n}{2}+k\alpha}), \ \ \forall k \in \mathbb{N}.
\end{equation} 
Therefore, by (\ref{eqn: eigenfunction decomposition}), we have
\begin{equation}
R_{1}u = o(r^{1-\frac{n}{2}+k\alpha}), \ \ \forall \in \mathbb{N}.
\end{equation}
Then the asymptotic expansion (\ref{eqn: asymptotic r2L0R1u}) implies $a_p(y)\equiv 0$, and this contradicts with the assumption: $a_p(y) \not\equiv 0$. Thus Claims $(\ref{eqn: claim})$ and $(\ref{eqn: claim2})$ hold. 

Clearly, the equation $(\ref{eqn: claim})$ says that $a_p(y)$ is an eigenfunction of $-4\Delta_{\hat{g}}+R_{\hat{g}}$ with eigenvalue $\nu$. By solving (\ref{eqn: nu}), we obtain
\begin{equation}\label{eqn: mu-nu-relation}
    \mu=-\frac{n-2}{2}+\frac{\sqrt{\nu-(n-2)}}{2}.
\end{equation}
Note that we have used our previous result, Theorem 1.4 of \cite{DW18}, together with $\nu > n-2$ via our assumption $R_{\hat{g}}>n-2$, to rule out the negative sign in the quadratic formula here.

We can write the smooth function $a_{p-1}(y)$ on $F$ as a linear combination of eigenfunctions of $-4\Delta_{\hat{g}}+R_{\hat{g}}$. Note that eigenfunctions with eigenvalue $\nu$ will go away in the expression: $-\nu a_{p-1}(x)+(-4\Delta_{\hat{g}}+R_{\hat{g}})a_{p-1}(x)$, and recall that $a_{p}(x)$ is an eigenfunction with eigenvalue $\nu$. Thus $a_{p}(y)$ and $- \nu a_{p-1}(y) + (-4\Delta_{\hat{g}}+R_{\hat{g}})a_{p-1}(y)$ are linearly independent functions. Thus, $(\ref{eqn: claim2})$ implies:
\begin{equation}\label{eqn: p-mu-relation}
   \sigma= -4p(2\mu-1) -4(n-1)p=0.
\end{equation}
Then by plugging $(\ref{eqn: mu-nu-relation})$ into $(\ref{eqn: p-mu-relation})$, we obtain
\begin{equation*}
    -4p\sqrt{\nu-(n-2)}=0.
\end{equation*}
Hence $p=0$, since $\nu>(n-2)$ (by our assumption $R_{\hat{g}} > (n-2)$).

Now to complete the proof, by the decomposition in (\ref{eqn: eigenfunction decomposition}), it suffices to show that \begin{equation}\label{eqn: asymptotic of Qu}
    Q(\lambda r^{2}u)+Q(r^{2}(L_{0}-L)u) = o(r^{\mu^{\prime}}), \ \ \text{for some} \ \ \mu^\prime > \mu.
\end{equation} 
By $(\ref{eqn: Q first term})$ and $(\ref{eqn: Q second term})$, we have that 
\begin{equation}
    Q(\lambda r^{2}u)+Q(r^{2}(L_{0}-L)u) \in W^{\infty, 2}_{1-\frac{n}{2}+\alpha},
\end{equation}
and so
\begin{equation}
     Q(\lambda r^{2}u)+Q(r^{2}(L_{0}-L)u) = o(r^{1-\frac{n}{2}+\alpha}),
\end{equation}
by Lemma 8.1 in \cite{DW18}.
If $1-\frac{n}{2}+\alpha > \mu$, then set $\mu^{\prime}=1-\frac{n}{2} + \alpha$, and we are done.

Otherwise, $R_{1}u \sim a_0(y)r^{\mu} \in W^{\infty, 2}_{1-\frac{n}{2}+\beta}$ for any $\beta<\alpha$ by $(\ref{eqn: polynomial weighted Sobolev condition})$. Fix $0< \beta<\alpha$. Then (\ref{eqn: eigenfunction decomposition}) implies $u \in W^{\infty, 2}_{1-\frac{n}{2}+\beta}$. Starting with this inclusion, by repeating above derivation, one can obtain $Q(\lambda r^{2}u)+Q(r^{2}(L_{0}-L)u) = o(r^{1-\frac{n}{2}+\beta+\alpha}).$ Now if $1-\frac{n}{2}+\beta+\alpha > \mu$, then we are done. Otherwise, we can iterate above process and 
obtain (\ref{eqn: asymptotic of Qu}). Actually, we have
\begin{equation}
    |\nabla^{i}(Q(\lambda r^{2}u) + Q(r^{2}(L_{0}-L)u))|_{g}=o(r^{\mu^{\prime}-i})
\end{equation}
for all $i\in\mathbb{N}$.

Thus, we obtain
\begin{equation}\label{eqn: eigenfunction decomposition 2}
    u=a_0(y)r^{\mu} + u_{1},
\end{equation}
and $|\nabla^{i} u_{1}|=o(r^{\mu^{\prime}-i})$, as $r\rightarrow 0$, for some $\mu^{\prime}>\mu$ and all $i\in \mathbb{N}\cup \{ 0 \}$.
\end{proof}

Using a similar but more complicated argument, we now derive a partial asymptotic expansion for the minimizer $u$ of $W$-functional as in Theorem $\ref{thm: partial asymptotic expansion w-functional minimizer}$ below. Recall that the minimizer $u$ satisfies the equation (see (4.17) in \cite{DW20})
\begin{equation}\label{eqn: w-functional minimizer equation}
(-4\Delta + R)u = \frac{2}{\tau}u\ln u + \frac{n+m}{\tau} u,
\end{equation}
where $m$ is the infimum of $W$-functional with the scaling parameter $\tau>0$ as in $(\ref{eqn: mu-functional})$. The main difference between this equation and that of the eigenfunction as discussed in Theorem \ref{prop: partial asymptotic expansion} is the $u\ln u$ term on the right hand of the Euler-Lagrange equation (\ref{eqn: w-functional minimizer equation}).  So for the simplicity of presentation, instead of (\ref{eqn: w-functional minimizer equation}), we will consider the equation
\begin{equation}\label{eqn: u ln u}
(-4\Delta + R)u = u\ln u.
\end{equation}

The basic idea of the proof of Theorem \ref{thm: partial asymptotic expansion w-functional minimizer} is similar to that of Theorem \ref{prop: partial asymptotic expansion}. We will still use the parametrix of the Schr\"odinger operator $L_{0}$ in Melrose's b-calculus theory. In the derivation for the partial asymptotic expansion  of eigenfunction in Theorem \ref{prop: partial asymptotic expansion}, another key ingredient is that one can obtain $u\in W^{\infty, 2}_{1-\frac{n}{2}}$ by using a weighted elliptic bootstrapping with the help of Proposition \ref{prop: weighted elliptic estimate}. This helps us estimate the term $Q(\lambda r^{2}u) + Q(r^{2}(L_{0}- L)u)$ in (\ref{eqn: eigenfunction decomposition}). 

However, the $u\ln u$ term prevents us from getting the weighted $L^{2}$ estimate for high order derivatives of a solution to $(\ref{eqn: u ln u})$. Instead, we combine the b-calculus theory and weighted $L^{p}$ elliptic bootstrapping argument to obtain the following asymptotic estimate. This improves the asymptotic estimate obtained in \cite{DW20}.

\begin{thm}\label{thm: partial asymptotic expansion w-functional minimizer}
Let $u>0$ be a solution of 
\begin{equation}
(-4\Delta_{g}+R_{g})u = u\ln u.
\end{equation}
on a manifold $(M^n, g, o)$ with a conical singularity at $o$, satisfying $R_{\hat{g}} > (n-2)$.
 Then 
\begin{equation}
    |\nabla^{i}u|= o(r^{\mu-i})
\end{equation}
as $r\rightarrow0$ (approaching the singular point), for $i=0, 1$ and any $\mu < -\frac{n-2}{2}+\frac{\sqrt{\nu-(n-2)}}{2}$, where $\nu$ is the smallest eigenvalue of $-4\Delta_{\hat{g}}+R_{\hat{g}}$ on $(F, \hat{g})$.

\end{thm}

\begin{proof}
We prove the theorem in two steps. In the first step we show that, for some real number $\mu^\prime$ and any $\mu < \mu^\prime$, $|\nabla^i u| = o(r^{\mu - i})$ as $r \to 0$, $i=0, 1$. In the second step, we estimate $\mu^\prime$ to complete the proof.

\noindent{\bf Step 1}.
The solution $u>0$ satisfies the equation $Lu= u \ln u$. So
\begin{equation}\label{eqn: rearranged w-functional equation}
    r^{2}L_{0}u = r^{2}u\ln u + r^{2}(L_{0} - L)u.
\end{equation}
Applying the parametrix $Q$ of the operator: $r^2L_0: W^{2,2}_{1-\frac{n}{2}} \to W^{0, 2}_{1-\frac{n}{2}}$, to this equation and using (\ref{eqn: parametrix equation}), we write $u$ as
\begin{equation}\label{eqn: w-functional minimizer decomposition}
    u=R_{1}u + Q( r^{2}u\ln u) + Q(r^{2}(L_{0}-L)u).
\end{equation}

Recall that $u\in W^{1,2}_{1-\frac{n}{2}}$ and $u$ is smooth by the local elliptic regularity. By the weighted Sobolev embedding in Proposition 3.4 in \cite{DW20}, we have
\begin{equation}
u\in W^{0, p}_{1-\frac{n}{2}}, \quad \forall 1\leq p \leq \frac{2n}{n-2}.
\end{equation}
Now, for any $\gamma>0$, there is a constant $a(\gamma)$ such that
\begin{equation}
|u\ln u|  \leq a(\gamma) + |u|^{1+\gamma}.
\end{equation}
It follows that, for any small $\gamma>0$,
\begin{equation}
u\ln u \in W^{0, p}_{\left(1-\frac{n}{2}\right)(1+\gamma)}, \quad \forall 1\leq p \leq \frac{2n}{(n-2)(1+\gamma)}.
\end{equation}
In particular, for $p=2$, 
\begin{equation}
Q(r^{2}u\ln u) \in W^{2, 2}_{(1-\frac{n}{2})(1+\gamma)+2}.
\end{equation}
Thus, by taking $\gamma$ sufficiently small, we obtain
\begin{equation}\label{eqn: W Q first term}
Q(r^{2}u\ln u) \in W^{2, 2}_{1-\frac{n}{2}+\alpha}.
\end{equation}

In addition, as shown in the proof of Theorem 5.2 in \cite{DW20}, by Proposition 5.1 in \cite{DW20}, we can obtain 
\begin{equation}
u\in W^{2, p}_{(1-\frac{n}{2})(1+\gamma)}, \quad \forall 1\leq p \leq \frac{2n}{(n-2)(1+\gamma)}, \ \ \gamma>0.
\end{equation}
Then by the boundedness of the map in ($\ref{eqn: difference map}$), we have 
\begin{equation}
 r^{2}(L_{0} - L) u \in W^{0, 2}_{1-\frac{n}{2}+\alpha},
\end{equation}
and so
\begin{equation}\label{eqn: W Q second term}
Q(r^{2}(L_{0} - L)u) \in W^{2, 2}_{1-\frac{n}{2}+\alpha}.
\end{equation}

Thus, by $(\ref{eqn: W Q first term})$ and $(\ref{eqn: W Q second term})$, we obtain
\begin{equation}
Q(r^{2}u\ln u) + Q(r^{2}(L_{0} - L)u) \in W^{2, 2}_{1-\frac{n}{2}+\alpha}.
\end{equation}

On the other hand, $R_{1}u\in W^{\infty, 2}_{1-\frac{n}{2}} \cap \mathcal{A}_{phg}$; in particular, by looking at the leading order term,
\begin{equation}\label{eqn: asymptotic R1u w-functional}
    R_{1}u \sim a(y)r^{\mu^{\prime}}(\ln r)^{p},
\end{equation}
as $r\rightarrow0$ for some $\mu^{\prime}\in\mathbb{R}$. So $R_{1}u \in W^{2, 2}_{\mu}$ for any $\mu < \mu^{\prime}$. 

We first {\bf Claim}:
\begin{equation}\label{eqn: W-functional claim}
u \in W^{2, 2}_{\mu}, \quad \text{for any} \ \ \mu < \mu^{\prime}.
\end{equation}

If $\mu^{\prime} \leq 1-\frac{n}{2}+\alpha$, then 
\begin{equation}
Q(r^{2}u\ln u) + Q(r^{2}(L_{0} - L)u) \in W^{2, 2}_{1-\frac{n}{2}+\alpha} \subset W^{2, 2}_{\mu},
\end{equation}
for any $\mu < \mu^{\prime} \leq 1- \frac{n}{2} + \alpha$. Consequently,  (\ref{eqn: w-functional minimizer decomposition}) implies that 
$
u\in W^{2, 2}_{\mu},
$
for any $\mu < \mu^{\prime}$, and Claim $(\ref{eqn: W-functional claim})$ follows.

If $\mu^{\prime} > 1-\frac{n}{2} + \alpha$, then $u\in W^{2, 2}_{1-\frac{n}{2}+\alpha}$. Starting from this regularity estimate for $u$, and applying the above process, we obtain 
\begin{equation}
Q(r^{2}u\ln u) + Q(r^{2}(L_{0} - L)u) \in W^{2, 2}_{1-\frac{n}{2}+2\alpha}.
\end{equation}
Now if $\mu^{\prime} \leq 1-\frac{n}{2}+2\alpha$, then we obtain $u\in W^{2, 2}_{\mu}$, for any $\mu < \mu^{\prime}$. Otherwise, we repeat the above argument finitely many times to prove Claim $(\ref{eqn: W-functional claim})$.

Starting from the regularity estimate in Claim $(\ref{eqn: W-functional claim})$, a weighted $L^{p}$ elliptic bootstrapping argument as in the proof of Theorem 5.2 in \cite{DW20} gives 
\begin{equation}
|\nabla^{i}u| = o(r^{\mu-i})
\end{equation}
as $r\rightarrow 0$, for $i=0, 1$ and any $\mu<\mu^{\prime}$.

\noindent{\bf Step 2}.
Finally, we estimate the exponent $\mu^{\prime}$ in the leading order term of $R_{1}u$ to complete the proof. For that, we apply $r^2L_{0}$ to $R_{1}u$ to obtain, via $(\ref{eqn: asymptotic R1u w-functional})$, 
\begin{equation}\label{eqn: asymptotic r2L0R1u w-functional}
    r^{2}L_{0} \circ R_{1}u 
    \sim 
    \{ [-4\mu^\prime (\mu^\prime +n-2) - (n-1)(n-2)]a(y) + (-4\Delta_{\hat{g}}+R_{\hat{g}})a(y)\} r^{\mu^\prime}(\ln r)^{p},
\end{equation}
as $r\rightarrow0$.

On the other hand, applying $r^{2}L_{0}$ to the equation (\ref{eqn: w-functional minimizer decomposition}), using the equation $Lu = u\ln u$ and (\ref{eqn: parametrix equation}), one has
\begin{equation} \label{eq-r2L0R1u-Sobolev}
    r^{2}L_{0}\circ R_{1}u = R_{2}(r^{2} u\ln u) + R_{2}(r^{2}(L_{0}-L)u) \in W^{\infty, 2}_{1-\frac{n}{2}+\alpha}.
\end{equation}

Now we separate the discussion into two cases:

{\bf Case 1}: If 
\begin{equation}
[-4\mu^\prime (\mu^\prime+n-2) - (n-1)(n-2)]a(y) + (-4\Delta_{\hat{g}}+R_{\hat{g}})a(y) \equiv 0,
\end{equation}
then $a(y)$ is an eigenfunction of $-4\Delta_{\hat{g}}+R_{\hat{g}}$ and
\begin{equation}
    \mu^\prime =-\frac{n-2}{2}+\frac{\sqrt{\nu^\prime -(n-2)}}{2},
\end{equation}
where $\nu^\prime $ is an eigenvalue of $-4\Delta_{\hat{g}}+R_{\hat{g}}$ on $(F, \hat{g})$. And we are done. Here we have used  Theorem 5.2 in \cite{DW20}, together with $\nu'>n-2$ via our assumption on the scalar curvature, to rule out the other root.

{\bf Case 2}: If
\begin{equation}
[-4\mu^\prime (\mu^\prime +n-2) - (n-1)(n-2)]a(x) + (-4\Delta_{\hat{g}}+R_{\hat{g}})a(x) \not\equiv 0,
\end{equation}
then $\mu'> 1-\frac{n}{2}+\alpha$ by \eqref{eq-r2L0R1u-Sobolev}, (\ref{eqn: asymptotic r2L0R1u w-functional}), and (\ref{eqn: polynomial weighted Sobolev condition}).
Thus (\ref{eqn: asymptotic R1u w-functional}) and (\ref{eqn: polynomial weighted Sobolev condition}) 
imply
\begin{equation}
R_{1}u \in W^{\infty, 2}_{1-\frac{n}{2}+\alpha}.
\end{equation}
It follows from (\ref{eqn: w-functional minimizer decomposition}) that
\begin{equation}
    u \in W^{1, 2}_{1-\frac{n}{2}+\alpha}.
\end{equation}
Iterating the above process leads to 
\begin{equation}
    u\in W^{1, 2}_{1-\frac{n}{2}+k\alpha}, \ \ \forall k\in \mathbb{N}.
\end{equation} 
Then as in the proof of Theorem 5.2 in \cite{DW20}, we can obtain that 
\begin{equation}
    |\nabla^{i}u|=o(r^{1-\frac{n}{2}+k\alpha - i})
\end{equation}
for $i=0, 1$, and any $k\in\mathbb{N}$. In particular, $|\nabla^{i}u|=o(r^{\mu-i})$, for $i=0, 1$ and any $\mu < -\frac{n-2}{2} + \frac{\sqrt{\nu - (n-2)}}{2}$, where $\nu$ is the smallest eigenvalue of $-4\Delta_{\hat{g}}+R_{\hat{g}}$ on $(F, \hat{g})$. This completes the proof. 

\end{proof}


\section{Gradient Ricci solitons with isolated conical singularities}\label{section: gradient ricci solitions}

In this section, we apply the same idea as in the last section to study gradient Ricci solitons with isolated conical singularities. We first derive some asymptotic estimate for the potential functions of compact gradient Ricci solitons with isolated conical singularities, and use that to prove Theorem $\ref{thm: soliton}$. That is, there are no nontrivial compact gradient steady or expanding Ricci solitons with isolated conical singularities, when the model cone is scalar flat.

Let $(M^n, g, o)$ be a compact manifold with a single conical singularity at $o$ as described at the beginning of \S$\ref{subsect: eigenfunction partial asymptotic expansion}$. Then as we did there, we modify the metric $g$ near the singular point $o$ such that the metric is an exact cone metric: $dr^2 + r^2 \hat{g}$ in a sufficiently small neightborhood of the singular point: $(0, \epsilon) \times F$ (small $\epsilon$ is to be determined as we need in Proposition \ref{prop: Fredholm}), and the new metric is the same as the original metric $g$ on $M \setminus \left( (0, 4\epsilon) \times F \right)$. Let $\Delta_0$ denote the Laplacian with respect to the modified metric, and $\Delta_g$ be the Laplacian with respect to the original (asymptotically) conical metric $g$. Then the map 
\begin{equation}\label{eqn: laplacian difference operator}
\Delta_g - \Delta_0: W^{k+2, 2}_{\delta} \to W^{k, 2}_{\delta -2 + \alpha} \ \ \text{is bounded,}
\end{equation}\label{eqn: laplacian difference operator-1}
for any $\delta \in \mathbb{R}$ and $k \in \mathbb{N} \cup \{0\}$, and so the map
\begin{equation}
r^2\left( \Delta_g - \Delta_0 \right): W^{k+2, 2}_\delta \rightarrow W^{k, 2}_{\delta+\alpha} \ \ \text{is bounded}.
\end{equation} 
Recall that $\alpha$ is the decay order in $(\ref{eqn: asymptotic control of metric})$.

\begin{prop}\label{prop: Fredholm}
Let $(M^n, g, o)$ be a compact manifold with a conical singularity at $o$.
For each $k \in \mathbb{N} \cup \{0\}$ and 
\begin{equation}\label{eqn: indicial root of Laplacian}
\delta \not\in \left\{ \frac{2-n}{2} \pm \frac{\sqrt{(n-2)^2 + 4\nu}}{2} \ \ \mid \ \ \nu \ \ \text{is an eigenvalue of } -\Delta_{\hat{g}} \right\},
\end{equation}
the  map 
\begin{equation}
r^{2}\Delta_{g}: W^{k+2, 2}_{\delta} \rightarrow  W^{k, 2}_{\delta} 
\end{equation}
is Fredholm.
\end{prop}

\begin{proof}
From Example $\ref{ex: cone}$, $r^2 \Delta_0$ is a b-elliptic differential operator, and its indicial roots are 
given by: $\frac{2-n}{2} \pm \frac{\sqrt{(n-2)^2 + 4\nu}}{2}$, where $\nu$ is an eigenvalue of $-\Delta_{\hat{g}}$. Thus by b-calculus theory, for each $\delta$ satisfying $(\ref{eqn: indicial root of Laplacian})$,
\begin{equation}
r^{2}\Delta_0: W^{k+2, 2}_{\delta} \rightarrow  W^{k, 2}_{\delta}
\end{equation}
is Fredholm, and there is a parametrix $Q$, which is a bounded map
\begin{equation}\label{eqn: Q}
Q: W^{k, 2}_{\delta} \rightarrow W^{k+2, 2}_{\delta},
\end{equation}
such that
\begin{equation}
Q \circ r^{2}\Delta_0 = id + R_{L},
\end{equation}
and
\begin{equation}
R_{L} : W^{k+2, 2}_{\delta} \rightarrow W^{k+2, 2}_{\delta}
\end{equation}
is compact.

Then 
\begin{eqnarray*}
Q \circ r^{2}\Delta_{g} 
  & = & id + Q\circ r^{2}(\Delta_{g} - \Delta_0) + R_{L}.
\end{eqnarray*}
For $u\in W^{k+2, 2}_{\delta}(C_{\epsilon}(F))$, by using $(\ref{eqn: laplacian difference operator})$, we have
\begin{eqnarray}
\|Q \circ r^{2}(\Delta_{g} - \Delta_0)u\|_{W^{k+2, 2}_{\delta}}
& \leq & C \|r^{2} (\Delta_{g} - \Delta_0)u\|_{W^{k, 2}_{\delta}} \nonumber \\
& \leq & \epsilon^{\alpha} C \|r^{2} (\Delta_{g} - \Delta_0)u\|_{W^{k, 2}_{\delta+\alpha}} \nonumber \\
& \leq & \epsilon^{\alpha} C \|u\|_{W^{k+2, 2}_{\delta}} \label{eqn: Q r2 laplacian estimate}.
\end{eqnarray}
Here the first inequality is because of the boundedness of the map $Q$ in $(\ref{eqn: Q})$. For the second inequality, we note that $r^2 \left( \Delta_g - \Delta_0 \right) u$ is supported in $(0, 4\epsilon)\times F$ and we use the definition of weighted Sobolev norm $(\ref{eqn: WSN})$. The third inequality then follows from the fact in $(\ref{eqn: laplacian difference operator-1})$.
Then the estimate in $(\ref{eqn: Q r2 laplacian estimate})$ implies that for sufficiently small $\epsilon>0$ the operator
\begin{equation*}
Q \circ r^{2}(\Delta_{g} - \Delta_{g_{0}}): W^{k+2, 2}_{\delta} \rightarrow W^{k+2, 2}_{\delta} 
\end{equation*}
has norm $\leq \frac{1}{2}$. Consequently, $id + Q \circ r^{2}(\Delta_{g} - \Delta_{g_{0}})$ is invertible.
Thus, the map
\begin{equation}
r^{2}\Delta_{g}: W^{k+2, 2}_{\delta} \rightarrow  W^{k, 2}_{\delta}
\end{equation}
is also Fredholm. 
\end{proof}

\begin{prop}\label{prop: surjectivity}
Let $(M^n, g, o)$ be a compact manifold with a conical singularity at $o$. For $2-n \leq  \delta \leq 1 - \frac{n}{2}$ satisfying $(\ref{eqn: indicial root of Laplacian})$, the map
\begin{equation}\label{eqn: surjective map}
r^2 \Delta_g: W^{2, 2}_\delta \to W^{0, 2}_\delta
\end{equation}
is surjective.

In particular, the map $r^2 \Delta_g: W^{2, 2}_{1-\frac{n}{2}} \to W^{0, 2}_{1-\frac{n}{2}}$ is surjective.
\end{prop}
\begin{proof}
Proposition $\ref{prop: Fredholm}$ tells use that for any $\delta \in \mathbb{R}$ satisfying $(\ref{eqn: indicial root of Laplacian})$, the map
$(\ref{eqn: surjective map})$
is Fredholm. Thus, in order to prove the map $(\ref{eqn: surjective map})$ is surjective for $2 - n \leq \delta \leq 1 - \frac{n}{2}$ satisfying $(\ref{eqn: indicial root of Laplacian})$, it suffices to prove its usual $L^2$ adjoint map
\begin{equation}\label{eqn: adjoint map}
\left( r^2 \Delta_g \right)^*: \left( W^{0, 2}_\delta \right)^* \to \left( W^{2, 2}_\delta \right)^*
\end{equation}
is injective. Note that the usual $L^2$-pairing $(\cdot, \cdot)_{L^2(M)}$ identifies the toplogy dual space $\left( W^{0, 2}_\delta \right)^*$ with $W^{0, 2}_{-\delta -n}$. For each small $r_0 >0$, we take a cut-off function $\varphi_{r_0}$ such that 
\begin{equation}
\varphi_{r_0} = \begin{cases}
0, & \text{on} \ \ B_{r_0}(o), \cr
1, & \text{on} \ \ M \setminus B_{2r_0}(o),
\end{cases}
\end{equation}
and $|d\varphi_{r_0}| \leq \frac{10}{r_0}$ on $M$. Then for each $u \in {\rm Ker}\left( \left( r^2 \Delta_g \right)^* \right) \subset W^{0, 2}_{-\delta - n}$, we have $\Delta_g (r^2u) = 0$ and so
\begin{eqnarray}
\int_{M} |\nabla (\varphi_{r_0} (r^2 u))|^2 
& = & \int_{M} |d\varphi_{r_0}|^2 (r^2u)^2 - \int_{M} (\varphi_{r_0})^2 (r^2u)\Delta_g (r^2u) \nonumber \\
& = & \int_{M} |d\varphi_{r_0}|^2 (r^2u)^2 .\end{eqnarray}
Thus, 
\begin{equation}
\int_{M \setminus B_{2r_0}(o)} |\nabla (r^2 u)|^2 \leq 100 \int_{B_{2r_0}(o) \setminus B_{r_0}(o)} r^2 u^2 \to 0, \ \ \text{as} \ \ r_0 \to 0,
\end{equation}
since for $\delta \leq 1 - \frac{n}{2}$,
\begin{equation}
\int_{(0, 1) \times F} r^2 u^2 \leq \int_{(0, 1)\times F} r^{2\delta + n} u^2 \leq \|u\|^2_{W^{0, 2}_{-\delta -n}(M)} < +\infty.
\end{equation}
Therefore, $\nabla (r^2 u) =0$, and so $r^2 u$ is a constant function in $ W^{0, 2}_{-\delta -n +2}$. This implies $r^2 u$ must be zero, and so $u=0$, by the assumption $\delta \geq 2-n$ and (\ref{eqn: polynomial weighted Sobolev condition}). Thus, the map $(\ref{eqn: adjoint map})$ is injective and so the map $(\ref{eqn: surjective map})$ is surjective. 

Finally, we note that $\delta=1-\frac{n}{2}$ is not a indicial root of $r^2 \Delta_g$, via the expression of indicial roots of $r^2 \Delta_g$ in $(\ref{eqn: indicial root of Laplacian})$, and the fact that the eigenvalues of $-\Delta_{\hat{g}}$: $\nu \geq 0$.
\end{proof}

\begin{prop}\label{prop: asymptotic for potential function of steady soliton}
Let $(M^{n}, g, f)$ be a compact gradient Ricci soliton with a conical singularity, i.e. $(M^{n}, g)$ is a compact Riemannian manifolds with a conical singularity and the Ricci curvature $Ric_g$ on the regular part satisfies the equation
\begin{equation}\label{eqn: gradient soliton equation}
    Ric_g + Hess_g f = \Lambda g,
\end{equation}
for some constant $\Lambda$.
 If the cross section of the model cone at the singular point has scalar curvature  $R_{\hat{g}}=(n-1)(n-2)$, then a potential function $f$ satisfies
\begin{equation}
|\nabla^{i}f| = o(r^{\delta_0 -i}), 
\end{equation}
as $r\rightarrow 0$, for some $\delta_0 > 0$ and all $i \in \mathbb{N}_0$. 

In particular, 
$f \rightarrow 0$ at the conical singularity and hence the potential function $f$ is bounded on $M$.
\end{prop}

\begin{remark}\label{remark: non-uniqueness of potential function}
{\rm
Note that potential function $f$ the gradient Ricci soliton equation $(\ref{eqn: gradient soliton equation})$ is not unique, and one can add a constant to a potential function to get a new potential function. 
}
\end{remark}

\begin{proof}
By taking trace for the gradient Ricci soliton equation (\ref{eqn: gradient soliton equation}),
one has 
\begin{equation}\label{eqn: potential function equation 1}
\Delta_{g}f = -R_{g} + n\Lambda.
\end{equation}
Note that the scalar curvature of the asymptotic conical metric $g$ is given by
\begin{equation}
R_{g} = \frac{1}{r^{2}}(R_{\hat{g}}-(n-1)(n-2) + o(r^{\alpha})), \ \ \text{as} \ \ r \rightarrow 0.
\end{equation}
If $R_{\hat{g}}=(n-1)(n-2)$, then $r^2 R_g = o(r^\alpha)$ as $r \to 0$. Thus, by $(\ref{eqn: potential function equation 1})$, $f$ satisfies the equation 
\begin{equation}\label{eqn: potential function equation 2}
r^2 \Delta_g f = - r^2 R_g + r^2 n\Lambda \in W^{0, 2}_{1-\frac{n}{2}},
\end{equation}
where the inclusion follows from $(\ref{eqn: polynomial weighted Sobolev condition})$. Then by the surjectivity obtained in Proposition $\ref{prop: surjectivity}$, we have: $f \in W^{2, 2}_{1-\frac{n}{2}}$. 
 
Now we rewrite the equation $(\ref{eqn: potential function equation 2})$ as
\begin{equation}\label{eqn: potential function equation 3}
r^2 \Delta_0 f = r^2(\Delta_0 - \Delta_g) f - r^2R_g + r^2n \Lambda,
\end{equation}
and use the parametrix $Q$ of $r^2 \Delta_0$ to improve the estimate of asymptotic order of $f$ to complete the proof, starting with the regularity: $f \in W^{2 ,2}_{1-\frac{n}{2}}$. 

Recall that for each non-indicial root $\delta \in \mathbb{R}$, the parametrix $Q: W^{0, 2}_{\delta} \rightarrow W^{2, 2}_{\delta}$ satisfies
\begin{equation}\label{eqn: parametrix equation laplace}
    \begin{cases}
    Q \circ r^{2}\Delta_{0}  = id - R_{1} , \cr
    r^{2}\Delta_{0} \circ Q  = id - R_{2},
    \end{cases}
\end{equation}
where
\begin{equation}\label{eqn: remainder map laplace}
    \begin{cases}
    R_1: & W^{2, 2}_{\delta}(M, g) \rightarrow W^{\infty, 2}_{\delta} \cap \mathcal{A}_{phg}(M), \cr
    R_2: & W^{0, 2}_{\delta}(M, g) \rightarrow W^{\infty, 2}_{\delta} \cap \mathcal{A}_{phg}(M)
    \end{cases}
\end{equation}
are bounded operators, and $\mathcal{A}_{phg}(M)$ denotes the space of smooth functions on $M$ that admit a full asymptotic expansion, as $r\rightarrow0$, in terms of power of $r$ and $\ln r$ with smooth coefficients. 

By applying $Q$ to the equation $(\ref{eqn: potential function equation 3})$ and using $(\ref{eqn: parametrix equation laplace})$, we obtain a decomposition of $f$ as
\begin{equation}\label{eqn: decomposition of potential function}
f = R_1 f + Q\left( r^2 (\Delta_0 - \Delta_g)f \right) - Q\left(r^2R_g - r^2 n \Lambda\right).
\end{equation}
Then $f \in W^{2, 2}_{1-\frac{n}{2}}$ implies $Q\left( r^2 (\Delta_0 - \Delta_g) f \right) \in W^{2, 2}_{1-\frac{n}{2}+\alpha}$ by $(\ref{eqn: laplacian difference operator})$. Moreover, $Q\left(r^2R_g - r^2 n \Lambda\right) \in W^{2, 2}_\delta$ for any $\left( 1-\frac{n}{2} + \alpha < \right) \delta < \alpha$ (assuming without loss of generality that $0 < \alpha <2$). Thus,
\begin{equation}\label{eqn: regularity of second and third terms}
Q\left( r^2 (\Delta_0 - \Delta_g)f \right) - Q\left(r^2R_g - r^2 n \Lambda\right) \in W^{2, 2}_{1-\frac{n}{2} + \alpha},
\end{equation}
by the inclusion in $(\ref{eqn: weighted Sobolev space inclusion})$.
The first term, $R_1 f$, in (\ref{eqn: decomposition of potential function}) has a full asymptotic expansion, and by looking at the leading order term, we have
\begin{equation}\label{eqn: leading order term of remainder}
R_1 f \sim  a(y) r^\mu (\ln r)^p, \quad \text{as} \ \ r \rightarrow 0.
\end{equation}

Then by applying $r^2 \Delta_0$ to the equation $(\ref{eqn: decomposition of potential function})$, and using $(\ref{eqn: potential function equation 2})$ and $(\ref{eqn: parametrix equation laplace})$, we obtain
\begin{equation}\label{eqn: laplace acts on remainder}
r^2 \Delta_0 \circ R_1 f = R_2\left( r^2 (\Delta_0 - \Delta_g) f \right) + R_2\left( r^2 R_g + r^2 n\Lambda \right) \in W^{\infty, 2}_{1-\frac{n}{2} + \alpha}.
\end{equation}
Here the inclusion in the weighted Sobolev space can be obtain similar to the way for $(\ref{eqn: regularity of second and third terms})$. On the other hand, by applying $r^2 \Delta_0$ to $(\ref{eqn: leading order term of remainder})$, we obtain
\begin{equation}\label{eqn: leading order term of remainder 2}
r^2\Delta_0 \circ R_1 f \sim \left[ (\mu^2 + (n-2)\mu) a(y) + \Delta_{\hat{g}} a(y) \right] r^\mu (\ln r)^p.
\end{equation}
Now we have two possible cases:

{\bf Case 1}:
$
(\mu^2 + (n-2)\mu) a(y) + \Delta_{\hat{g}} a(y) = 0.
$
In this case, $\mu = \frac{2-n}{2} + \frac{\sqrt{(n-2)^2 + 4\nu}}{2}$ where $\nu$ is an eigenvalue of $-\Delta_{\hat{g}}$ and $a(y)$ is a corresponding eigenfunction. Here we rule out $\mu = \frac{2 - n}{2} - \frac{\sqrt{(n-2)^2 + 4\nu}}{2}$ by the facts: $R_1 f \in W^{\infty, 2}_{1-\frac{n}{2}}$ and $\nu \geq 0$ and (\ref{eqn: polynomial weighted Sobolev condition}). Note that $\mu \geq 0$, as $\nu \geq 0$.
In addition, similarly as in the proof of Theorem \ref{thm: asymptotic of eigenfunctions}, we can show that $p =0$ in (\ref{eqn: leading order term of remainder}). Therefore, we have
\begin{equation}
R_1 f = a(y)r^\mu + O(r^{\mu^\prime}), \quad \text{as} \ \ r \to 0,
\end{equation}
where $\mu^\prime > \mu \geq 0$. Furthermore, if $\mu = 0$, the associated eigenvalue $\nu=0$, and so the corresponding eigenfunction $a(y)$ is constant function on the cross section. As a result, the leading order term $a(y) r^\mu$ is a constant. By Remark \ref{remark: non-uniqueness of potential function}, we can subtract the constant leading order term from the potential function $f$, and then obtain
\begin{equation}\label{eqn: asymptotic of R1f}
R_1 f = O(r^{\mu^\prime}), \quad \text{as} \ \ r \to 0, \ \ \text{for some} \ \ \mu^{\prime} >0.
\end{equation}
If $\mu > 0$, the we already have $(\ref{eqn: asymptotic of R1f})$. Then by (\ref{eqn: decomposition of potential function}), (\ref{eqn: regularity of second and third terms}) and (\ref{eqn: asymptotic of R1f}), we obtain $ f \in W^{2, 2}_{1-\frac{n}{2}+\alpha}$.  Then we can repeat the above argument by starting with $f \in W^{2, 2}_{1-\frac{n}{2}+\alpha}$ to further increase the index in the weighted Sobolev space. By iterating the argument finitely many times, we can obtain $f \in W^{2, 2}_{\delta}$ for some $\delta >0$.


{\bf Case 2}: 
$
(\mu^2 + (n-2)\mu) a(y) + \Delta_{\hat{g}} a(y) \neq 0.
$
In this case, by comparing leading orders in $(\ref{eqn: leading order term of remainder})$ and $(\ref{eqn: leading order term of remainder 2})$, and using $(\ref{eqn: polynomial weighted Sobolev condition})$, the inclusion in (\ref{eqn: laplace acts on remainder}) implies: $R_1 f \in W^{\infty, 2}_{1-\frac{n}{2} + \alpha}$. By combining with (\ref{eqn: decomposition of potential function}) and (\ref{eqn: regularity of second and third terms}), this then implies $f \in W^{2, 2}_{1-\frac{n}{2}+\alpha}$. Again, by iterating the argument finitely many times, we can obtain $f \in W^{2, 2}_{\delta}$ for some $\delta >0$.



In both cases, we can obtain $f \in W^{2, 2}_{\delta_0}$ for some $\delta_0 >0$. By weighted elliptic estimate in Proposition $\ref{prop: weighted elliptic estimate}$, we then have $f \in W^{k, 2}_{\delta_0}$ for any $k\in \mathbb{N}$. Finally, by the asymptotic control for functions in weighted Sobolev space on manifold with isolated conical singularities (see Lemma 8.1 in \cite{DW18}), we obtain $|\nabla^i f| = o(r^{\delta_0 - i})$ as $r \rightarrow 0$, for all $i \in \mathbb{N}$.
\end{proof}

We now use the asymptotic estimate obtained in Proposition \ref{prop: asymptotic for potential function of steady soliton} to study steady and expanding gradient Ricci solitons with isolated conical singularities. First, we recall a basic fact about the potential function of gradient Ricci soliton that we need in the proofs of Theorems \ref{thm: steady soliton is ricci flat} and \ref{thm: expander is Einstein}.
\begin{lem}\label{lem: basic fact of potential function}
Let $(M^n, g, f)$ be a $n$-dimensional (either smooth or singular) gradient Ricci soliton, that is, on the regular part of $M$,
\begin{equation}
Ric + \nabla^2 f = \Lambda g,
\end{equation}
holds, for some constant $\Lambda$. Then the potential function $f$ satisfies the differential equation:
\begin{equation}\label{eqn: basic fact of potential function}
|\nabla f|^2 - \Delta f + n\Lambda -2 \Lambda f - C = 0,
\end{equation}
on the regular part of $M$, for some constant $C$.
\end{lem}
The conclusion of Lemma \ref{lem: basic fact of potential function} is a pointwise differential equation $(\ref{eqn: basic fact of potential function})$ on the regular part of the manifold. Its proof for the singular manifolds is the same as for smooth manifolds, which can be found in literature, see e.g. the proof of Proposition 1.1.1 in \cite{CZ06}.

\begin{thm}\label{thm: steady soliton is ricci flat}
Let $(M^{n}, g, f)$ be a compact gradient steady Ricci soliton with isolated conical singularities. If cross section of model cone at each singularity has scalar curvature  $R_{\hat{g}}=(n-1)(n-2)$, then $(M^{n}, g)$ is a Ricci flat manifold with isolated conical singularities. 
\end{thm}
\begin{proof}
Recall steady gradient Ricci soliton equation:
\begin{equation*}
Ric + \nabla^{2}f=0.
\end{equation*}
Let $\Lambda = 0 $ in Lemma \ref{lem: basic fact of potential function}, then the equation $(\ref{eqn: basic fact of potential function})$ becomes:
\begin{equation}\label{eqn: potential function equation}
\Delta f - |\nabla f|^{2} = C,
\end{equation}
on the regular part $\mathring{M}$, for some constant $C$. 

Then by the asymptotic estimate for the potential function $f$ obtained in Proposition \ref{prop: asymptotic for potential function of steady soliton}, i.e.,
\begin{equation}
|\nabla f| = o(r^{\delta-1}), \ \ \text{as} \ \ r\rightarrow 0, \ \ \text{for some} \ \ \delta>0,
\end{equation}
one obtains 
\begin{equation}
\int_{M} \Delta(e^{f}) d\vol_{g} =0.
\end{equation}
Thus,
\begin{equation}
C\int_{M} e^{f}d\vol_{g} = \int_{M} (\Delta f - |\nabla f|^{2})e^{f} d\vol_{g} = \int_{M} \Delta (e^{f})d\vol_{g}=0.
\end{equation}
Consequently, $C=0$, and again by the asymptotic estimate of $f$ one also has $\int_{M} \Delta f=0$. 
So by integrating the equation (\ref{eqn: potential function equation}), one obtains that $f$ has to be constant function, and so the metric $g$ is Ricci flat.
\end{proof}

Similarly, for expanding solitons, we have: 
\begin{thm}\label{thm: expander is Einstein}
Let $(M^{n}, g, f)$ be a compact expanding gradient Ricci soliton with isolated conical singularities. If 
$R_{\hat{g}}=(n-1)(n-2)$, then $(M^{n}, g)$ is an Einstein manifold with isolated conical singularities. 
\end{thm}

\begin{proof}
Recall the expanding gradient Ricci soliton equation: 
\begin{equation}
Ric + \nabla^2 f = \Lambda g, \ \ \text{for a  constant} \ \ \Lambda < 0.
\end{equation}
By Lemma $\ref{lem: basic fact of potential function}$, the potential function $f$ satisfies:
\begin{equation}\label{eqn: equation1 for potential function}
|\nabla f|^2 - \Delta f + n\Lambda - 2\Lambda f - C = 0,
\end{equation}
on the regular part $\mathring{M}$, for some constant $C$.

Now the same as in the proof of Theorem \ref{thm: steady soliton is ricci flat}, by the asymptotic estimate in Proposition \ref{prop: asymptotic for potential function of steady soliton}, we obtain
\begin{equation}\label{eqn: integral of laplacian of exponent of potential function}
\int_{M} (\Delta f - |\nabla f|^{2})e^{f} d\vol_{g} = \int_{M} \Delta (e^{f})d\vol_{g}=0.
\end{equation}
Therefore, by (\ref{eqn: equation1 for potential function}) and (\ref{eqn: integral of laplacian of exponent of potential function}), we obtain
\begin{equation}\label{eqn: integral equation for potential function}
\int_{M} (n\Lambda - 2\Lambda f - C)e^{f} =0.
\end{equation}

Recall that in Proposition \ref{prop: asymptotic for potential function of steady soliton} we also obtain the boundedness of the potential function and $f\rightarrow 0$ as approaching the singular points. Using this, we will prove in a moment that
\begin{equation}\label{eqn: equation2 for potential function}
n\Lambda - 2 \Lambda f - C \equiv 0.
\end{equation}
Granted, we obtain from $(\ref{eqn: equation2 for potential function})$, (\ref{eqn: equation1 for potential function}),
\begin{equation}\label{eqn: equation3 for potential function}
|\nabla f|^2 - \Delta f =0.
\end{equation}
By integrating $(\ref{eqn: equation3 for potential function})$, we find that $f$ must be constant, since $\int_{M} \Delta f =0$ (this follows from classical Stokes theorem and asymptotic estimate for $f$ obtained in Proposition \ref{prop: asymptotic for potential function of steady soliton}). 

Now it remains to prove (\ref{eqn: equation2 for potential function}). We separate the discussion into the following three cases.

{\bf Case 1}. Assume $f\geq 0$ on $M$. Because $f$ is bounded on $M$ and $f\rightarrow 0$ at the singular points, there exists a global maximal point $x_{\max}\in \mathring{M}$ of $f$ such that
\begin{equation*}
\nabla f (x_{\max}) =0, \quad \Delta f(x_{\max}) \leq 0,
\end{equation*}
and 
\begin{equation*}
f(x) \leq f(x_{\max}), \quad \forall x \in M.
\end{equation*}
Thus by evaluating the equation (\ref{eqn: equation1 for potential function}) at $x_{\max}$, we have
\begin{equation*}
n\Lambda - 2\Lambda f(x_{\max}) - C \leq 0.
\end{equation*}
On the other hand, $\Lambda <0$. Hence 
\begin{equation}\label{eqn: expanding soliton case1}
n\Lambda - 2 \Lambda f(x) - C  \leq 0, \quad \forall x \in M.
\end{equation}
Combining $(\ref{eqn: expanding soliton case1})$ with the integral equation (\ref{eqn: integral equation for potential function}) gives (\ref{eqn: equation2 for potential function}).

{\bf Case 2}. Assume $f \leq 0$ on $M$. Similarly, there exists a global minimal point $x_{\min} \in \mathring{M}$ of $f$ such that
\begin{equation*}
\nabla f (x_{\min}) =0, \quad \Delta f(x_{\min}) \geq 0,
\end{equation*}
and 
\begin{equation*}
f(x) \geq f(x_{\min}), \quad \forall x \in M.
\end{equation*}
Thus by evaluating the equation (\ref{eqn: equation1 for potential function}) at $x_{\min}$, we have
\begin{equation*}
n\Lambda - 2\Lambda f(x_{\min}) - C \geq 0 .
\end{equation*}
But $\Lambda <0$. Hence in fact
\begin{equation}\label{eqn: expanding soliton case2}
n\Lambda - 2 \Lambda f(x) - C  \geq 0, \quad \forall x \in M.
\end{equation}
Combining the inequality $(\ref{eqn: expanding soliton case2})$ with the integral equation (\ref{eqn: integral equation for potential function}) gives the equation (\ref{eqn: equation2 for potential function}).

{\bf Case 3}. Assume $f$ dose not have definite sign on $M$. Again because $f$ is bounded on $M$ and $f\rightarrow 0$ at the singular points, there exists a global maximal point $x_{\max}\in \mathring{M}$ and a global minimal point $x_{\min}\in \mathring{M}$ of $f$ such that
\begin{equation*}
\nabla f (x_{\max}) = \nabla f(x_{\min})=0, \quad \Delta f(x_{\max}) \leq 0 \leq \Delta f(x_{\min}),
\end{equation*}
and 
\begin{equation*}
f(x_{\min}) \leq f(x) \leq f(x_{\max}), \quad \forall x \in M.
\end{equation*}
Thus by evaluating the equation (\ref{eqn: equation1 for potential function}) at $x_{\max}$ and $x_{\min}$, we have
\begin{equation*}
n\Lambda - 2\Lambda f(x_{\max}) - C \leq 0 \leq n \Lambda - 2 \Lambda f(x_{\min}) - C.
\end{equation*}
On the other hand, $\Lambda <0$, and hence
\begin{equation*}
n\Lambda - 2\Lambda f(x_{\max}) - C \geq n \Lambda - 2 \Lambda f(x_{\min}) - C.
\end{equation*}
Then we must have (\ref{eqn: equation2 for potential function}).
\end{proof}





\end{document}